\documentclass[11pt]{article}
\usepackage{xcolor}
\usepackage{bm}
\usepackage{geometry}
\usepackage{amsmath}
\usepackage{amsfonts}
\usepackage{amssymb}
\usepackage{amsthm}
\usepackage{mathtools}
\DeclarePairedDelimiterX{\inp}[2]{\langle}{\rangle}{#1, #2}
\usepackage{mathrsfs}
\usepackage{authblk}
\usepackage{hyperref}
\usepackage[capitalize]{cleveref}
\usepackage{microtype}
\hypersetup{hidelinks}

\newtheorem{notation}{Notation}
\newtheorem{defn}{Definition}
\newtheorem{concl}{Conclusion}

\newtheorem{theorem}{Theorem}
\newtheorem{proposition}{Proposition}
\newtheorem{lemma}{Lemma}[section]

\newtheorem{remark}{Remark}

\renewcommand{\d}{\mathrm{d}}
\newcommand{\dt}{\partial_t}
\newcommand{\dx}{\partial_x}
\newcommand{\dy}{\partial_y}
\newcommand{\dX}{\partial_X}
\newcommand{\dY}{\partial_Y}
\newcommand{\Q}{\mathcal{Q}}
\renewcommand{\u}{\mathbf{u}}

\newcommand{\R}{\mathbb{R}}
\newcommand{\co}{\operatorname{co}}
\newcommand{\sgn}{\operatorname{sgn}}

\numberwithin{equation}{section}

\renewcommand{\thedefn}{\thesection.\arabic{defn}}

\renewcommand{\thelemma}{\thesection.\arabic{lemma}}

\DeclareMathOperator{\supp}{supp}

\begin{document}
    \title{Small Viscosity Limit of Navier-Stokes Equations in A Quarter Plane with Navier Slip Conditions}

\author{
    Wen GUO\thanks{
        Corresponding author. 
        School of Mathematical Sciences, Shanghai Jiao Tong University, Shanghai, China. Email: \href{mailto:guo-wen@sjtu.edu.cn}{guo-wen@sjtu.edu.cn}} 
     \qquad Ya-Guang WANG\thanks{
        School of Mathematical Sciences, MOE-LSC and SHL-MAC, Shanghai Jiao Tong University, Shanghai, China. Email:
        \href{mailto:ygwang@sjtu.edu.cn}{ygwang@sjtu.edu.cn}}
        }

    \date{\today}
    \maketitle
    \selectfont
    \tableofcontents

    \begin{abstract}
        We study the vanishing-viscosity limit of Leray solutions to the two-dimensional incompressible Navier--Stokes equations in a quarter plane with Navier slip boundary conditions on both edges for a class of compatible well-prepared initial data. By multi-scale analysis, we deduce that a boundary layer of order $O(\sqrt{\epsilon})$ arises near each edge of the corner and is described by a linearized Prandtl equation, and near the corner, their overlap and nonlinear interaction are captured by a two-dimensional corner profile. Although this profile occurs at order $O(\sqrt{\epsilon})$, it is governed, in the fast variables, by a nonlinear initial-boundary value problem of Navier--Stokes type. The well-posedness and Sobolev regularity of the outer flow problems, together with the well-posedness and weighted estimates of the boundary-layer and corner-layer problems are obtained on a common time interval $[0,T]$ independent of $\epsilon$. To ensure the required regularity of the outer Euler flow in the critical right-angle domain, it is assumed that the initial vorticity of the outer leading order flow vanishes in a neighborhood of the corner. We then rigorously justify the resulting composite expansion in the small viscosity limit. More precisely, after retaining the leading outer flow, the two weak boundary layers, the nonlinear corner layer, and the order-$\epsilon$ outer correction, the error between the exact solution and this truncated expansion is $O(\epsilon^{5/4})$ in $L^\infty(0,T;L^2(\Q))$ and $O(\epsilon^{3/4})$ in $L^2(0,T;H^1(\Q))$.

    \end{abstract}

~

\section{Introduction}\label{sec-intro}
Consider the small viscosity limit for the following incompressible Navier-Stokes equations in $\R_+\times \Q$, where $\Q = \{x,y>0\}$ is a quarter plane:
\begin{equation}\label{INSE}
\begin{cases}
\partial_t \u^\epsilon + \u^\epsilon \cdot \nabla \u^\epsilon + \nabla p^\epsilon = \epsilon \Delta \u^\epsilon, \quad \text{in } \Q,\\
\nabla \cdot \u^\epsilon = 0,\quad \text{in } \Q,
\end{cases}
\end{equation}
with the Navier boundary conditions
\begin{equation}\label{N-BC}
    \u^\epsilon \cdot \vec{n} = 0, \quad \left(D (\u^\epsilon) \vec{n} + \alpha \u^\epsilon \right)_{\vec{\tau}} = 0, \quad \text{on } \partial \Q,
\end{equation}
and the initial data
\begin{equation}\label{icd}
    \u^\epsilon(0,\cdot) = \u^\epsilon_0 \quad \text{in } \Q,
\end{equation}
where $\u^\epsilon = (u^\epsilon,v^\epsilon) $, $p^\epsilon$ are the velocity and the pressure, $\epsilon$ is the viscosity and $D(\u^\epsilon) = \frac{1}{2}\left(\nabla \u^\epsilon + (\nabla \u^\epsilon)^T \right)$ is the rate of the strain tensor,  $\vec{n}$ and $\vec{\tau}$ are unit normal and tangent vectors on the boundary $\partial \Q$, $\alpha\in\mathbb R$ is a friction coefficient independent of $\epsilon$, and $\u^\epsilon_0$ is a given vector field in $L^2(\Q)$, whose precise asymptotic structure and regularity are specified below, satisfying
\begin{equation}\label{cc-0}
    \nabla\cdot \u^\epsilon_0 = ({\u^\epsilon_0} \cdot \vec{n})_{|\partial\Q} = 0
\end{equation}

The existence and uniqueness of the Leray–Hopf solution to the above problem follows from standard arguments, 
one can refer to \cite{kelliher2006navier, iftimie2011viscous} and references therein. 
In particular, for any given $L^2(\Q)$ divergence-free vector field $\u^\epsilon_0$ tangent to the boundary $\partial\Q$, there exists a unique global Leray-Hopf solution in the energy class,
\begin{equation*}
    \u^\epsilon \in C^0_w([0, +\infty);L^2(\Q)) \cap L^2_{\text{loc}}((0,+\infty);H^1(\Q))
\end{equation*}
of \eqref{INSE}-\eqref{icd} satisfying the energy inequality
\begin{equation}\label{NS-ineq}
\begin{aligned}
\|\u^\epsilon (t)\|_{L^2(\Q)}^2
+ \epsilon \int_0^t \|\nabla\u^\epsilon(s)\|_{L^2(\Q)}^2 \le \exp\!\left(C\epsilon t\right)\|\u^\epsilon_0\|_{L^2(\Q)}^2, \quad \forall t>0.
\end{aligned}
\end{equation}
The latter estimate follows by controlling the boundary term with the trace inequality; here $C$ is independent of $\epsilon$ and may depend on $|\alpha|$.

The vanishing viscosity limit for the incompressible Navier-Stokes equations with Navier boundary conditions has been extensively studied when the boundary is smooth. In two dimensions, Clopeau, Mikeli\'c and Robert \cite{clopeau1998vanishing} proved convergence to the Euler equations by using uniform vorticity estimates. This result was extended to initial vorticities in $L^p$, $p>2$, by Lopes Filho, Nussenzveig Lopes and Planas \cite{lopesfilho2005inviscid}; see also Kelliher \cite{kelliher2006navier}. Iftimie and Planas \cite{iftimie2006inviscid} proved the strong $L^2$ convergence of Leray solutions to a sufficiently regular Euler flow by a direct energy argument. For strong solutions, Masmoudi and Rousset \cite{masmoudi2012uniform} established estimates uniform with respect to the viscosity in conormal Sobolev spaces and obtained the inviscid limit by a compactness argument, and it was extended to the inviscid limit problem of the incompressible MHD equations in \cite{MengWang}. On the other hand, Iftimie and Sueur \cite{iftimie2011viscous} justified an asymptotic expansion involving a weak boundary layer of amplitude and thickness $O(\sqrt{\epsilon})$, and Gie and Kelliher \cite{gie2012boundary} constructed explicit boundary correctors and obtained quantitative convergence estimates.
The asymptotic expansions were derived in \cite{WangWangXin} for the incompressible Navier-Stokes equations with Navier boundary conditions of the slip length depending on viscosity.
The small viscosity limit for compressible Navier-Stokes equations with Navier boundary conditions was studied in \cite{WangWilliams}.
The above works were focused on problems defined in half-spaces or domains with smooth boundaries, in which there is not any interaction of boundary layers.

The goal of this paper is to understand the interaction of boundary layers near non-smooth points of boundaries. As a starting point, we study the asymptotic expansion of the small viscosity limit for the incompressible Navier-Stokes equations in a quarter plane with Navier boundary conditions. 
By multi-scale analysis, we shall see that there is a boundary layer at order $O(\sqrt{\epsilon})$ along each boundary component away from the corner, and the boundary layer profiles satisfy the problems of linearized Prandtl equations. Near the corner, however, there exists a corner layer, which is produced by the interaction of two boundary layers along the edges of the corner; although it appears at order $O(\sqrt{\epsilon})$, the corner-layer profile satisfies a nonlinear problem.
A formal asymptotic expansion involving a corner layer was previously derived by Rubin \cite{rubin1966incompressible} for a steady three-dimensional no-slip flow along a right-angle corner. 
We shall construct the corner layer profiles and rigorously justify the resulting asymptotic expansion by deriving a quantitative estimate for the remainder.
To the best of our knowledge, this is the first observation of a nonlinear corner layer capturing the interaction of two weak Navier boundary layers, and the first rigorous result on the inviscid limit of the incompressible Navier--Stokes equations with Navier boundary conditions in a corner domain.


More precisely, we look for the solution $\mathbf{u}^\epsilon$ to \eqref{INSE}-\eqref{icd} in the form
\begin{equation}\label{ex-1}
    \u^\epsilon(t,x,y) = \u^0(t,x,y) + \sqrt{\epsilon} \mathbf{U}\left(t,x,y,\frac{x}{\sqrt{\epsilon}},\frac{y}{\sqrt{\epsilon}} \right) + O({\epsilon}),
\end{equation}
where $\u^0$ is the solution to the following problem of the Euler equation:
\begin{equation}\label{limit}
    \begin{cases}
        \partial_t \u^0 + \u^0 \cdot \nabla \u^0 + \nabla p^0 = 0, \quad \text{in } \Q,\\ 
        \nabla \cdot \u^0 = 0,\quad \text{in } \Q,\\
        \u^0 \cdot \vec{n} = 0,\quad \text{on } \partial\Q\setminus\{{\bf 0}\},
    \end{cases}
\end{equation}
 and $\mathbf{U}$ contains the boundary layers near $\{x=0\}\cup\{y=0\}$ and corner layer near $\{x=y=0\}$ decaying rapidly in its last two variables. 
 
 The global existence of the solution to the problem \eqref{limit} can be found in \cite{MR3054600, MR3981393}. The justification of the asymptotic expansion \eqref{ex-1}, however, requires additional regularity of the Euler profile. In particular, in order to obtain an $O(\epsilon)$ remainder in $L^\infty(0,T; L^2(\Q))$, we need
$$
(\mathbf u^0,p^0)\in L^2(0,T;H^2(\mathcal Q))
\times L^2(0,T;H^1(\mathcal Q)).
$$
Indeed, for a divergence-free vector field tangent to the boundary, the corresponding corner Hodge estimates imply that, if without additional compatibility conditions, the derivation 
$$
\omega^0\in H^m(\Q)\quad\Longrightarrow\quad \u^0\in H^{m+1}(\Q)
$$
is guaranteed under the restriction that the measure of the angle $\theta<\frac{\pi}{m+1}$, see \cite{grisvard2011elliptic,dauge1988elliptic,zajaczkowski1984solvability}. Thus, for $\theta=\pi/2$, the case $m=1$ is critical and an $r^2\log r$-type term may prevent $H^2$-regularity. To exclude this corner singularity we shall assume that the initial vorticity vanishes near the corner, this property is preserved by the Euler evolution. The corresponding well-posedness and regularity result for \eqref{limit} is stated as follows.
\begin{theorem}\label{thm-Euler}
    Let $m\ge 4$ be an integer. Suppose that $\u_0 \in H^m(\Q)$ satisfies
    \begin{equation*}
        \nabla \cdot \u_0 = 0, \quad \u_0 \cdot \vec{n}_{|\partial\Q} = 0, \qquad \nabla \times \u_0 = 0 \text{ in a small neighborhood of }\mathbf{0}.
    \end{equation*}
    Then there exists $T>0$ such that there exists a unique solution $$(\u^0, p^0) \in L^\infty(0,T; H^{m}(\Q)) \times L^\infty(0,T; H^1(\Q) / \R)$$ to \eqref{limit} with initial data $\u_0$.
\end{theorem}

Now we assume that the initial function $\u^\epsilon_0$ in \eqref{icd} takes the following form:
\begin{equation}\label{u_0}
    \begin{aligned}
        \u^\epsilon_0(x,y)  \sim &
        \u_0 (x,y) +
    \sqrt{\epsilon}
    \u^{E,1}_0 (x,y) 
          + \sqrt{\epsilon} \left(u^{B,1}_0 (x,\frac{y}{\sqrt{\epsilon}}), v^{B,1}_0 (\frac{x}{\sqrt{\epsilon}},y)\right)^T \\
        & + \sqrt{\epsilon} \u^{C}_0(\frac{x}{\sqrt{\epsilon}}, \frac{y}{\sqrt{\epsilon}})
    + \epsilon \u^{E,2}_0 (x,y) + \epsilon \left(u^{B,2}_0(\frac{x}{\sqrt{\epsilon}},y), v^{B,2}_0(x,\frac{y}{\sqrt{\epsilon}})\right)^T +
    \ldots,
    \end{aligned}
\end{equation}
where
\begin{equation}\label{decay-asmp}
    \begin{aligned}
        & u^{B,1}_0(x,Y), \u_0^C(\frac{x}{\sqrt{\epsilon}}, Y), v^{B,2}_0(x,Y) \text{ decay rapidly to } 0 \text{ as } Y = \frac{y}{\sqrt{\epsilon}} \to \infty, \text{ and} \\
        & v^{B,1}_0(X,y), \u_0^C(X, \frac{y}{\sqrt{\epsilon}}), u^{B,2}_0(X,y) \text{ decay rapidly to } 0 \text{ as } X = \frac{x}{\sqrt{\epsilon}} \to \infty.
    \end{aligned}
\end{equation}

At the level of the general formal hierarchy, the initial expansion may contain the regular outer corrections $\sqrt{\epsilon}\,\u^{E,1}_0$ and $\epsilon\u^{E,2}_0$, together with further higher-order correctors represented by the dots in \eqref{u_0}. Since the purpose of this paper is to focus on the interaction of the two weak boundary layers and the associated corner-layer mechanism, for simplicity of presentation, we shall restrict ourselves to the well-prepared subclass
$\mathbf u^{E,1}_0=\mathbf u^{E,2}_0=0 $
and take all further initial correctors represented by the dots to vanish. Consequently, \(\mathbf u^{E,1}(t,x,y)\equiv0\), since it satisfies the homogeneous linearized Euler system displayed in Section \ref{sec-faa}. The profile \(\mathbf u^{E,2}\), however, is retained in the expansion, since its equation is forced and it does not vanish for positive times in general. Allowing a nonzero compatible initial value \(\u^{E,1}_0, \u^{E,2}_0\) would require no essential change in the argument, but we do not pursue this additional generality here. Accordingly, throughout the rest of the paper, we take the initial data to be exactly of the following form and require \eqref{cc-0} to hold for every sufficiently small $\epsilon$:
\begin{equation}\label{u_0-smp}
\begin{aligned}
        \u^\epsilon_0(x,y) = & \u_0(x,y)
+\sqrt{\epsilon}
\left(
u^{B,1}_0(x,\frac{y}{\sqrt{\epsilon}}),
v^{B,1}_0(\frac{x}{\sqrt{\epsilon}},y)
\right)^T \\
& + \sqrt{\epsilon}
\u^C_0\left(\frac{x}{\sqrt{\epsilon}},
\frac{y}{\sqrt{\epsilon}}\right)
+\epsilon
\left(
u^{B,2}_0(\frac{x}{\sqrt{\epsilon}},y),
v^{B,2}_0(x,\frac{y}{\sqrt{\epsilon}})
\right)^T,
\end{aligned}
\end{equation}
where the initial boundary layer profiles
$u_0^{B,1}, v_0^{B,1}, \u^C_0$ etc. decay fast as $Y=\frac{y}{\sqrt{\epsilon}}$, $X=\frac{x}{\sqrt{\epsilon}}$ go to infinity. 

For $x$ bounded away from zero, in the $y$-boundary-layer regime $Y=y/\sqrt{\epsilon}=O(1)$, the profiles depending on $X=x/\sqrt{\epsilon}$ decay as $\epsilon\to0$, and \eqref{u_0-smp} reduces to the usual half-plane expansion along $\{y=0\}$:
$$\u_0(x,y) +
\sqrt{\epsilon} \left(u^{B,1}_0(x,\frac{y}{\sqrt{\epsilon}}), 0\right) + \epsilon \left(0, v^{B,2}_0(x,\frac{y}{\sqrt{\epsilon}})\right),$$
which is the form of the leading expansions of solutions to the Navier-Stokes equation in the half-plane $\{y>0\}$, as studied in \cite{iftimie2011viscous}; by symmetry, the same holds for $\{x>0\}$. 

The main results of this paper are the following:
\begin{theorem}\label{thm-1}
    Let $\u^\epsilon_0$ be given by \eqref{u_0-smp}, satisfying the compatibility conditions \eqref{cc-0} and the decay assumption \eqref{decay-asmp},
    with
    $$
    \begin{aligned}
    & \u_0 \in H^4(\Q), \, \nabla \times \u_0 = 0 \text{ in a small neighborhood of }\mathbf{0}, \\
    & \langle Y \rangle^k u^{B,1}_0 \in H^2_xH^2_Y(\R_+)^2,\, \langle X \rangle^k v^{B,1}_0 \in H^2_y H^2_X(\R_+)^2 \quad \text{for any }k\in \mathbb{N},\\
    & \langle(X,Y)\rangle^k \u^C_0 \in H^1_{X,Y}(\R_+)^2 \text{ for any }k \in \mathbb{N},
    \end{aligned}
    $$
    where the notation $\langle Y \rangle=(1+Y^2)^{\frac 12}$. 
    Let $\u^\epsilon$ be the weak Leray solution of \eqref{INSE} on $\Q$ with \eqref{N-BC} and \eqref{icd}. 
    Then there exists $T>0$ such that there exist unique solutions $\u^{E,0} \in L^\infty(0,T; H^4(\Q))$ to \eqref{limit} with initial $\u_0$, boundary layer profiles $u^{B,1,2}, v^{B,1,1}$ satisfying
    \begin{equation*}
        \begin{aligned}
            u^{B,1,2}\in L^\infty(0,T; H^2_xL^2_Y(\R_+)^2) \cap L^2(0,T;H^2_xH^1_Y(\R_+)^2), \\
            v^{B,1,1}\in L^\infty(0,T; H^2_yL^2_X(\R_+)^2) \cap L^2(0,T;H^2_yH^1_X(\R_+)^2),
        \end{aligned}
    \end{equation*}
to the problems \eqref{B-1-2} and \eqref{B-1-1}
of linearized Prandtl equations with the initial data $u^{B,1,2}|_{t=0}=u_0^{B,1}(x,Y), v^{B,1,1}|_{t=0}=v_0^{B,1}(X,y)$ respectively,     a corner layer profile $$\u^{C,1} \in L^\infty(0,T;H^1_{X,Y}(\R_+)^2)$$
to the nonlinear problem \eqref{u-C,1} with with the initial data $\u^{C,1}|_{t=0}=\u_0^{C}(X,Y)$, and an out flow $\u^{E,2}\in L^\infty(0,T;H^1(\Q))$, such that the following expansion:
    \begin{equation}\label{expansion}
    \begin{aligned}
        \u^\epsilon\left(t,x,y\right)  = & \u^{E,0}\left(t,x,y\right) + \sqrt{\epsilon} \left(u^{B,1,2}(t,x,\frac{y}{\sqrt{\epsilon}}), v^{B,1,1}(t,\frac{x}{\sqrt{\epsilon}},y)\right)^T\\
        & + \sqrt{\epsilon} \u^{C,1}\left(t,\frac{x}{\sqrt{\epsilon}},\frac{y}{\sqrt{\epsilon}}\right) + \epsilon \u^{E,2}\left(t,x,y\right)+{\bf r}^{\epsilon}(t,x,y)
    \end{aligned}
    \end{equation}
    holds as $\epsilon \to 0$, with the error ${\bf r}^\epsilon$ being $O(\epsilon^{\frac{5}{4}})$ in $L^\infty(0,T;L^2(\Q))$ and $O(\epsilon^{\frac{3}{4}})$ in $L^2(0,T;H^1(\Q))$.
\end{theorem}

\begin{remark}\label{rmk-1}
Equating the coefficients of the resulting powers of $\sqrt{\epsilon}$ in \eqref{cc-0}, and using \eqref{decay-asmp}, yield
\begin{enumerate}
    \item[(1)] $\nabla \cdot \u_0 = \nabla \cdot \u^C_0 = 0$;
    \item[(2)] $u^{B,2}_0(X,y) = \int_X^\infty \dy v^{B,1}_0(X',y) \, \d X', \quad v^{B,2}_0(x,Y) = \int_Y^\infty \dx u^{B,1}_0(x,Y') \, \d Y'$;
    \item[(3)] ${u_0}_{|x=0} = {v_0}_{|y=0} = 0$, ${u^{B,2}_0}_{|X=0} = {v_0^{B,2}}_{|Y=0} = 0$ 
    ; 
    \item[(4)] ${u^{B,1}_0}_{|x=0} + {u^{C}_0}_{|X=0} = 0$, ${v^{B,1}_0}_{|y=0} + {v^{C}_0}_{|Y=0} = 0$.
\end{enumerate}
Moreover, (1)-(4) necessarily imply that
\begin{equation*}
    \int_0^\infty {u^{B,1}_0}(0,Y) \,\d Y + \int_0^\infty {v^{B,1}_0}(X,0) \,\d X = 0,
\end{equation*}
and
\begin{equation*}
    \dx \left(\int_0^\infty  u^{B,1}_0(x, Y) \,\d Y \right) = \dy  \left(\int_0^\infty  v^{B,1}_0(X,y) \,\d X \right) = 0.
\end{equation*}
Such nontrivial initial datum $u^{B,1}_0, v^{B,1}_0, \u^C_0$ satisfying the compatibility conditions do exist. For example, one can take $u^{B,1}_0 = x (1-2Y^2)e^{-(x+Y^2)}$, $v^{B,1}_0 = y (1-2X^2)e^{-(y+X^2)}$, and $\u^C_0 = \nabla^\perp \left(XY e^{-(X^2+Y^2)}\right)$.
\end{remark}
\begin{remark}
    The assumption that the initial vorticity vanishes near the corner is only for obtaining the required regularity of the outer Euler profile in the critical right-angle domain. Once a sufficiently regular Euler flow,  e.g. $\u^{E,0}\in L^\infty(0, T; W^{2,\infty}(\Q)\cap H^3(\Q))$), is known, all remaining calculation in this paper works without this assumption.
\end{remark}
\begin{remark}\label{rmk-2}
    For zero initial corner datum, the dynamical generation of the leading corner profile is governed by the two traces $\overline{u^{B,1,2}}$ and $\underline{v^{B,1,1}}$. In the right-angle domain considered here, when $\mathbf u_0^\epsilon=\mathbf u_0(x,y)$, the initial boundary-layer profiles vanish, and the trace problem \eqref{u-b,1,2}, together with its symmetric counterpart, has zero initial and boundary data. Hence uniqueness gives $\overline{u^{B,1,2}}=\underline{v^{B,1,1}}=0$, and consequently $\mathbf u^{C,1}\equiv0$. For a general opening angle, in particular when nonzero Euler vorticity reaches the corner, the trace equations are coupled through the corner geometry, and these traces need not vanish even for viscosity-independent initial data; thus a nontrivial leading corner layer may be generated dynamically. This mechanism will be studied in a forthcoming paper.
\end{remark}

Noticing that the vorticity equation of \eqref{INSE} simply reads
\begin{equation*}
    \dt \omega^\epsilon + \u^\epsilon\cdot \nabla \omega^\epsilon - \epsilon \Delta \omega^\epsilon = 0,
\end{equation*}
where $\omega^\epsilon = \nabla \times \u^\epsilon$.
When $\alpha = 0$ in \eqref{N-BC}, the boundary conditions for $\omega^{\epsilon}$ are homogeneous: $\omega^\epsilon_{|x=0} = \omega^\epsilon_{|y=0} = 0$.
The uniform estimate then follows directly from the standard energy method, combined with the two-dimensional Hodge-type estimate (which is obtained by a reflection argument and Fourier transformations):
\begin{equation}\label{Hodge}
    \|\u^\epsilon\|_{H^1(\Q)} \lesssim \|\u^\epsilon\|_{L^2(\Q)} + \|\omega^\epsilon\|_{L^2(\Q)},
\end{equation}
valid for divergence-free vector fields tangent to the boundary of $\Q$. Then, one has the following result.
\begin{theorem}\label{thm-2}
    Let $\u_0 \in H^1(\Q)$ be a divergence-free vector field tangent to the boundary, and  $\u_0^\epsilon=\u_0$ for all $\epsilon\in(0,1]$. If $\alpha = 0$, then there exists a unique solution $\u^\epsilon$ to \eqref{INSE}-\eqref{icd} such that the following estimate holds:
    \begin{equation*}
        \sup_{t\in [0,T]} \|\u^\epsilon(t)\|^2_{H^1(\Q)} + \epsilon \int_0^T \|\nabla\times \u^\epsilon(t)\|^2_{H^1(\Q)} \mathrm{d} t \le C\|\u_0\|^2_{H^1(\Q)}.
    \end{equation*}
    Moreover, for every sequence $\epsilon\to0$, there exists a subsequence (not relabeled) and a weak solution $\u$ of \eqref{limit} such that $\u^\epsilon\rightharpoonup^*\u$ in $L^\infty(0,T;H^1(\Q))$ and $\u^\epsilon\to\u$ strongly in $L^2(0,T;L^2_{\mathrm{loc}}(\Q))$.
\end{theorem}

\begin{remark}
Theorem \ref{thm-2} shows that, when $\alpha=0$, the inviscid limit can be approached through an estimate uniform with respect to $\epsilon$. 
For $\alpha\neq0$, one can still introduce a modified vorticity 
\[
\eta^\epsilon
=
\omega^\epsilon+2\alpha(u^\epsilon-v^\epsilon),
\qquad
\eta^\epsilon_{|x=0}=\eta^\epsilon_{|y=0}=0.
\]
satisfying homogeneous boundary conditions. However, the argument of Masmoudi-Rousset \cite{masmoudi2012uniform} relies on elliptic estimates that recover the regularity of the velocity from that of the vorticity. In the present quarter-plane, the opening angle $\pi/2$ is precisely the critical angle for such estimates, and hence their smooth-boundary argument cannot be applied directly. We therefore use an asymptotic expansion, which also reveals the interaction of the two weak boundary layers and the resulting corner-layer structure.
\end{remark}

The remainder of this paper is organized as follows. In Section 2, by using multi-scale analysis, we derive the asymptotic expansion \eqref{expansion} of the solution $\u^\epsilon$ to the incompressible Navier-Stokes equation \eqref{INSE} with Navier boundary condition \eqref{N-BC} in a quarter plane in the small viscosity limit. 
In Sections 3 and 4, we study the well-posedness of the Euler flow problem in the quarter plane, the boundary-layer problems on both sides of the quarter, and the corner layer nonlinear problem of the Navier-Stokes type equations in the quarter plane in fast variables, respectively. In particular, we use Kozlov's weighted elliptic estimate to show the asymptotic behavior of $\u^{C,1}$ as $X, Y \to +\infty$.
The asymptotic expansion of the small viscosity limit is rigorously justified by studying the remainder problem in Section 5. Finally, in the Appendix, we give some elementary estimates of boundary value problems of the Laplace equation in a quarter, which is used in previous sections.

\section{Formal asymptotic analysis}\label{sec-faa}

First in this section, we study the asymptotic behavior of solutions to the problem \eqref{INSE}-\eqref{icd} in the small viscosity limit via multi-scale analysis. Inspired by the boundary layer theory introduced by Prandtl in \cite{prandtl}, take the following ansatz for the solution of \eqref{INSE}-\eqref{icd}:
\begin{equation}\label{ansatz}
    \begin{aligned}
        (\u^\epsilon,p^\epsilon)(t,x,y) =
        & \sum_{k\ge 0} \epsilon^{\frac{k}{2}} \bigg((\u^{E,k},p^{E,k})(t,x,y) + (\u^{B,k,1},p^{B,k,1})(t,\frac{x}{\sqrt{\epsilon}},y) \\
        &
        + (\u^{B,k,2},p^{B,k,2})(t,x,\frac{y}{\sqrt{\epsilon}})  + (\u^{C,k},p^{C,k})(t,\frac{x}{\sqrt{\epsilon}},\frac{y}{\sqrt{\epsilon}}) \bigg).
    \end{aligned}
\end{equation}
All boundary-layer and corner-layer profiles decay rapidly as $X = \frac{x}{\sqrt{\epsilon}} \to +\infty$ and $Y = \frac{y}{\sqrt{\epsilon}} \to +\infty$, respectively.


\vspace{.1in}
\noindent\textbf{\underline{(1) Div-free matching}} 

Plugging \eqref{ansatz} into the divergence-free condition given in \eqref{INSE}, it follows that the vanishing of the order $O(\epsilon^{-\frac 12})$ implies
\begin{equation}\label{div-B,0}
    \partial_X u^{B,0,1}(t,\frac{x}{\sqrt{\epsilon}},y) + \partial_Y v^{B,0,2}(t,x,\frac{y}{\sqrt{\epsilon}})
    + \partial_X u^{C,0}(t,\frac{x}{\sqrt{\epsilon}},\frac{y}{\sqrt{\epsilon}}) + \partial_Y v^{C, 0}(t,\frac{x}{\sqrt{\epsilon}},\frac{y}{\sqrt{\epsilon}})
    =0,
\end{equation}
and for each $k\ge 0$,
the vanishing of the order $O(\epsilon^{\frac k2})$ implies
\begin{equation}\label{div-B,k}
\begin{array}{ll}
&\partial_xu^{E,k}(t,x,y)+\partial_X u^{B,k+1,1}(t,\frac{x}{\sqrt{\epsilon}},y)+
\partial_xu^{B,k,2}(t,x,\frac{y}{\sqrt{\epsilon}})+\partial_X u^{C,k+1}(t,\frac{x}{\sqrt{\epsilon}},\frac{y}{\sqrt{\epsilon}})\\[2mm]
&
+\partial_yv^{E,k}(t,x,y)+\partial_y v^{B,k,1}(t,\frac{x}{\sqrt{\epsilon}},y)+
\partial_Yv^{B,k+1,2}(t,x,\frac{y}{\sqrt{\epsilon}})+\partial_Y v^{C,k+1}(t,\frac{x}{\sqrt{\epsilon}},\frac{y}{\sqrt{\epsilon}})=0.
\end{array}
\end{equation}

For any fixed $ y>0~(x>0\; {\rm resp.})$, letting $\epsilon\to 0$ in \eqref{div-B,0}, it follows that
\begin{equation}\label{2.4}
\partial_X u^{B,0,1}(t,X,y)=0
\quad (\partial_Y v^{B,0,2}(t,x,Y)=0 \quad {\rm resp.}),\end{equation}
and 
\begin{equation}\label{2.5}\partial_X u^{C,0}(t,X, Y) + \partial_Y v^{C, 0}(t,X, Y)
    =0.
\end{equation}

From \eqref{2.4}, we get immediately that
\begin{equation}\label{2,6}
u^{B,0,1}(t,X,y)\equiv 0,
\quad  v^{B,0,2}(t,x,Y)\equiv 0\end{equation}
by using the hypothesis of rapidly decreasing property of these two functions in $X\to+\infty$ and
$Y\to+\infty$ respectively.

In the same way as above, from \eqref{div-B,k} we can deduce that
\begin{equation}\label{div-E/C}
    \nabla_{x,y} \cdot \u^{E,k} = 0 \quad (k\ge 0), \qquad \nabla_{X,Y} \cdot \u^{C,k+1} = 0 \quad (k\ge 0),
\end{equation}
and
\begin{equation}\label{div-B-k}
    \begin{aligned}
        &\left(\partial_X u^{B,k+1,1} + \partial_y v^{B,k,1}\right) (t,X,y) = 0,\\
        &\left(\partial_x u^{B,k,2} + \partial_Y v^{B,k+1,2}\right) (t,x,Y) = 0.
    \end{aligned}
\end{equation}

\vspace{.1in}
\noindent\textbf{\underline{(2) Boundary conditions matching}} 

It is easy to see that the boundary conditions \eqref{N-BC} are equivalent to
\begin{equation}\label{BC-1}
    {u^\epsilon}_{|x=0} = 0, \quad {v^\epsilon}_{|y=0} = 0,
\end{equation}
\begin{equation}\label{BC-2}
    {\left(\dx v^\epsilon - 2\alpha v^\epsilon\right)}_{|x=0} = 0, \quad {\left(\dy u^\epsilon - 2\alpha u^\epsilon\right)}_{|y=0} = 0.
\end{equation}
Plugging \eqref{ansatz} into \eqref{BC-1} gives
\begin{equation}\label{BC1-x}
    \left({u^{E,k}}_{|x=0} + {u^{B,k,1}}_{|X=0}\right)(t,y) = 0, \quad
    \left({u^{B,k,2}}_{|x=0} + {u^{C,k}}_{|X=0}\right)(t,Y) = 0,
\end{equation}
and
\begin{equation}\label{BC1-y}
    \left({v^{E,k}}_{|y=0} + {v^{B,k,2}}_{|Y=0}\right)(t,x) = 0, \quad
    \left({v^{B,k,1}}_{|y=0} + {v^{C,k}}_{|Y=0}\right)(t,X) = 0.
\end{equation}

Secondly,  by plugging \eqref{ansatz} into  \eqref{BC-2},
the vanishing of the order $O(\epsilon^{-\frac 12})$  implies
$$
(\partial_X v^{B, 0, 1})(t,0,y)+(\partial_X v^{C,0})(t, 0, Y)=0, 
\quad
(\partial_Y u^{B, 0, 2})(t,x,0)+(\partial_Y u^{C,0})(t, X, 0)=0, 
$$
which gives
\begin{equation}\label{BC2-1}
    \begin{aligned}
         {\partial_X v^{B,0,1}}_{|X=0}(t,y) = 0, \quad {\partial_X v^{C,0}}_{|X=0}(t,Y) = 0,\\
         {\partial_Y u^{B,0,2}}_{|Y=0}(t,x) = 0, \quad {\partial_Y u^{C,0}}_{|Y=0}(t,X) = 0,
    \end{aligned}
\end{equation}
and for each $k\ge0$, the vanishing of the order $O(\epsilon^{\frac k2})$ implies
\begin{equation}\label{BC2-k}
    \begin{aligned}
        x>0:\quad &({\partial_y u^{E,k}}_{|y=0} + {\partial_Y u^{B,k+1,2}}_{|Y=0})(t,x) = 2\alpha \left({u^{E,k}}_{|y=0} + {u^{B,k,2}}_{|Y=0}\right)(t,x),\\
        X>0: \quad &({\partial_y u^{B,k,1}}_{|y=0} + {\partial_Y u^{C,k+1}}_{|Y=0})(t,X) = 2\alpha \left({u^{B,k,1}}_{|y=0} + {u^{C,k}}_{|Y=0}\right)(t, X),\\
        y>0: \quad &({\partial_x v^{E,k}}_{|x=0} + {\partial_X v^{B,k+1,1}}_{|X=0})(t,y) = 2\alpha \left({v^{E,k}}_{|x=0} + {v^{B,k,1}}_{|X=0}\right)(t,y),\\
        Y>0: \quad &({\partial_x v^{B,k,2}}_{|x=0} + {\partial_X v^{C,k+1}}_{|X=0})(t, Y) = 2\alpha \left({v^{B,k,2}}_{|x=0} + {v^{C,k}}_{|X=0}\right)(t, Y).
    \end{aligned}
\end{equation}

\vspace{.1in}
\noindent\textbf{\underline{(3) Problems of $\u^{E,k}$}}

Plugging \eqref{ansatz} into the momentum equation in \eqref{INSE}, and
letting $X,Y \to +\infty$, it gives
\begin{equation}\label{E,0}
    \partial_t \u^{E,0} + (\u^{E,0} \cdot \nabla) \u^{E,0} + \nabla p^{E,0} = 0,
\end{equation}
and for $k\ge 1$,
\begin{equation}\label{E,k}
    \partial_t \u^{E,k} + (\u^{E,0} \cdot \nabla) \u^{E,k} + (\u^{E,k} \cdot \nabla) \u^{E,0} + \nabla p^{E,k}
    = \begin{cases}
    0, & k=1,\\
    \Delta \u^{E,k-2}, & k=2,3,\\
    \Delta \u^{E,k-2} - \displaystyle\sum_{j=2}^{k-2} (\u^{E,j} \cdot \nabla) \u^{E,k-j}, & k\ge 4.
    \end{cases}
\end{equation}

From \eqref{2,6}, \eqref{BC1-x} and \eqref{BC1-y}, we deduce that the leading order outer solution $\u^{E,0}$ satisfies the incompressible Euler equation \eqref{E,0}-\eqref{div-E/C} with no-penetration boundary conditions
\begin{equation}\label{2.18}
    \u^{E,0}\cdot {\vec n}=0\quad {\rm on}\quad \partial\Q\setminus \{(0,0)\}.
\end{equation}

Before deriving the problems of boundary layers near $\{x=0\}$, $\{y=0\}$, and corner layer near $\{x=y=0\}$, we introduce notations:

\begin{notation}
(1) For functions $f(t,x,\cdot)$ and $f(t,X,\cdot)$ defined in $\{x\ge 0\}$ and $\{X\ge 0\}$ respectively, denote by 
$\overline{f}(t,\cdot) = {f}(t,0,\cdot)$ the trace
on $\{x=0\}$ and $\{X=0\}$ respectively.
and for $f(t,\cdot, y)$ and $f(t,\cdot, Y)$ defined in $\{y\ge 0\}$ and $\{Y\ge 0\}$ respectively, denote by 
$\underline{f}(t,\cdot) = f(t,\cdot, 0)$
 the trace on $\{y=0\}$ and $\{Y=0\}$ respectively.

(2) For a function $f$ defined in $\{x\ge 0, y\ge 0\}$, denote by $\hat{f}$ the value of $f$ at the corner $\{x=y=0\}$.

(3) For any $k \in \mathbb{N}$,
denote by 
\begin{equation}\label{b-1}
\begin{cases}   
(\u^{P,k,1},p^{P,k,1})(t,X,y) = (\u^{B,k,1},p^{B,k,1})(t,X,y) + \sum_{l=0}^k \frac{X^l}{l!} (\overline{\partial_x^l \u^{E,k-l}},\overline{\partial_x^l p^{E,k-l}})(t,y)\\[2mm]
(\u^{P,k,2},p^{P,k,2})(t,x,Y) = (\u^{B,k,2},p^{B,k,2})(t,x,Y) + \sum_{l=0}^k \frac{Y^l}{l!} (\underline{\partial_y^l \u^{E,k-l}},\underline{\partial_y^l p^{E,k-l}})(t,x),\\[2mm]
(\u^{G,k}, p^{G,k}) (t,X, Y)= (\u^{C,k}, p^{C,k})(t,X,Y) + \sum_{l_2 = 0}^k \frac{Y^{l_2}}{l_2!} \underline{\partial_y^{l_2} (\u^{B,k-l_2,1}, p^{B,k-l_2,1})} (t,X) \\[2mm] 
\hspace{1.6in}+ \sum_{l_1 = 0}^k \frac{X^{l_1}}{l_1!} \overline{\partial_x^{l_1} (\u^{B,k-l_1,2}, p^{B,k-l_1,2})} (t,Y) 
    \\[2mm]
    \hspace{1.6in}
    + \sum_{l_1+l_2 = 0}^k \frac{X^{l_1}Y^{l_2}}{l_1!l_2!} { \partial_x^{l_1}\partial_y^{l_2} (\u^{E,k-(l_1+l_2)}, p^{E,k-(l_1+l_2)})}(t,0,0),
\end{cases}
\end{equation}
the $O(\epsilon^{\frac k2})-$order boundary layer profiles near $\{x=0\}$, $\{y=0\}$ respectively, and  corner layer profiles near $\{x=y=0\}$.

\end{notation}

\noindent\textbf{\underline{(4) Equations of $p^{B,k,1},p^{B,k,2}$}}

Plugging the ansatz \eqref{ansatz} into the first component equation of \eqref{INSE}, and noting  that for any fixed $y \ge \delta >0$, $\epsilon \to 0$ implies $Y=\frac{y}{\sqrt{\epsilon}} \to \infty$, so one gets
\begin{equation}\label{ae-2}
\begin{aligned}
    & \sum_{k\ge -1} \epsilon^{\frac{k}{2}} 
        \left(\partial_t u^{P,k,1} + \sum_{j=1}^{k} u^{P,j,1} \partial_X u^{P,k+1-j,1} + \sum_{j=0}^{k} v^{P,j,1} \partial_y u^{P,k-j,1} + \partial_X p^{P,k+1,1} \right)\\
        = & \sum_{k\ge -1} \epsilon^{\frac{k}{2}}  \left(\partial_X^2 u^{P,k,1} + \partial_y^2 u^{P,k-2,1} \right),
\end{aligned}
\end{equation}
where by convention, we set $(\u^{P,j,1},p^{P,j,1}) = 0$ as $j\le -1$.

The vanishing of the orders $O(\epsilon^{-1/2}), O(\epsilon^0)$ in \eqref{ae-2} gives $\partial_X p^{B,0,1} = 0$ and $\partial_X p^{B,1,1} = - \overline{\partial_x p^{E,0}} = 0$ by using $u^{B,0,1}(t,X,y)= u^{E,0}(t,0,y)= 0$, which immediately implies
\begin{equation}\label{2.20}
    p^{B,0,1} = p^{B,1,1} \equiv 0.
\end{equation}

For each $k\ge 1$, the vanishing of the order $O(\epsilon^{k/2})$ in \eqref{ae-2} implies that $p^{B,k+1,1}(t,X,y)$ can be determined from
\begin{equation}\label{p-k-x}
\begin{aligned}
    \partial_X p^{P,k+1,1}   = & 
        \partial_X^2 u^{P,k,1} + \partial_y^2 u^{P,k-2,1} - 
        \dt u^{P,k,1} \\
        & - \sum_{j=1}^{k} u^{P,j,1} \dX u^{P,k+1-j,1} - \sum_{j=0}^{k-1} v^{P,j,1} \dy u^{P,k-j,1},
\end{aligned}
\end{equation}
once $\{u^{P,j,1}\}_{j\le k}$ and $\{v^{P,j,1}\}_{j\le k-1}$ are known.

Similarly, by considering the behavior of the equation
\eqref{INSE} near $\{y=0\}$, one can derive that 
\begin{equation}\label{2.22}
p^{B,0,2} = p^{B,1,2} \equiv 0,
\end{equation}
and for $k\ge 1$, one has
\begin{equation}\label{p-k-y}
\begin{aligned}
    \partial_Y p^{P,k+1,2}  = & 
        \partial_y^2 v^{P,k,2} + \partial_X^2 v^{P,k-2,2} - 
        \dt v^{P,k,2} \\
        & - \sum_{j=1}^{k} v^{P,j,2} \dY u^{P,k+1-j,2} - \sum_{j=0}^{k-1} u^{P,j,2} \dx v^{P,k-j,2}.
\end{aligned}
\end{equation}

\vspace{.1in}
\textbf{\underline{(5) Problems of boundary layer profiles}}

Plugging the ansatz \eqref{ansatz} into the first component equation of \eqref{INSE}, and noting  that for $x \ge \delta >0$, $\epsilon \to 0$ implies $X=\frac{x}{\sqrt{\epsilon}} \to \infty$, so one gets
\begin{equation}\label{ae-3}
\begin{aligned}
    & \sum_{k\ge 0} \epsilon^{\frac{k}{2}} \left(\partial_t u^{P,k,2} + \sum_{0\le j\le k} u^{P,j,2} \partial_x u^{P,k-j,2} + \sum_{0\le j\le k+1} v^{P,j,2} \partial_Y u^{P,k+1-j,2}
     +\partial_x p^{P,k,2}\right) \\
     = & \sum_{k\ge 0} \epsilon^{\frac{k}{2}} \left(\partial_Y^2 u^{P,k,2} + \partial_x^2 u^{P,k-2,2}\right).
    \end{aligned}
\end{equation}
The vanishing of the order $O(\epsilon^0)$ in \eqref{ae-3} gives rise to that the leading order boundary layer profile $(u^{P,0,2}(t,x,Y), v^{P,1,2}(t,x,Y))$ near $\{y=0\}$ satisfies the following problem of the Prandtl equation in $\{t>0, x>0, Y>0\}$: 
\begin{equation}\label{Prandtl-0-2}
    \begin{cases}
        \partial_t u^{P,0,2} + u^{P,0,2} \partial_x u^{P,0,2} + v^{P,1,2}\partial_Y u^{P,0,2} + \partial_x p^{P,0,2} = \partial_Y^2 u^{P,0,2},\quad  x,Y>0,\\
        \partial_x u^{P,0,2} + \partial_Y v^{P,1,2} = 0,\quad  x,Y>0,\\
        {\partial_Y u^{P,0,2}}_{|Y=0} = {v^{P,1,2}}_{|Y=0} = 0, \quad \lim\limits_{Y\to \infty} (u^{P,0,2} - \underline{u^{E,0}}) = 0,
    \end{cases}
\end{equation}
by using \eqref{div-B-k}
and \eqref{BC2-1}.

Similarly, one can derive the following problem for the leading order boundary layer profile $(v^{P,0,1}(t,X,y), u^{P,1,1}(t,X,y))$ near $\{x=0\}$,
\begin{equation}\label{Prandtl-0-1}
    \begin{cases}
        \partial_t v^{P,0,1} + u^{P,1,1} \partial_X v^{P,0,1} + v^{P,0,1}\partial_y v^{P,0,1} + \partial_y p^{P,0,1} = \partial_X^2 v^{P,0,1}, \quad X,y>0,\\
        \partial_X u^{P,1,1} + \partial_y v^{P,0,1} = 0,\quad X,y>0,\\
        {\partial_X v^{P,0,1}}_{|X=0} = {u^{P,1,1}}_{|X=0} = 0, \quad \lim_{Y\to \infty} (v^{P,0,1} - \overline{v^{E,0}}) = 0.
    \end{cases}
\end{equation}

It is easy to verify that
\begin{equation}\label{2.25}
\begin{array}{l}
\left(u^{P,0,2}, v^{P,1,2}, p^{P,0,2}\right)(t,x,Y)\equiv \left(u^{E,0}, Y\partial_yv^{E,0}, p^{E,0}\right)(t,x,0), \\[2mm]
\left( u^{P,1,1}, v^{P,0,1}, p^{P,0,1} \right)(t,X,y) \equiv \left(X\partial_xu^{E,0}, v^{E,0}, p^{E,0}\right)(t,0,y)
\end{array}
\end{equation}
satisfy the problems \eqref{Prandtl-0-2} and \eqref{Prandtl-0-1} respectively, by using \eqref{E,0}. Thus, we deduce that
\begin{equation}\label{2.26}
\begin{array}{l}
\left(u^{B,0,2}, v^{B,1,2}, p^{B,0,2}\right)(t,x,Y)= 0, \\[2mm]
\left(u^{B,1,1}, v^{B,0,1}, p^{B,0,1}\right)(t,X,y)=0.
\end{array}
\end{equation}
And furthermore, by \eqref{E,k} ($k=1$), the boundary conditions \eqref{BC1-x}-\eqref{BC1-y} ($k=1$) and the vanishing of $\u^{E,1}_0$, one deduces that
$$
\u^{E,1}(t,x,y) \equiv 0.
$$

By using the above properties, from the 
the vanishing of the order $O(\sqrt{\epsilon})$ in \eqref{ae-3}, and using \eqref{div-E/C}-\eqref{div-B-k} and \eqref{BC1-y},\eqref{BC2-k}, we deduce that the $O(\sqrt{\epsilon})-$order boundary layer profiles
$$
\begin{cases}
 u^{P,1,2}(t,x,Y)=u^{B,1,2}(t,x,Y)+Y\partial_yu^{E,0}(t,x,0), \\
v^{P,2,2}(t,x,Y)=v^{B,2,2}(t,x,Y)+v^{E,2}(t,x,0)+
\frac{Y^2}{2}\partial^2_yv^{E,0}(t,x,0),
\end{cases}$$
near $\{y=0\}$ satisfy the following problem for a linearized Prandtl equation defined in $\{x, Y>0\}$,
\begin{equation}\label{Prandtl-1-2}
    \begin{cases}
        \partial_t u^{P,1,2} + \underline{u^{E,0}} \partial_x u^{P,1,2} +
        u^{P,1,2}\underline{\partial_xu^{E,0}}
        +Y\underline{\partial_yv^{E,0}}\partial_Y u^{P,1,2} + \partial_x p^{P,1,2} = \partial_Y^2 u^{P,1,2},\quad x,Y>0,\\
        \partial_x u^{P,1,2} + \partial_Y v^{P,2,2} = 0,\quad x,Y>0,\\
        {\partial_Y u^{P,1,2}}_{|Y=0} =2\alpha u^{E,0}(t,x,0), \quad
        {v^{P,2,2}}_{|Y=0} = 0, \quad \lim\limits_{Y\to \infty} (u^{P,1,2}-Y\underline{\partial_yu^{E,0}}) = 0.
    \end{cases}
\end{equation}

Similarly,  by considering the behavior of the equation
\eqref{INSE} near $\{x=0\}$, 
the $O(\sqrt{\epsilon})-$order boundary layer profiles 
$$
\begin{cases}
 u^{P,2,1}(t,X,y)=u^{B,2,1}(t,X,y)+u^{E,2}(t,0,y)+
\frac{X^2}{2}\partial^2_xu^{E,0}(t,0,y),\\
v^{P,1,1}(t,X,y)=v^{B,1,1}(t,X,y)+X\partial_xv^{E,0}(t,0,y),
\end{cases}$$ satisfy the following problem for a linearized Prandtl equation defined in $\{X, y>0\}$,
\begin{equation}\label{Prandtl-1-1}
    \begin{cases}
        \partial_t v^{P,1,1} + X\overline{\partial_xu^{E,0}} \partial_X v^{P,1,1}
        +\overline{v^{E,0}}\partial_yv^{P,1,1}+v^{P,1,1}\overline{\partial_yv^{E,0}}
        + \partial_y p^{P,1,1} = \partial_X^2 v^{P,1,1},\quad X,y>0,\\
        \partial_X u^{P,2,1} + \partial_y v^{P,1,1} = 0,\quad X,y>0,\\
        {\partial_X v^{P,1,1}}_{|X=0} =2\alpha v^{E,0}(t,0,y), \quad
        {u^{P,2,1}}_{|X=0} = 0, \quad \lim\limits_{X\to \infty} (v^{P,1,1}-X\overline{\partial_xv^{E,0}}) = 0.
    \end{cases}
\end{equation}

By using \eqref{E,0}, \eqref{2.20}, and \eqref{2.22}, the problems \eqref{Prandtl-1-2} and \eqref{Prandtl-1-1} can be simplified as
the following two for $(u^{B,1,2}, v^{B,2,2})(t,x,Y)$ and 
$(u^{B,2,1}, v^{B,1,1})(t,X,y)$ respectively,
\begin{equation}\label{B-1-2}
    \begin{cases}
        \partial_t u^{B,1,2} + \underline{u^{E,0}} \partial_x u^{B,1,2} + 
        u^{B,1,2}\underline{\partial_xu^{E,0}} 
        +Y\underline{\partial_yv^{E,0}}\partial_Y u^{B,1,2} = \partial_Y^2 u^{B,1,2},\quad  x,Y>0,\\
        \partial_x u^{B,1,2} + \partial_Y v^{B,2,2} = 0,\quad  x,Y>0,\\
        {\partial_Y u^{B,1,2}}_{|Y=0} =2\alpha u^{E,0}(t,x,0)-\partial_yu^{E,0}(t,x,0),\quad
        \lim\limits_{Y\to \infty} (u^{B,1,2}, v^{B,2,2})=(0,0),
    \end{cases}
\end{equation}
and
\begin{equation}\label{B-1-1}
    \begin{cases}
        \partial_t v^{B,1,1} + X\overline{\partial_xu^{E,0}} \partial_X v^{B,1,1} 
        +\overline{v^{E,0}}\partial_yv^{B,1,1}+v^{B,1,1}\overline{\partial_yv^{E,0}}= \partial_X^2 v^{B,1,1},\quad  X,y>0,\\
        \partial_X u^{B,2,1} + \partial_y v^{B,1,1} = 0,\quad  X,y>0,\\
        {\partial_X v^{B,1,1}}_{|X=0} =2\alpha v^{E,0}(t,0,y)-\partial_xv^{E,0}(t,0,y), \quad 
        \lim\limits_{X\to \infty} (v^{B,1,1}, u^{B,2,1})=(0,0).   
    \end{cases}
\end{equation}

\vspace{.1in}
\textbf{\underline{(6) Problems of the corner layer profiles $\u^{C,k},p^{C,k}$}}

By plugging \eqref{ansatz} in \eqref{INSE} and letting $(x,y) = \sqrt{\epsilon}(X,Y)$, the equation given in \eqref{INSE} can be rewritten as
\begin{equation}\label{ae-4}
\begin{cases}
\sum_{k\ge 0} \epsilon^{\frac{k}{2}}\left(\partial_t u^{G,k} + \sum\limits_{0 \le j \le k+1} \left( u^{G,j} \partial_X + v^{G,j} \partial_Y\right) u^{G,k+1-j} + \partial_X p^{G,k+1}\right) \\
\hspace{.6in}+\epsilon^{-\frac{1}{2}}\left(( u^{G,0} \partial_X  + v^{G,0} \partial_Y )u^{G,0}+\partial_Xp^{G,0}\right)
    =\sum\limits_{k\ge 0} \epsilon^{\frac{k}{2}} \Delta_{X,Y} u^{G,k},\\[2mm]
\sum_{k\ge 0} \epsilon^{\frac{k}{2}}\left(\partial_t v^{G,k} + \sum\limits_{0 \le j \le k+1} \left( u^{G,j} \partial_X + v^{G,j} \partial_Y\right) v^{G,k+1-j} + \partial_Y p^{G,k+1}\right) \\
\hspace{.6in}+\epsilon^{-\frac{1}{2}}\left(( u^{G,0} \partial_X + v^{G,0} \partial_Y) v^{G,0}+\partial_Yp^{G,0}\right)
    =\sum\limits_{k\ge 0} \epsilon^{\frac{k}{2}} \Delta_{X,Y} v^{G,k},

\end{cases}
\end{equation}
for the corner layer profiles $u^{G,k}, v^{G,k}$ etc. defined in \eqref{b-1}.

The vanishing of $O(\epsilon^{-\frac{1}{2}})$ in \eqref{ae-4} gives that $(u^{G,0}, v^{G,0})$ satisfies
\begin{equation*}
\begin{cases}
 u^{G,0} \partial_X u^{G,0} + v^{G,0} \partial_Y u^{G,0} + \partial_X p^{G,0} = 0,\\
 u^{G,0} \partial_X v^{G,0} + v^{G,0} \partial_Y v^{G,0} + \partial_Y p^{G,0} = 0,
\end{cases}    
\end{equation*}
which implies that
$\u^{C,0} = (u^{C,0},v^{C,0})^T $ and $p^{C,0}$ satisfy the following boundary value problem  for the steady incompressible Euler equation in a quarter plane:
\begin{equation}\label{G,0}
    \begin{cases}
    (\u^{C,0} \cdot \nabla_{X,Y}) \u^{C,0} + \nabla_{X,Y} p^{C,0} = 0, \quad X,Y>0,\\
    \nabla_{X,Y}\cdot \u^{C,0} = 0, \quad X,Y>0, \\
    {u^{C,0}}_{|X=0} = {v^{C,0}}_{|Y=0} = 0,\\
    {\partial_Y u^{C,0}}_{|Y=0} = {\partial_X v^{C,0}}_{|X=0} = 0,\\
    \lim_{X\to \infty} u^{C,0} =\lim_{Y\to \infty} u^{C,0} = 0,\quad
    \lim_{X\to \infty} v^{C,0} =\lim_{Y\to \infty} v^{C,0} =0,
    \end{cases}
\end{equation}
by noting that
\begin{equation*}
  u^{G,0}(t,X,Y) =  u^{C,0}(t,X,Y), \quad 
   v^{G,0}(t,X,Y) =v^{C,0}(t,X,Y), 
\end{equation*}
and using \eqref{2.5} and \eqref{BC2-1}.
Obviously, the problem \eqref{G,0} has the trivial solution 
\begin{equation}\label{2.36}
\u^{C,0}(t,X,Y)\equiv 0, \quad 
p^{C,0}(t,X,Y)\equiv 0.
\end{equation}

From \eqref{p-k-x}-\eqref{p-k-y} with $k = 1$, one immediately derives that $\dX p^{B,2,1} = \dY p^{B,2,2} = 0$, which implies
\begin{equation}\label{2.37}
    p^{B,2,1}(t,X, y)\equiv 0, \quad p^{B,2,2}(t,x,Y) \equiv 0.
\end{equation}

The vanishing of $O(\epsilon^{\frac{k}{2}})$ for $k\ge 0$ in \eqref{ae-4} give that
\begin{equation}\label{G,k}
    \partial_t \u^{G,k} + \sum_{0 \le j \le k+1}  \u^{G,j} \nabla_{X,Y} \u^{G,k+1-j} + \nabla_{X,Y} p^{G,k+1} = \Delta_{X,Y} \u^{G,k}.
\end{equation}
 Taking $k=0$ in \eqref{G,k} and using $\u^{G,0}=0$, one gets $\nabla_{X,Y} p^{G,1} \equiv 0$, which implies that
\begin{equation}\label{2.39}
    p^{C,1}(t,X, Y) \equiv 0.
\end{equation}
Choosing $k=1$ in \eqref{G,k}, and using \eqref{B-1-2}, \eqref{B-1-1} and \eqref{E,0}, it follows that 
the corner layer profiles $\u^{C,1}=(u^{C,1}, v^{C, 1})^T(t,X,Y)$ and $p^{C,2}(t,X,Y)$ satisfy the following problem of the incompressible  Navier-Stokes type equations in $\{t>0, X>0, Y>0\}$:
\begin{equation}\label{u-C,1}
\begin{cases}
    \partial_t \u^{C,1} + (\u^{G,1} \cdot \nabla) \u^{C,1} + \nabla {p}^{C,2} - \Delta \u^{C,1} + M \u^{C,1} + B = 0, \quad X>0,Y>0\\
    \nabla \cdot \u^{C,1} = 0, \quad X>0,Y>0\\
    {u^{C,1}}_{|X=0} = - {u^{B,1,2}}(t,0,Y), \quad {\partial_Y u^{C,1}}_{|Y=0} = 0,\\
    {v^{C,1}}_{|Y=0} = -{v^{B,1,1}}(t,X,0), \quad {\partial_X v^{C,1}}_{|X=0} = 0, \\
    \lim\limits_{X,Y \to \infty} \u^{C,1} = 0,
    \end{cases}
        \end{equation}
        where $$\u^{G,1}(t,X,Y) = \begin{pmatrix}
        u^{C,1}(t,X,Y) + u^{B,1,2}(t,0,Y)+ X(\partial_x u^{E,0})(t,0,0)\\[2mm]
        v^{C,1}(t,X,Y) + v^{B,1,1}(t,X,0)+ Y(\partial_y v^{E,0})(t,0,0)       \end{pmatrix}$$
        and
        $$
        M = \begin{pmatrix}
            {\dx u^{E,0}}(t,0,0) & \dY {u^{B,1,2}}(t,0,Y) \\
            \dX {v^{B,1,1}}(t,X,0) & {\dy v^{E,0}}(t,0,0)
        \end{pmatrix},\quad B = 
        \begin{pmatrix}
            {v^{B,1,1}}(t,X,0) \dY {u^{B,1,2}}(t,0,Y) \\
            {u^{B,1,2}}(t,0,Y) \dX {v^{B,1,1}}(t,X,0)
        \end{pmatrix}.
        $$

The above formal derivation can be summarized as follows.

\begin{concl}\label{concl}
    Under the well-prepared condition of $\u^\epsilon_0$ given in Theorem \ref{thm-1}, the solution $(\u^\epsilon, p^\epsilon)(t,x,y)$ to the problem \eqref{INSE} formally have the following asymptotic expansions:
    \begin{equation}\label{Exp}
        \begin{cases}
    \begin{aligned}
\u^\epsilon = & \u^{E,0}(t,x,y) + \sqrt{\epsilon} \left(
(u^{B,1,2}(t,x,\frac{y}{\sqrt{\epsilon}}), v^{B,1,1}(t,\frac{x}{\sqrt{\epsilon}},y))^T+
\u^{C,1}(t,\frac{x}{\sqrt{\epsilon}},\frac{y}{\sqrt{\epsilon}}) \right)  \\
& + \sum_{k\ge 2} \epsilon^{\frac{k}{2}} \left(\u^{E,k}(t,x,y) + \u^{B,k,1}(t,\frac{x}{\sqrt{\epsilon}},y) + \u^{B,k,2}(t,x,\frac{y}{\sqrt{\epsilon}}) + \u^{C,k}(t,\frac{x}{\sqrt{\epsilon}},\frac{y}{\sqrt{\epsilon}})\right),
\end{aligned}\\
\begin{aligned}
p^\epsilon = & p^{E,0}(t,x,y)
+ \epsilon \left(p^{E,2}(t,x,y) + p^{C,2}(t,\frac{x}{\sqrt{\epsilon}},\frac{y}{\sqrt{\epsilon}}) \right)\\
& + \sum_{k\ge 3} \epsilon^{\frac{k}{2}} \left(p^{E,k}(t,x,y) + p^{B,k,1}(t,\frac{x}{\sqrt{\epsilon}},y) + p^{B,k,2}(t,x,\frac{y}{\sqrt{\epsilon}}) + p^{C,k}(t,\frac{x}{\sqrt{\epsilon}},\frac{y}{\sqrt{\epsilon}})\right),
\end{aligned}
        \end{cases}
    \end{equation}
    with boundary layer and corner layer profiles $\u^{B,k,1}, \u^{B,k,2}, p^{B,k,1}, p^{B,k,2}, \u^{C,k}, p^{C,k}$ being rapidly decreasing when $X=\frac{x}{\sqrt{\epsilon}}$ and $Y=\frac{y}{\sqrt{\epsilon}}$ go to infinity respectively, where
    \begin{enumerate}
        \item[(1)] $(\u^{E,0},p^{E,0})$ are solutions to the problem \eqref{limit} of the Euler equations,
        and for $k\ge 2$, $(\u^{E,k},p^{E,k})$ are solutions to problems of the linearized Euler equations \eqref{E,k} with the divergence-free condition \eqref{div-E/C} and the boundary conditions
        \begin{equation*}
            {u^{E,k}}_{|x=0} = - {u^{B,k,1}}_{|X=0}, \quad {v^{E,k}}_{|y=0} = - {v^{B,k,2}}_{|Y=0}, \quad k\ge 2.
        \end{equation*}

\item[(2)] the boundary layer profiles $(u^{B,1,2},v^{B,2,2})(t,x,Y)$ ($(v^{B,1,1},u^{B,2,1})(t,X,y)$ resp.)
        near the boundary $\{y=0\}$ ($\{x=0\}$ resp.) satisfy the problem \eqref{B-1-2}
        (\eqref{B-1-1} resp.) of the linearized Prandtl equations, and the higher order profiles $(u^{B,k,2},v^{B,k+1,2})(t,x,Y)$ and $(v^{B,k,1},u^{B,k+1,1})(t,X,y)$ with $k\ge 2$ satisfy similar problems, which can be derived from \eqref{ae-3}.

        \item[(3)] the corner layer profiles $(\u^{C,1},p^{C,2})(t,X,Y)$ satisfy the problem \eqref{u-C,1} of the Navier-Stokes equations, and the higher order corner layer profiles $(\u^{C,k},p^{C,k+1})(t,X,Y)$ with $k\ge 2$ satisfy problems of the linearized Navier-Stokes equations, which can be derived from \eqref{G,k}.
    \end{enumerate}
\end{concl}

\section{Well-posedness of the outer flow and boundary-layer problems}
In this section, we   proceed to determine partial profiles appeared in the expansions given in \eqref{Exp}, including the leading outer flow profile $(\u^{E,0}, p^{E,0})$, the boundary layer profiles along the edges, $(u^{B,1,2},v^{B,2,2})$, $(v^{B,1,1},u^{B,2,1})$ and the second order outer profile $\u^{E,2}$.

\subsection{Well-posedness of the Euler equation in a quarter plane}

To study the problem \eqref{limit} of the Euler equation in the quarter plane $\Q=\{x>0, y>0\}$ with the initial data ${\bf u}_0(x,y)$, first, let us recall some classical results from  \cite{MR3054600, bardos2013euler}. If $\Omega$ is a bounded, simply connected subset of $\R^2$ whose
boundary is a piecewise $C^{1,\gamma}$ ($\gamma>0$) Jordan curve with finitely many angles smaller than $\pi$, there is a global weak solution of the Euler equations on $\Omega$ in the Yudovich class, i.e. $\u^{0} \in L^\infty(\R_+, L^2(\Omega)), \nabla\times \u^{0} \in L^\infty(\R_+\times \Omega)$. Moreover, by using the  elliptic regularity theory, one can also show that if the angles are no less than $\frac{\pi}{2}$, then the solution $\u^{0}$ has at most $H^{2-\sigma}$ ($\sigma >0$) regularity in space. 
To have better regularity, we shall take advantage of the hypothesis that $\nabla\times \u_0$ is supported away from the corner point. Indeed, by using \cite{MR3981393} one gets that for any fixed $T>0$, $\nabla\times \u^{0}$ is supported away from the corner point as well, which implies that $\u^{0}$ is harmonic in a small neighborhood of $\mathbf{0}$.

Moreover, the quarter plane offers a significant technical advantage, which enables us to construct solutions to \eqref{limit} via {\it even-odd extensions}, in the sense that extends $u^{0}$ (resp. $v^{0}$) oddly across $x=0$ (resp. $y=0$) then evenly across $y=0$ (resp. $x=0$), and extends $p^{0}$ evenly across both $x=0$ and $y=0$. It is easy to know that the resulting functions $(\widetilde{\u}^{0},\widetilde{p}^{0})$ satisfy the following initial value problem:
\begin{equation}\label{e,o-ext}
    \begin{cases}
        \partial_t \widetilde{\u}^{0} + \widetilde{\u}^{0} \cdot \nabla \widetilde{\u}^{0} + \nabla \widetilde{p}^{0} = 0, \quad \text{in }\R^2,\\
        \nabla \cdot \widetilde{\u}^{0} = 0, \quad \text{in }\R^2,\\
        \widetilde{\u}^{0}(0,\cdot) = \widetilde{\u}_0,
    \end{cases}
\end{equation}
where $\widetilde{\u}_0$ is the extension of $\u_0$ following the same procedure as above. Note that the even-odd extension generally only ensures that $\widetilde{\u}_0$ remains bounded in $H^{\frac{3}{2}-\varepsilon}(\R^2)$, since we lack the necessary conditions -- such as $\partial_x^{2l} u_0(0,y) = 0~ (l\ge 1)$ or $\partial_y^{2l+1} u_0(x,0) = 0~(l\ge 0)$ -- to preserve higher-order regularity in standard Sobolev spaces. 
However, this extension does not give away the regularity of $\widetilde{\u}_0$ in the following conormal spaces:
\begin{equation*}
    H^m_{\co}(\R^2) := \left\{u\left| \, \sum_{k+l\le m} (\chi(x)\partial_x)^k (\chi(y)\partial_y)^l u \in L^2(\R^2) \right.\right\},
\end{equation*}
with $\chi(\cdot)$ being a smooth cut-off function on $\R$, $\chi(x) = x$ if $|x|\le \frac{1}{2}$, $\chi(x) \equiv \sgn(x)$ if $|x|\ge 1$. $H^m_{\co}(\R^2)$ is a Hilbert space with its canonical norm
\begin{equation*}
    \|u\|_{m,\co}: = \sum_{k+l\le m} \|(\chi(x)\partial_x)^k (\chi(y)\partial_y)^l u \|_{L^2(\R^2)}.
\end{equation*}
Similarly, we define 
\begin{equation*}
    W^{m,\infty}_{\co} (\R^2) := \left\{u\left| \, \sum_{k+l\le m} (\chi(x)\partial_x)^k (\chi(y)\partial_y)^l u \in L^\infty(\R^2)\right. \right\}.
\end{equation*}
In the sequel, the domain $\R^2$ in the conormal spaces may be replaced by $\Q$ or its subset.

Now we introduce a lemma that will be useful throughout this subsection.
\begin{lemma}\label{lemma-tan-vor}
    Suppose that $\u = (u,v)^T$ is a divergence-free smooth vector field defined in $\Q$, and $\omega := \nabla \times \u = 0$ in $\Q_r:=(0,r)^2$ for some $r>0$, then for any fixed integer $k \ge 1$, $\nabla^k \u$ can be decomposed in two ways in $\Q$:
    \begin{equation}\label{decom-1}
        \nabla^k \u = D^{k-1}_{I} \omega + \partial_y^k \mathbf{U}_{I},
    \end{equation}
    and
    \begin{equation}\label{decom-2}
        \nabla^k \u = D^{k-1}_{I\!I} \omega + \partial_x^k \mathbf{U}_{I\!I},
    \end{equation}
    where $D^{k-1}_{I}, D^{k-1}_{I\!I}$ are vectors consisting of $(k-1)$-th order differential operators acting on the vorticity $\omega$, and $\mathbf{U}_I, \mathbf{U}_{I\!I}$ are vectors whose components are either $\pm u$ or $\pm v$. Moreover, one has 
    \begin{equation}\label{4.4}
        \|\nabla^k \u\|_{L^2(\Q\backslash\Q_r)} \lesssim \|\u\|_{k,\co} + \|\omega \|_{H^{k-1}(\Q)}.
    \end{equation}
\end{lemma}
\begin{proof}
(1) When $k = 1$, by using the divergence-free constrain of $\u$, it holds
    $$
    \begin{aligned}
        \nabla \u = (\dx u, \dy u, \dx v, \dy v) = & (-\dy v, \dy u, \omega + \dy u, \dy v)\\
        = & (\dx u, \dx v - \omega, \dx v, -\dx u),
    \end{aligned}
    $$
    which can be written as the form of \eqref{decom-1} or \eqref{decom-2} with $k=1$. For general $k\ge 1$, the decomposition formulae \eqref{decom-1} and \eqref{decom-2} can be easily obtained by induction on $k$ and the same calculation as above. 

    (2)
    The estimate \eqref{4.4} is a  direct consequence from the decompositions \eqref{decom-1} and \eqref{decom-2}. Indeed, if denote by $\Sigma_1 = \{x\in (0,r),y>r\}$ and $\Sigma_2 = \{x>r,y\in (0,r)\}$, then
    \begin{equation*}
        \begin{aligned}
            \|\nabla^k \u\|_{L^2(\Q\backslash\Q_r)} \lesssim & \|D_I^{k-1}\omega\|_{L^2(\Sigma_1)} + \|\partial_y^k \mathbf{U}_{I}\|_{L^2(\Sigma_1)} \\
            & + \|D_{I\!I}^{k-1}\omega\|_{L^2(\Sigma_2)} + \|\partial_x^k \mathbf{U}_{I\!I}\|_{L^2(\Sigma_2)} + \|\nabla^k \u\|_{L^2(\Q\backslash (\Sigma_1\cup \Sigma_2 \cup \Q_r))} \\
             \lesssim & \|\u\|_{H^k_{\co}(\Q\backslash \Q_r)} + \|\omega\|_{H^{k-1}(\Sigma_1\cup \Sigma_2)}.
        \end{aligned}
    \end{equation*}
\end{proof}


To study the well-posedness of the problem \eqref{limit}  for the Euler equation in the quarter plane, we first introduce the space
$E^m = \{\u|~ \u \in H^m_{\co}(\Q), \nabla\times \u \in H^{m-1}_{\co}\cap W^{1,\infty}(\Q)\}$.

\begin{proposition}\label{prop-Euler}
    Let $m\ge 4$ be an integer, $\u_0 \in E^m$ satisfying that 
    $\nabla\cdot \u_0 = 0$, ${\u_0 \cdot \vec{n}}_{|\partial\Q} = 0$, and 
    $\nabla\times \u_0 = 0$ in a small neighborhood of $\mathbf{0}$. 
    Then, there exists $T>0$ such that the problem \eqref{limit} with initial data $\u_0$ has a unique solution  $(\u^{0},p^{0})$ in $0<t<T$, and $(\u^{0}, \nabla p^{0}) \in L^\infty(0,T; E^m) \times L^\infty(0,T; H^{m-1}_{\co}(\Q))$. Moreover, there exists $\delta>0$ and a constant $C$, such that $\omega^{0} := \nabla \times \u^{0} = 0$ in $Q_\delta = (0,\delta)^2$, and for a.e. $t\in (0,T)$,
    \begin{equation}\label{est-E,0}
        \begin{aligned}
            & \|\u^{0}(t)\|^2_{m,\co} + \|\omega^{0}(t)\|_{m-1,\co}^2 + \|\omega^{0}(t)\|_{W^{1,\infty}(\Q)}^2 \le C.
        \end{aligned}
    \end{equation}
\end{proposition}

\begin{proof}
Throughout the proof, we focus on establishing a priori estimate \eqref{est-E,0} for a sufficiently smooth solution of \eqref{limit}. The existence of a solution can be proved via a compactness argument. 
Indeed, first denote by $\widetilde{\u}_0$
the even-odd extension  of $\u_0$ as defined in the above of \eqref{e,o-ext}. Obviously, it holds
\begin{equation}\label{est-E,0-bdd}
    \|\widetilde{\u}_0\|_{H^m_{\co}(\R^2)} + \|\nabla \times \widetilde{\u}_0\|_{H^{m-1}_{\co}(\R^2) \cap W^{1,\infty}(\R^2)} \lesssim 
    \|\u_0\|_{H^m_{\co}(\Q)} + \|\nabla \times {\u}_0\|_{H^{m-1}_{\co}(\Q) \cap W^{1,\infty}(\Q)},
\end{equation}
by using ${\u_0 \cdot \vec{n}}_{|\partial\Q} = 0$. Now, 
we construct a sequence of smooth approximations $\widetilde{\u}^{(k)}_0 \in C^\infty_c(\R^2)$ of $\widetilde{\u}_0$, satisfying the uniform bound \eqref{est-E,0-bdd}. For each $k$, we invoke a standard well-posedness result for incompressible Euler equations on $\R^2$ to obtain a unique smooth solution $\widetilde{\u}^{(k)}$ to the extended Cauchy problem \eqref{e,o-ext} on a time interval $[0, T_k]$. Using the a priori estimates (which are independent of $k$), we show that there exists a common time $T>0$ such that all $\widetilde{\u}^{(k)}$ exist on $[0, T]$ and satisfy the uniform bounds \eqref{est-E,0}. Moreover, the uniform bounds yield that $\widetilde{\u}^{(k)}$ is bounded in $L^\infty(0,T; H^m_{\co}(\R^2))$ and $\partial_t \widetilde{\u}^{(k)}$ is bounded in $L^2(0,T; H^{-1}_{\operatorname{loc}}(\R^2))$. Applying the Aubin-Lions compactness lemma and by a diagonal extraction process, we obtain a subsequence of $\widetilde{\u}^{(k)}$  that converges strongly in $C([0, T]; H^{m-1}_{\co,\operatorname{loc}}(\R^2))$ to a limit $\widetilde{\u}^{0}$. The uniform bounds and weak lower semi-continuity imply $\widetilde{\u}^{0}\in L^\infty(0,T; H^m_{\co}(\R^2))$ and $\nabla \times \widetilde{\u}^{0} \in L^\infty(0,T; H^{m-1}_{\co}(\R^2)\cap W^{1,\infty}(\R^2))$.
Using the strong convergence in $H^{m-1}_{\co,\operatorname{loc}}$ and the weak convergence of gradients (from the uniform bound in $H^m_{\co}$), we pass to the limit in the nonlinear terms to show that $\widetilde{\u}^{0}$ satisfies the extended Euler equations in the sense of distributions; the regularity then ensures it is a strong solution. The solution to the original problem is obtained by restricting $\widetilde{\u}^{0}$ to $\Q$. The uniqueness follows from the Lipschitz regularity of the velocity field.

    From now on, let us fix a $T>0$, and assume that $(\u^{0},p^{0})$ is a sufficiently smooth solution to \eqref{limit} on $(0,T)$. Denote by for all $t\in (0,T)$,
    \begin{equation}\label{Q(u,t)}
        \begin{aligned}
        & Q_m(\u^{0}; t) := \|\u^{0}(t)\|^2_{m,\co} 
        + \|\nabla\times\u^{0}(t)\|^2_{m-1,\co} + \|\nabla\times\u^{0}(t)\|_{W^{1,\infty}(\Q)}^2.
        \end{aligned}
    \end{equation}
    
    \noindent \underline{\textbf{Reduction to $Q_m(\u^{0}; t)$.}}
    
    Without loss of generality, we assume that
    \begin{equation}\label{cdt-vor'}
        \supp\left\{\nabla\times \u^{0}(t)\right\} \subset \Q\backslash \Q_2 \quad \forall t\in [0,T],
    \end{equation}
    where and in the sequel, $\Q_r$ denotes the cube $(0,r)^2$.
    This means that the solution $\u^{0}$ is harmonic in $\Q_2$. Then from $\u^{0}\cdot \vec{n}_{|\partial\Q} = 0$,
    $\u^{0}$ has the following trace properties:
    \begin{equation*}
        \begin{aligned}
        {\dx^{2k} u^{0}}_{|x=0, y \in (0,2)} = 0, \quad {\partial^{2k+1}_y u^{0}}_{|y = 0, x\in (0,2)} = 0,
        \end{aligned}
    \end{equation*}
    as long as the above traces make sense. Therefore, one can extend ${u^{0}}_{|\Q_2}$ oddly across $x=0$ and evenly across $y=0$, the resulting function $\widetilde{u}^{0}$ is still harmonic in $\widetilde{\Q}_2 = (-2,2)^2$.
    Similarly, one can also extend ${v^{0}}_{|\Q_2}$ evenly across $x=0$ and oddly across $y=0$,  the resulting function is denoted as $\widetilde{v}^{0}$. Then $\widetilde{\u}^{0} := (\widetilde{u}^{0}, \widetilde{v}^{0})$ is harmonic in $\widetilde{\Q}_2$, and
    \begin{equation*}
        \|\widetilde{\u}^{0}(t)\|_{H^m(\widetilde{\Q}_2)} \simeq \|\u^{0}(t)\|_{H^m(\Q_2)}. 
    \end{equation*}

    Denote by $B_r$ the disk $\{(x,y):x^2+y^2<r^2\}$ and $A_{r_1,r_2}$ the angulus $B_{r_2}\backslash B_{r_1}$ ($r_2>r_1$). Then, by the elliptic estimates on a disk and a trace theorem, one has
    \begin{equation*}
        \begin{aligned}
            \|\widetilde{\u}^{0}(t)\|_{H^m(\widetilde{\Q}_1)} \lesssim & \|\widetilde{\u}^{0}(t)\|_{H^m(B_{\frac{3}{2}})} \lesssim \|\widetilde{\u}^{0}(t)\|_{H^{m-\frac{1}{2}}(\partial B_{\frac{3}{2}})} \\ \lesssim & \|\widetilde{\u}^{0}(t)\|_{H^{m}(A_{\frac{3}{2},2})}  \lesssim \|{\u}^{0}(t)\|_{H^{m}(\Q\cap A_{\frac{3}{2},2})}.
        \end{aligned}
    \end{equation*}
    Then, by the above elliptic estimate and the decompositions \eqref{decom-1}-\eqref{decom-2} in Lemma \ref{lemma-tan-vor},
    one has
    \begin{equation}\label{4.9}
        \|\u^{0}(t)\|_{H^m(\Q_1)}
        \lesssim
        \|\u^{0}(t)\|_{m,\co} \le Q^{\frac{1}{2}}_m(\u^{0};t).
    \end{equation}
    We apply the Hodge estimates \eqref{Hodge} (see \cite[Appendix B]{guo2025well} for the weighted estimates) to get
    \begin{equation}\label{est-E,0-Hodge}
        \|\nabla \u^{0}(t)\|_{m-1,\co} \lesssim \|\u^{0}(t)\|_{m-1,\co} + \|\omega^{0}(t)\|_{m-1,\co}.
    \end{equation}
    By using Lemma \ref{lemma-tan-vor}, the Sobolev embedding theorem, and \eqref{4.9} we obtain
    \begin{equation}\label{Q4}
        \begin{aligned}
            \|\u^{0}(t)\|_{W^{2,\infty}(\Q)} \lesssim & \|\u^{0}(t)\|_{W^{2,\infty}(\Q_1)}  + \|\u^{0}(t)\|_{W^{2,\infty}_{\co}(\Q\backslash\Q_1)} + \|\omega^{0}(t)\|_{W^{1,\infty}(\Q)}\\
        \lesssim & Q^{\frac{1}{2}}_4(\u^{0};t) + \|\u^{0}(t)\|_{H^1_{x,y} H^3_{\co}(\Q\backslash\Q_1)} \\
        \lesssim & Q^{\frac{1}{2}}_4(\u^{0};t) .
        \end{aligned}
    \end{equation}
    
    \noindent \underline{\textbf{Conormal estimates for $\u^{0}$.}}

    By acting $Z^\alpha = (\chi(x)\partial_x)^{\alpha_1}(\chi(y)\partial_y)^{\alpha_2}$ with $|\alpha|\le m$ on the problem \eqref{limit}, we get that
    \begin{equation}\label{est-E,0-co}
        (\partial_t+ \u^{0} \cdot \nabla)( Z^\alpha \u^{0}) + \nabla Z^\alpha p^{0} = [\u^{0}\cdot \nabla, Z^\alpha] \u^{0} + [\nabla, Z^\alpha] p^{0}.
    \end{equation}
    Noting that from the divergence-free constrain and the boundary condition given  in \eqref{limit}, it follows
    \begin{equation*}
        \nabla\cdot (Z^\alpha \u^{0}) = [\nabla\cdot, Z^\alpha] \u^{0},
    \end{equation*}
    and 
    \begin{equation*}
        {Z^\alpha u^{0}}_{|x=0} = {Z^\alpha v^{0}}_{|y=0} = 0,
    \end{equation*}
   so by employing the standard energy estimate for the equation \eqref{est-E,0-co} and using the boundary condition, it follows
    \begin{equation}\label{4.13}
        \begin{aligned}
            \frac{\mathrm{d}}{2\mathrm{d} t} \|Z^\alpha \u^{0}(t)\|^2  \le & |\inp*{Z^\alpha p^{0}}{[\nabla\cdot, Z^\alpha]\u^{0}(t)}| + |\inp*{[\nabla, Z^\alpha] p^{0}}{Z^\alpha \u^{0}(t)}| \\
            & + |\inp*{[\u^{0}\cdot \nabla, Z^\alpha] \u^{0}}{Z^\alpha \u^{0}(t)}|.
        \end{aligned}
    \end{equation}
    Now, we estimate the right hand side of \eqref{4.13} term by term. First, when $\alpha_1 \ge 1$, there exists smooth functions $c_{\alpha_1,j}$ for $0\le j\le \alpha_1-1$, such that
    \begin{equation*}
        \begin{aligned}
            [\partial_x, (\chi(x)\partial_x)^{\alpha_1}] & =  \sum_{0\le j\le \alpha_1 -1} c_{\alpha_1,j}(x) \partial_x (\chi(x)\partial_x)^j,
        \end{aligned}
    \end{equation*}
 which implies that 
    \begin{equation*}
        \begin{aligned}
            & \quad |\inp*{(\chi(x)\partial_x)^{\alpha_1} p^{0}}{[\nabla\cdot, (\chi(x)\partial_x)^{\alpha_1}]\u^{0}(t)}| \\
            \le & \sum_{0\le j\le \alpha_1 - 1}
            |\inp*{c_{\alpha_1,j}
            \partial_x (\chi(x)\partial_x)^{\alpha_1-1} p^{0}}{(\chi(x)\partial_x)^{j+1}u^{0}(t)}|,
        \end{aligned}
    \end{equation*}
   as $\alpha_1 \ge 1$. The same holds for the commutator $[\partial_y, (\chi(y)\partial_y)^{\alpha_2}]$ with $\alpha_2\ge 1$. Thus, one can conclude that
    \begin{equation}\label{4.14}
        |\inp*{Z^\alpha p^{0}}{[\nabla\cdot, Z^\alpha]\u^{0}(t)}| \lesssim \|\nabla p^{0}(t)\|_{m-1,\co} \|\u^{0}(t)\|_{m,\co}.
    \end{equation}
    Also, we have
    \begin{equation}\label{4.15}
        |\inp*{[\nabla, Z^\alpha] p^{0}}{Z^\alpha \u^{0}(t)}| \lesssim \|\nabla p^{0}(t)\|_{m-1,\co} \|\u^{0}(t)\|_{m,\co}.
    \end{equation}
    Finally,
    \begin{equation}\label{4.16}
        \begin{aligned}
            [\u^{0}\cdot \nabla, Z^\alpha] \u^{0} & = \sum_{\beta + \gamma = \alpha, \beta \ne 0} c_{\beta,\gamma} Z^\beta \u^{0} \cdot Z^\gamma \nabla \u^{0} + \left(\u^{0} \cdot [Z^\alpha, \nabla]\right)\u^{0},
        \end{aligned}
    \end{equation}
    where $c_{\beta,\gamma}$ are smooth functions. We observe that
    \begin{equation}\label{est-E,0-4}
        \begin{aligned}
            \|\left(\u^{0} \cdot [Z^\alpha, \nabla]\right)\u^{0}(t)\|  \lesssim 
            &
            \sum_{0\le j\le \alpha_1 - 1}\|c_{\alpha_1,j}(x) u^{E,0} \partial_x (\chi(x)\partial_x)^j (\chi(y)\partial_y)^{\alpha_2}\u^{0}(t)\|\\
            & + \sum_{0\le j\le \alpha_2 - 1}\|c_{\alpha_2,j}(y) v^{0} (\chi(x)\partial_x)^{\alpha_1}  \partial_y (\chi(y)\partial_y)^j \u^{0}(t)\|\\
             \lesssim 
             &
             \left(\left\|\frac{u^{0}(t)}{\chi(x)}\right\|_{L^\infty(\Q)} + \left\|\frac{v^{0}(t)}{\chi(y)}\right\|_{L^\infty(\Q)} \right)\|\u^{0}(t)\|_{m,\co}\\
             \lesssim 
             &
             \|\u^{0}(t)\|_{W^{1,\infty}(\Q)}\|\u^{0}(t)\|_{m,\co},
        \end{aligned}
    \end{equation}
    where we have used Hardy's inequality.
    To estimate the other terms given in \eqref{4.16}, we  use the following generalized Sobolev-Gagliardo-Nirenberg-Moser inequality, its proof can be found \cite{gues1990probleme, adams2003sobolev} for  example.
    \begin{lemma}\label{lemma-1}
        For $k\ge 2$, $u,v \in L^\infty\cap H^k_{\co}(\Q)$, we have
        \begin{equation*}
            \|Z^{\alpha_1} u Z^{\alpha_2} v\| \lesssim \|u\|_{L^\infty} \|v\|_{k,\co} +  \|u\|_{k,\co} \|v\|_{L^\infty}, \quad |\alpha_1| + |\alpha_2| = k.
        \end{equation*}
    \end{lemma}
    Thus for $\beta\ne 0, |\beta| + |\gamma| = |\alpha|\le m$, we have
    \begin{equation}\label{4.18}
        \begin{aligned}
            & \quad \|c_{\beta, \gamma} Z^{\beta} \u^{0} \cdot Z^{\gamma} \nabla \u^{0}(t)\|\\
            \le & \|\u^{0}(t)\|_{W^{1,\infty}_{\co}(\Q)} \|\nabla \u^{0}(t)\|_{m-1,\co} +  \|\u^{0}(t)\|_{m,\co} \|\u^{0}(t)\|_{W^{1,\infty}(\Q)}.
        \end{aligned}
    \end{equation}
    
    Substituting estimates \eqref{4.14}-\eqref{4.18} into \eqref{4.13}, it follows
    \begin{equation}\label{est-E,0-3}
        \begin{aligned}
            \frac{\mathrm{d}}{2\mathrm{d} t} \|\u^{0}(t)\|^2_{m,\co} \lesssim 
            &
            \|\nabla p^{0}(t)\|_{m-1,\co} \|\u^{0}(t)\|_{m,\co} + \|\u^{0}(t)\|_{W^{1,\infty}(\Q)}\|\u^{0}(t)\|^2_{m,\co}\\
            & + \|\u^{0}(t)\|_{W^{1,\infty}_{\co}(\Q)} \|\omega^{0}(t)\|^2_{m-1,\co} \\
             \lesssim 
             &
             \|\nabla p^{0}(t)\|^2_{m-1,\co}  + \left(1+Q^{\frac{1}{2}}_4(\u^{0};t)\right) Q_m(\u^{0}; t).
        \end{aligned}
    \end{equation}

    \noindent \underline{\textbf{Conormal estimates for $\omega^{0}$.}}

    The vorticity equation of \eqref{limit} is simply a transport equation
    \begin{equation}\label{E,0-vor}
        \partial_t \omega^{0} + \u^{0} \cdot \nabla \omega^{0} = 0.
    \end{equation}
    By applying $Z^\alpha$ ($|\alpha| \le m-1$) to the equation \eqref{E,0-vor}, we obtain
    \begin{equation}\label{est-E,0-vor-co}
        \partial_t( Z^\alpha \omega^{0}) + \left(\u^{0} \cdot \nabla\right) Z^\alpha \omega^{0} = [\u^{0} \cdot \nabla, Z^\alpha] \omega^{0}.
    \end{equation}
    By using the energy method for the above equation, it gives
    \begin{equation*}
        \begin{aligned}
            &\quad \frac{\mathrm{d}}{2\mathrm{d}t} \|Z^\alpha \omega^{0}(t)\|^2 \le |\inp*{[\u^{0} \cdot \nabla, Z^\alpha] \omega^{0}}{Z^\alpha \omega^{0}(t)}|\\
            \lesssim & \left(\sum_{\beta + \gamma = \alpha, \beta \ne 0}\left\|c_{\beta,\gamma}Z^\beta \u^{0} \cdot Z^\gamma \nabla \omega^{0}(t)\right\| + \left\|\left(\u^{0}\cdot [Z^\alpha, \nabla] \right)\omega^{0}(t)\right\|\right) \|Z^\alpha \omega^{0}(t)\|.
        \end{aligned}
    \end{equation*}
    for which by a similar argument as in \eqref{est-E,0-4} and with the help of \eqref{Q4}, we have
    \begin{equation*}
        \begin{aligned}
            \left\|\left(\u^{0}\cdot [Z^\alpha, \nabla] \right)\omega^{0}(t)\right\|
            \lesssim \|\u^{0}(t)\|_{W^{1,\infty}(\Q)} \|\omega^{0}(t)\|_{m-1,\co} \lesssim Q^{\frac{1}{2}}_4(\u^{0};t) Q^{\frac{1}{2}}_m(\u^{0};t),
        \end{aligned}
    \end{equation*}
    and also by Lemma \ref{lemma-1} and Hardy's inequality,
    \begin{equation*}
        \begin{aligned}
        & \quad \sum_{\beta+ \gamma = \alpha, \beta\ne 0}\|c_{\beta,\gamma}Z^\beta \u^{0} \cdot Z^\gamma \nabla \omega^{0}(t)\| \\
        \lesssim & \sum_{\beta + \gamma \le \alpha, |\gamma|\le m-2}\left\| Z^\beta \left(\frac{ u^{0}}{\chi(x)} \right) Z^\gamma (\chi(x)\partial_x) \omega^{0}(t)\right\| + \left\| Z^\beta \left(\frac{ v^{0}}{\chi(y)} \right) Z^\gamma (\chi(y)\partial_y) \omega^{0}(t)\right\| \\
        \lesssim & \|\u^{0}(t)\|_{W^{2,\infty}(\Q)} \|\omega^{0}(t)\|_{m-1,\co}\\
        & \hspace{.2in} + \left(\|\u^{0}(t)\|_{m-1,\co} + \|\nabla \u^{0}(t)\|_{m-1,\co} \right)\|\omega^{0}(t)\|_{W^{1,\infty}_{\co}(\Q)}\\
        \lesssim & Q^{\frac{1}{2}}_4(\u^{0};t) Q^{\frac{1}{2}}_m (\u^{0};t),
        \end{aligned}
    \end{equation*}
    where we have used \eqref{Q4} and \eqref{est-E,0-Hodge} for the last inequality.
    Then we conclude that
    \begin{equation}\label{est-E,0-5}
        \begin{aligned}
            \frac{\mathrm{d}}{2\mathrm{d}t} \|\omega^{0}(t)\|^2_{m-1,\co} \le & \sum_{|\alpha|\le m-1}|\inp*{[\u^{0} \cdot \nabla, Z^\alpha] \omega^{0}}{Z^\alpha \omega^{0}(t)}|\\
            \lesssim & Q^{\frac{1}{2}}_4(\u^{0}; t) Q_m(\u^{0}; t).
        \end{aligned}
    \end{equation}

    \noindent \underline{\textbf{Pressure estimates.}}
    Now we estimate $\|\nabla p^{0}(t)\|^2_{m-1,\co}$.
    By taking the divergence of \eqref{limit} and from the boundary conditions $(\u^{0}\cdot \vec{n})_{|\partial\Q} = 0$, we get that $p^{0}$ solves
    \begin{equation}\label{est-E,0-p}
        \begin{cases}
            - \Delta p^{0} = \nabla\cdot \left(\u^{0} \cdot \nabla \u^{0}\right), \quad \text{in }\Q,\\
        {\partial_x p^{0}}_{|x=0} = 0, \quad {\partial_y p^{0}}_{|y=0} = 0.
        \end{cases}
    \end{equation}
    This means that $p^{0}$ corresponds to the gradient part of the usual Leray-Hodge decomposition of the vector field $\u^{0}\cdot \nabla \u^{0}$. 
    
    By standard energy method, one gets
    \begin{equation*}
        \begin{aligned}
            \|\nabla p^{0}(t)\|^2 = -\inp*{\nabla p^{0}}{\u^{0}\cdot \nabla \u^{0}(t)} \le \|\nabla p^{0}(t)\| \|\u^{0}\cdot \nabla \u^{0}(t)\|,
        \end{aligned}
    \end{equation*}
    which implies
    \begin{equation*}
        \|\nabla p^{0}(t)\| \le \|\u^{0}\cdot \nabla \u^{0}(t)\|.
    \end{equation*}
    For $k \ge 1$ and $|\alpha| = k$, by applying $Z^k$ to \eqref{est-E,0-p} it gets
    \begin{equation*}
        - \Delta \left(Z^\alpha p^{0}\right) = \nabla \cdot \left(Z^\alpha(\u^{0}\cdot \nabla \u^{0})\right) + [Z^\alpha, \nabla \cdot] \left(\u^{0}\cdot \nabla \u^{0}\right) - [Z^\alpha, \Delta] p^{0}.
    \end{equation*}
    Then we have
    \begin{equation}\label{4.24}
        \begin{aligned}
            \|\nabla \left(Z^\alpha p^{0}\right)(t)\|^2 = & - \inp*{Z^\alpha \left(\u^{0}\cdot \nabla \u^{0}\right)}{\nabla \left(Z^\alpha p^{0}\right)(t)} \\
            & + \inp*{[Z^\alpha, \nabla \cdot] \left(\u^{0}\cdot \nabla \u^{0}\right)}{Z^\alpha p^{0}(t)} - \inp*{[Z^\alpha, \Delta] p^{0}}{Z^\alpha p^{0}(t)}.
        \end{aligned}
    \end{equation}
    Noting that whenever $\alpha_1 \ge 1$, ${Z^\alpha p^{E,0}}_{|x=0} = 0$, then by integration by parts, we have
    \begin{equation*}
        \begin{aligned}
            & \quad  \inp*{[Z^\alpha, \partial_x] \left(\u^{0} \cdot \nabla u^{0}\right)}{Z^\alpha p^{0}(t)}\\
            = & \sum_{0\le j\le \alpha_1-1} \inp*{\partial_x\left( c'_{\alpha_1,j}(x) (\chi(x)\partial_x)^{j} (\chi(y)\partial_y)^{\alpha_2} \left(\u^{0} \cdot \nabla u^{0}\right) \right)}{Z^\alpha p^{0}(t)}\\
            = & -\sum_{0\le j\le \alpha_1-1} \inp*{ c'_{\alpha_1,j}(x) (\chi(x)\partial_x)^{j} (\chi(y)\partial_y)^{\alpha_2} \left(\u^{0} \cdot \nabla u^{0}\right) }{\partial_x \left( Z^\alpha p^{0} \right) (t)}\\
            \lesssim & \|\u^{0} \cdot \nabla u^{0}(t)\|_{k-1,\co}\|\partial_x \left( Z^\alpha p^{0}\right) (t)\|,
        \end{aligned}
    \end{equation*}
    and
    \begin{equation*}
        \begin{aligned}
            & \quad - \inp*{[Z^\alpha, \partial_x^2] p^{E,0}}{Z^\alpha p^{E,0}(t)} = \inp*{\left( \partial_x[Z^\alpha, \partial_x] + [Z^\alpha, \partial_x]\partial_x \right) p^{E,0}}{Z^\alpha p^{E,0}(t)}\\
            = & - \sum_{0\le j\le \alpha_1 -1} \inp*{\partial_x \left( c''_{\alpha_1,j}(x) \partial_x (\chi(x)\partial_x)^{j} (\chi(y)\partial_y)^{\alpha_2} p^{E,0}\right)}{Z^\alpha p^{E,0}(t)}\\
            = & \sum_{0\le j\le \alpha_1 -1} \inp*{c''_{\alpha_1,j}(x) \partial_x \left(  (\chi(x)\partial_x)^{j} (\chi(y)\partial_y)^{\alpha_2} p^{E,0}\right)}{\partial_x \left( Z^\alpha p^{E,0}\right) (t)}\\
            \lesssim & \sum_{|\beta| \le k-1}\|\partial_x \left(Z^\beta p^{E,0}\right)(t)\| \|\partial_x \left(Z^\alpha p^{E,0}\right)(t)\|,
        \end{aligned}
    \end{equation*}
    where $c'_{\alpha_1,j}(x), c''_{\alpha_1,j}(x)$ are smooth functions.
    Thus, from \eqref{4.24} one gets
    \begin{equation*}
        \begin{aligned}
            \|\nabla \left(Z^\alpha p^{0}\right)(t)\|  \lesssim \|\u^{0} \cdot \nabla \u^{0}(t)\|_{k,\co} + \sum_{|\beta| \le k-1}\|\nabla \left(Z^\beta p^{0}\right)(t)\| .
        \end{aligned}
    \end{equation*}
    Taking induction on $0\le k\le m-1$, we have
    \begin{equation*}
        \begin{aligned}
            \|\nabla p^{0}(t)\|_{m-1,\co} \lesssim \|\u^{0} \cdot \nabla \u^{0}(t)\|_{m-1,\co}.
        \end{aligned}
    \end{equation*}
    Now by using Lemma \ref{lemma-1} and \eqref{Q4}, we obtain
    \begin{equation}\label{est-E,0-6}
        \begin{aligned}
            \|\nabla p^{0}(t)\|_{m-1,\co} \lesssim & \|\u^{0}(t)\|_{L^\infty(\Q)} \|\nabla \u^{0}(t)\|_{m-1,\co} + \|\u^{0}(t)\|_{m-1,\co} \|\nabla \u^{0}(t)\|_{L^\infty(\Q)}\\
            \lesssim & \|\u^{0}(t)\|_{W^{1,\infty}(\Q)} \left(\|\u^{0}(t)\|_{m-1,\co} + \|\omega^{0}(t)\|_{m-1,\co} \right) \\
            \lesssim & Q^{\frac{1}{2}}_{4}(\u^{0};t) Q^{\frac{1}{2}}_{m}(\u^{0};t).
        \end{aligned}
    \end{equation}


\noindent\underline{\textbf{$L^\infty$ estimates.}}
Combining \eqref{est-E,0-3}, \eqref{est-E,0-5} with \eqref{est-E,0-6}, we have
\begin{equation}\label{est-E,0-9}
    \begin{aligned}
        \frac{\mathrm{d}}{\mathrm{d} t} \left(\|\u^{0}(t)\|^2_{m,\co} + \|\omega^{0}(t)\|^2_{m-1,\co}\right) \lesssim & 
        (1+Q^{\frac{1}{2}}_4(\u^{0};t))^2  Q_m(\u^{0}; t).
    \end{aligned}
\end{equation}
It remains to estimate $\|\omega^{0}(t)\|^2_{W^{1,\infty}(\Q)}$ 
.
From the transport equation \eqref{E,0-vor}, we have
\begin{equation}\label{est-E,0-7}
    \begin{aligned}
        \|\omega^{0}(t)\|^2_{W^{1,\infty}(\Q)} \le & \|\omega_0\|^2_{W^{1,\infty}(\Q)} + t \int_0^t \left\|\nabla \u^{0} \cdot \nabla \omega^{0}(s)\right\|^2_{L^{\infty}(\Q)} \mathrm{d} s\\
        \le & \|\omega_0\|^2_{W^{1,\infty}(\Q)} + C t \int_0^t Q^2_4(\u^{0};t) \mathrm{d} s.
    \end{aligned}
\end{equation}


Finally, integrating \eqref{est-E,0-9} over $(0,t)$ and combining \eqref{est-E,0-7}, we get
\begin{equation}\label{est-E,0'}
    \begin{aligned}
        Q_m(\u^{0}; t) \le & Q_m(\u^{0}; 0) + C t\int_0^t Q_4^2(\u^{0}; s) \,\mathrm{d} s \\
        & + C\int_0^t \left(1+ Q^{\frac{1}{2}}_4(\u^{0}; s)\right)^2  Q_m(\u^{0}; s) \,\mathrm{d} s.
    \end{aligned}
\end{equation}
Then there exists a $T>0$ such that $\sup_{t\in [0,T]} Q_m(\u^{0}; t)$ is bounded, it finishes the proof of the estimate \eqref{est-E,0}.

\end{proof}

The next result shows that when the initial data has additional regularity away from the corner, then the solution given in Proposition \ref{prop-Euler} is also more regular away from the corner.
\begin{proposition}\label{prop-E,0-omega}
    Let $m$ be an integer satisfying $m\ge 4$, $\u_0 \in E^m$ be the initial data given in Proposition \ref{prop-Euler}, having additional regularity, $\nabla \times \u_0 \in H^{m-1}(\Q)$, then the velocity $\u^{0}$ of the solution to \eqref{limit} given in Proposition \ref{prop-Euler} belongs to $H^m(\Q)$.
\end{proposition}

\begin{proof}
    Without loss of regularity, we still set $\delta = 2$.
    We shall take induction to prove that $\nabla^k \u^{0}\in L^\infty(0,T;\Q \backslash \Q_1)$ for all $2\le k\le m$. 
By Lemma \ref{lemma-tan-vor}, it suffices to prove that $\nabla^{k} \omega^{0} \in L^\infty(0,T; \Sigma_1 \cup \Sigma_2)$ for all $1\le k\le m-1$.

    For $k=1$, we have in the distribution sense:
    \begin{equation}\label{reg-E,0-1}
        \partial_t (\nabla \omega^{0}) + \u^{0} \cdot \nabla (\nabla \omega^{0}) + \nabla \u^{0}\cdot \nabla \omega^{0} = 0.
    \end{equation}
    Here $\u^{0} \in L^\infty(0,T; W^{2,\infty}(\Q))$ having pointwise meaning is a known function.
    \eqref{reg-E,0-1} is a linear transport system on $\nabla\omega^{0} = (\partial_x \omega^{0}, \partial_y \omega^{0})$, and does not need any boundary condition. Therefore, the well-posedness of \eqref{reg-E,0-1} in $L^\infty(0,T; L^2(\Q))$ shows that $\omega^{0} \in L^\infty(0,T;H^1(\Q))$.
    

    For $2 \le k \le m-1$, assume that we already have $\omega^{0} \in L^\infty(0,T; H^{k-1}(\Q))$, then by Lemma \ref{lemma-tan-vor}, $\u^{0} \in L^\infty(0,T; H^{k}(\Q))$. In the distribution sense, one has
    \begin{equation}\label{reg-E,0-2}
        \partial_t (\nabla^{k} \omega^{0}) + \u^{0}\cdot \nabla (\nabla^{k} \omega^{0}) + k \nabla \u^{0} \cdot \nabla^k \omega^{0} = \sum_{2\le j\le k} c_{k,j} \nabla^j \u^{0} \cdot \nabla^{k+1-j} \omega^{0},
    \end{equation}
    where $c_{k,j}$ are constants.
    The left-hand side of \eqref{reg-E,0-2} is still a transport operator with $\u^{0} \in L^\infty(0,T;W^{2,\infty}(\Q))$, and
    \begin{equation*}
        \begin{aligned}
            &\quad \|(\text{RHS of \eqref{reg-E,0-2}})(t)\|_{L^2(\Q)}  \lesssim 
            \|\u^{0}(t)\|_{W^{2,\infty}(\Q)} \|\omega^{0}(t)\|_{H^{k-1}(\Q)} \\
            & \hspace{.5in}
            + \|\u^{0}(t)\|_{H^k(\Q)} \|\nabla \omega^{0}(t)\|_{L^\infty(\Q)} 
            + \sum_{3 \le j \le k-1} \|\nabla^{j} \u^{0}(t)\cdot \nabla^{k-j+1} \omega^{0}(t)\|_{L^2(\Q)}.
        \end{aligned}
    \end{equation*}
    The last summation is zero as $k\le 3$, and when $k\ge 4$, we use Lemmas \ref{lemma-tan-vor} and \ref{lemma-1} to obtain
    \begin{equation*}
        \begin{aligned}
            & \sum_{3 \le j \le k-1} \|\nabla^{j} \u^{0}(t)\cdot \nabla^{k-j+1} \omega^{0}(t)\|_{L^2(\Q)} \\
            & \qquad \lesssim \sum_{3 \le j \le k-1} \|(Z^{j} \u^{0} + \nabla^{j-1} \omega^{0})(t)\cdot \nabla^{k-j+1} \omega^{0}(t)\|_{L^2(\Q)} \\
            & \qquad \lesssim \sum_{3 \le j \le k-1} \Big(\|Z^j \u^{0}(t)\|_{H^1(\Q)} \|\nabla^{k-j+1}\omega^{0}(t)\|_{H^1(\Q)} \\
            & \qquad \quad + \|\nabla^{j-2} (\nabla \omega^{0}) \cdot \nabla^{k-j} (\nabla \omega^{0})(t)\|_{L^2(\Q)} \Big)\\
            & \qquad \lesssim \left( \|\u^{0}(t)\|_{H^1_{x,y}H^{k-1}_{\co}(\Q)} + \|\nabla \omega^{0}(t)\|_{L^\infty(\Q)} \right) \|\omega^{0}(t)\|_{H^{k-1}(\Q)}.
        \end{aligned}
    \end{equation*}
    Thus, we use \eqref{est-E,0-Hodge} again and obtain that
    \begin{equation*}
        \begin{aligned}
            &\quad \|(\text{RHS of \eqref{reg-E,0-2}})(t)\|_{L^2(\Q)}  \lesssim Q^{\frac{1}{2}}_m(\u^{0};t) \left(\|\u^{0}(t)\|_{H^k(\Q)} + \|\omega^{0}(t)\|_{H^{k-1}(\Q)}\right),
        \end{aligned}
    \end{equation*}
    which is bounded in $L^\infty(0,T)$. Therefore, the well-posedness of \eqref{reg-E,0-2} shows that $\omega^{0} \in L^\infty(0,T; H^{m-1}(\Q))$, and thus $\u^{0} \in L^\infty(0,T; H^m(\Q))$.

\end{proof}

Combining Propositions \ref{prop-Euler} and \ref{prop-E,0-omega} it finishes the proof of Theorem \ref{thm-Euler}.

At the end of this subsection, we list some trace properties that will be useful in the following subsections.
\begin{lemma}\label{lemma-symm}
For $m=4$, 
the velocity of the solution to \eqref{limit} given in Theorem \ref{prop-Euler} has the following properties: for $l = 1,2,3$,
    \begin{equation*}
        \widehat{\dy^l u^{0}} = \widehat{\dx^2 u^{0}} = \widehat{\dx \dy u^{0}} = 0, \quad
        \widehat{\dx^l v^{0}} = \widehat{\dy^2 v^{0}} = \widehat{\dx\dy v^{0}} = 0.
    \end{equation*}
\end{lemma}
\begin{proof}
    This is a direct conclusion from ${u^{0}}_{|x=0} = {v^{0}}_{|y=0} = 0$, $\nabla \cdot \u^{0} = 0$ and $\nabla \times \u^{0} = 0$ in $(0,\delta)^2$.
\end{proof}

\subsection{Boundary layer profiles along the edges}

From now on, we fix $m= 4$, and let $(\u^{E,0},p^{E,0})$ be the solution to \eqref{limit} with initial data $\u_0$, and $T>0$ be the time span of the existence given in Theorem \ref{thm-Euler}.
In this section, we are going to study the well-posedness and estimates of the boundary layer profiles $(u^{B,1,2}, v^{B,2,2})(t,x,Y)$ along the edge $\{y=0\}$ to the problem \eqref{B-1-2}. The  boundary layer profiles $(v^{B,1,1}, u^{B,2,1})(t,X,y)$ along the other edge $\{x=0\}$ to the problem \eqref{B-1-1} can be studied by an analogous procedure. 

First, we introduce the weighted anisotropic Sobolev space $H^{k,m,p}_{x, Y}$ for fixed integers $k, m$ and $p$ as follows 
\begin{equation*}
    H^{k,m,p}_{x,Y}(\Omega) = \{u \in L^2(\Omega): \langle Y\rangle^k \partial_x^l \partial_Y^j u \in L^2(\Omega), \quad \text{for }l\le m, j\le p\},
\end{equation*}
equipped with the norm
\begin{equation*}
    \|u\|_{H^{k,m,p}_{x,Y}(\Omega)} = \sum_{0\le j\le p} \|u\|_{k,m,j} = \sum_{0\le j \le p} \sum_{l \le m} \|\langle Y\rangle^k  \partial_{x}^{l} \partial_Y^j u\|_{L^2(\Omega)},
\end{equation*}
where $\Omega$ is either  $\R^2_+ = \{x\in \R, Y>0\}$ or $\Q_{x,Y}:=\{x,Y>0\}$, and $\langle Y \rangle = (1+Y^2)^{\frac{1}{2}}$.
Similarly, the space $H^{k,m}_Y(\R_+)$ is defined as the closure of $C^\infty_c(\overline{\R}_+)$ with respect to the norm
\begin{equation*}
    \|u\|_{H^{k,m}_Y(\R_+)} = \sum_{0\le j\le m} \|\langle Y\rangle^k \dY^j u\|_{L^2(\R_+)}.
\end{equation*}

\begin{proposition}\label{prop-Prandtl}
    For a given $u^{B,1}_0 \in \bigcap_{k\in \mathbb{N}} H^{k,2,0}_{x,Y}(\Q_{x,Y})$,  
    there exists a unique solution $(u^{B,1,2},v^{B,2,2})$  on $(0,T)$ to the problem \eqref{B-1-2} with the initial data $u^{B,1,2}|_{t=0}=u_0^{B,1}(x,Y)$, satisfying
    \begin{equation}\label{pro4.3-1}
        u^{B,1,2} \in L^\infty(0,T; H^{k, 2, 0}_{x,Y}(\Q_{x,Y}))\cap L^2(0,T; H^{k,2,1}_{x,Y}(\Q_{x,Y})),
    \end{equation}
    \begin{equation}\label{pro4.3-2}
        v^{B,2,2} \in L^\infty(0,T; H^{k,1,1}_{x,Y} (\Q_{x,Y})),
    \end{equation}
    for all $k \in \mathbb{N}$, and
    \begin{equation}\label{pro4.3-3}
    \int_0^\infty \dt u^{B,1,2} \, \d Y \in L^\infty(0,T; H^1_x(\R_+)),\quad \int_0^\infty \overline{u^{B,1,2}} \, \d Y = 0.
    \end{equation}
    If, in addition, $u^{B,1}_0 \in H^{0,1,2}_{x,Y}(\Q_{x,Y})$, then
\begin{equation}\label{pro4.3-Linfty}
\partial_Y u^{B,1,2}
\in L^\infty\big((0,T)\times\Q_{x,Y}\big).
\end{equation}
    Moreover, if additionally, $\overline{u^{B,1}_0} \in \bigcap_{k\in \mathbb{N}}H^{k,2}_{Y}(\R_+)$, then the trace of $u^{B,1,2}$ satisfies 
    \begin{equation}\label{pro4.3-4}
        \begin{aligned}
            & \overline{u^{B,1,2}} \in L^\infty(0,T; H^{k,2}_Y(\R_+)) \cap L^2(0,T; H^{k,3}_Y(\R_+)), \\
            & \dt \overline{u^{B,1,2}} \in L^\infty(0,T; H^{k,0}_Y(\R_+)) \cap L^2(0,T; H^{k,1}_Y(\R_+)), 
        \end{aligned}
    \end{equation}
for all $k \in \mathbb{N}$, with $\bar{f}$ denoting the trace of the related function on $\{x=0\}$.
\end{proposition}

\begin{proof}
(1) Since $\underline{u^{E,0}} = 0$ on $x=0$, one does not need any boundary condition on the boundary $\{x=0\}$ for solve $u^{B,1,2}$ from the problem \eqref{B-1-2}, so one can construct the solution $u^{B,1,2}$ to \eqref{B-1-2} by restricting to $\Q_{x,Y}$ the solution $\widetilde{u}^{B,1,2}$ to the extended problem defined on $\R^2_+=\{x\in \R, Y>0\}$ (whose well-posedness is classical, see \cite{xin1999zero, iftimie2011viscous} for example), by extending  $\underline{u^{E,0}}, \underline{\dy v^{E,0}}, \partial_x \underline{u^{E,0}}$ to be defined in whole $x\in \R$ in their corresponding Sobolev spaces. 

(2) Now, we  study a priori estimates \eqref{pro4.3-1}-\eqref{pro4.3-4}
for the smooth solution of \eqref{B-1-2}. 
First, since $\underline{\dy u^{E,0}} = \underline{\omega^{E,0}} = 0$ on $x \in (0,\delta)$, one has
$$
\begin{aligned}
    \left\|\left(2\alpha \underline{u^{E,0}} - \underline{\dy u^{E,0}}\right)(t)\right\|_{H^{\frac{5}{2}}_x(\R_+)} \le & 2|\alpha|\|\underline{u^{E,0}}(t)\|_{H^{\frac52}_x(\R_+)}+\|\underline{\dy u^{E,0}}(t)\|_{H^{\frac52}_x(\R_+)}\\
\lesssim & (1+|\alpha|)\left(\|u^{E,0}(t)\|_{H^3(\Q)}+\|u^{E,0}(t)\|_{H^4(\Q\backslash(0,\delta)^2)}\right),
\end{aligned}
$$
and is bounded in $L^\infty(0,T)$.

For $k\ge 1$, multiplying the first equation of \eqref{B-1-2} by $(1+Y^{2k})u^{B,1,2}$, we integrate in $x,Y$ to obtain
    \begin{equation*}
        \begin{aligned}
        & \quad \frac{\mathrm{d}}{\mathrm{d}t} \|u^{B,1,2}(t)\|^2_{k,0,0} + 2\|u^{B,1,2}(t)\|^2_{k,0,1} \\
        = & \iint_{\Q} \left(-(1+Y^{2k}) \partial_x\underline{u^{E,0}} + (1+(2k+1)Y^{2k}) \underline{\partial_y v^{E,0}}\right) |u^{B,1,2}|^2 \,\mathrm{d} x \mathrm{d} Y \\ 
        & - 2 \iint_{\Q} 2k Y^{2k-1} u^{B,1,2} \partial_Y u^{B,1,2}\,\mathrm{d} x \mathrm{d} Y \\
        & - 2 \int_{x>0} \left( 2\alpha \underline{u^{E,0}} - \underline{\partial_y u^{E,0}}\right) {u^{B,1,2}}_{|Y=0} \,\mathrm{d} x \\
        = & \iint_{\Q} \left(-(1+Y^{2k}) \partial_x\underline{u^{E,0}} + (1+(2k+1)Y^{2k}) \underline{\partial_y v^{E,0}}\right) |u^{B,1,2}|^2 \,\mathrm{d} x \mathrm{d} Y \\ 
        & + 2k(2k-1) \|u^{B,1,2}(t)\|^2_{k-1,0,0} + 2\iint_{\Q} \left( 2\alpha \underline{u^{E,0}} - \underline{\partial_y u^{E,0}}\right) \partial_Y u^{B,1,2} \,\mathrm{d} x \mathrm{d} Y\\
        \le & 2(2 k^2+1)\left(1+ \|\underline{\partial_y v^{E,0}}(t)\|_{L^\infty(\R_+)}\right)\|u^{B,1,2}(t)\|^2_{k,0,0} \\
        & + \|u^{B,1,2}(t)\|^2_{k,0,1} + 
        \|(1+Y^{2k})^{-\frac{1}{2}}\|_{L^2(\R_+)}
        \left\|\left( 2\alpha \underline{u^{E,0}} - \underline{\partial_y u^{E,0}}\right) (t)\right\|^2_{L^2(\R_+)}.
        \end{aligned}
    \end{equation*}
    Thus, by Gronwall's inequality, there exists a constant $C_k>0$, such that for any $t\in (0, T)$,
    \begin{equation}\label{4.31}
        \| u^{B,1,2}(t) \|^2_{k,0,0} + \int_0^t \|u^{B,1,2}(\tau)\|^2_{k,0,1} \,\mathrm{d} \tau \le C_k.
    \end{equation}

    To have the the estimate \eqref{pro4.3-1} for the derivatives in $x$, by
    applying $\partial_x^l$ on the first equation of \eqref{B-1-2} and multiplying it by $(1+Y^{2k}) \partial_x^l u$ for $l= 1,2$, we have
    \begin{equation}\label{est-P-1}
        \begin{aligned}
        &\frac{\mathrm{d}}{\mathrm{d}t} \| \partial_x^l u^{B,1,2}(t)\|^2_{k,0,0} + 2 \|\partial_x^l u^{B,1,2}(t)\|^2_{k,0,1} \\
        \le & \iint_{\Q} \left(-(1+Y^{2k}) \partial_x\underline{u^{E,0}} + (1+(2k+1)Y^{2k}) \underline{\partial_y v^{E,0}}\right) |\partial_x^l u^{B,1,2}(t)|^2 \,\mathrm{d} x \mathrm{d} Y \\ 
        & - 2 \iint_{\Q} (1+Y^{2k}) \left([\partial_x^l, \underline{u^{E,0}} \partial_x + Y \underline{\partial_y v^{E,0}}\partial_Y] \right) u^{B,1,2} \cdot \partial_x^l u^{B,1,2}\,\mathrm{d} x \mathrm{d} Y \\
        & - 2 \iint_{\Q} (1+Y^{2k}) [\partial_x^l, \partial_x \underline{u^{E,0}}] u^{B,1,2} \cdot \partial_x^l u^{B,1,2}\,\mathrm{d} x \mathrm{d} Y \\
        & - 2 \iint_{\Q} 2k Y^{2k-1} \cdot \partial_x^l u^{B,1,2}\cdot \partial_Y \partial_x^l u^{B,1,2} \, \mathrm{d}x \mathrm{d}Y \\
        & - 2 \int_{x>0} \partial_x^l \left(2\alpha \underline{u^{E,0}} - \underline{\dy u^{E,0}}\right) \cdot \left(\partial_x^l u^{B,1,2}\right)_{|Y=0} \,\mathrm{d} x \\
        =: & I_1 + \cdots + I_5,
    \end{aligned}
    \end{equation}
    which shall be estimated one by one. First, one has
    \begin{equation}\label{est-P-2}
        |I_1| \lesssim 2(k+1) \left(1+ \| v^{E,0}(t)\|_{W^{2,\infty}(\Q)}\right)\|u^{B,1,2}(t)\|^2_{k,l,0}.
    \end{equation}
    Using $(1+Y^{2k})Y \lesssim 1+ Y^{2k+2} \le (1+Y^{2k+4})^{\frac{1}{2}} (1+Y^{2k})^{\frac{1}{2}}$, one has
    \begin{equation}\label{est-P-3}
        \begin{aligned}
            |I_2 + I_3| \lesssim & \sum_{1\le j\le l} \left|\iint_{\Q} (1+Y^{2k}) \left(\partial_x^j \underline{u^{E,0}} \partial_x^{l+1-j} + Y \partial_x^j\underline{\partial_y v^{E,0}} \partial_x^{l-j}\partial_Y\right) u^{B,1,2} \cdot \partial_x^l u^{B,1,2} \,\mathrm{d}x \mathrm{d} Y\right|\\
            & + \sum_{1\le j\le l} \left|\iint_{\Q} (1+Y^{2k}) \partial_x^{j+1} \underline{u^{E,0}} \partial_x^{l-j} u^{B,1,2} \cdot \partial_x^l u^{B,1,2} \,\mathrm{d}x \mathrm{d} Y\right|\\
             \lesssim & \|\underline{u^{E,0}}(t)\|_{W^{l+1,\infty}(\R_+)} \|u^{B,1,2}(t)\|^2_{k,l,0}\\
            & + \|\underline{\partial_y v^{E,0}}(t)\|_{W^{l,\infty}(\R_+)} \|u^{B,1,2}(t)\|_{k+2, l-1,1} \|u^{B,1,2}(t)\|_{k,l,0}.
        \end{aligned}
    \end{equation}
and
    \begin{equation}\label{est-P-4}
        \begin{aligned}
            |I_4+I_5| = & 2k(2k-1) \|\partial_x^l u^{B,1,2}(t)\|^2_{k-1,0,0} + 2\iint_{\Q} \partial_x^l g \partial_Y\partial_x^l {u}^{B,1,2} \,\mathrm{d} x \mathrm{d}Y\\
            \lesssim & \|u^{B,1,2}(t)\|^2_{k-1,l,0} + \| u^{B,1,2}(t)\|_{k,l,1} 
            \|g(t)\|_{H^l(\R_+)},
        \end{aligned}
    \end{equation}
    with 
    $g=2\alpha \underline{u^{E,0}} - \underline{\dy u^{E,0}}$
    ,
    and we have used the fact that $\|(1+Y^{2k})^{-1/2}\|_{L^2(\R_+)} < +\infty$.
    Therefore, combining \eqref{est-P-1}-\eqref{est-P-4}, there exists a constant $C_k>0$ that depends only on $k$, $|\alpha|$, and the indicated norms of $\u^{E,0}$, such that
    \begin{equation*}
        \begin{aligned}
        \frac{\mathrm{d}}{\mathrm{d}t} \| u^{B,1,2}(t)\|^2_{k,l,0} + \| u^{B,1,2}(t)\|^2_{k,l,1} \le & C_k\Big(1+\|u^{B,1,2}(t)\|^2_{k,l,0} + \|u^{B,1,2}(t)\|^2_{k+2,l-1,1} \Big),
        \end{aligned}
    \end{equation*}
    which implies, by applying Gronwall's lemma, that
    \begin{equation*}
        \|u^{B,1,2}(t)\|_{k,l,0}^2 + \int_0^t \|u^{B,1,2}(\tau)\|^2_{k,l,1} \mathrm{d} \tau \le C_ke^{C_k t} \left(t+\int_0^t \|u^{B,1,2}(\tau)\|^2_{k+2,l-1,1} \mathrm{d} \tau\right).
    \end{equation*}
  Taking induction on $l = 1,2$, and using \eqref{4.31}, it concludes \eqref{pro4.3-1} for all $k\in \mathbb{N}$.

(3)
    We adapt the method of barrier functions for the uniform estimates of $$\|\partial_Y u^{B,1,2}\|_{L^\infty((0, T)\times \Q)}.$$ We set $\gamma >0$ a fixed constant, and
    $w(t,x,Y) = \partial_Y u^{B,1,2}(t,x,Y) e^{-\gamma t}$.
    Then we deduce that $w$ satisfies the equation
    \begin{equation}\label{barrier-eq}
        \dt |w|^2 + \left(\underline{u^{E,0}} \partial_x + Y \underline{\partial_y v^{E,0}} \partial_Y \right) |w|^2 - \partial_Y^2 |w|^2 + 2\gamma |w|^2 + 2|\partial_Y w|^2 = 0.
    \end{equation}
    Suppose that $|w|^2$ attains its maximum on $[0,T]\times \overline{\Q}$ in a point $(t_0,x_0,Y_0)$ that satisfies $t_0>0, Y_0>0$, then at this point one has that
    \begin{equation*}
        \partial_t |w|^2 \ge 0, \quad \partial_Y |w|^2 = 0, \quad \underline{u^{E,0}} \partial_x |w|^2 = 0, \quad \partial_Y^2 |w|^2 \le 0.
    \end{equation*}
    This reduces to
    \begin{equation*}
        2\gamma |w|^2 + 2 |\partial_Y w|^2 \le 0,
    \end{equation*}
    which is impossible unless $w \equiv 0$. The non-trivial boundary condition given in \eqref{B-1-2} shows that
    \begin{equation*}
        \|e^{-\gamma t} \partial_Y u^{B,1,2}\|_{L^\infty((0,T)\times \Q)} \le  \max\left\{
        \|\dY u^{B,1}_0\|_{L^\infty(\Q)}, \|e^{-\gamma t} \left(2\alpha\underline{u^{E,0}} - \underline{\dy u^{E,0}}\right)\|_{L^\infty((0,T)\times \R_+)} \right\},
    \end{equation*}
    and further
    \begin{equation*}
        \begin{aligned}
\|\partial_Y u^{B,1,2}\|_{L^\infty((0,T)\times \Q)} \lesssim \|u^{B,1}_0\|_{H^1_x H^2_Y(\Q)}+ \|\u^{E,0}\|_{L^\infty(0,T; W^{1,\infty}(\Q))} < + \infty.
\end{aligned}
    \end{equation*}

(4)    On the other hand, from \eqref{B-1-2}, one immediately has     \begin{equation*}
        v^{B,2,2}(t,x,Y) = \int_Y^{\infty} \partial_x u^{B,1,2}(t,x,Y') \, \mathrm{d} Y',
    \end{equation*}
    which implies
    \begin{equation*}
        \|v^{B,2,2}(t)\|_{k,1,1} = \|\dx u^{B,1,2}(t)\|_{k,1,0} = \|u^{B,1,2}(t)\|_{k,2,0}\in L^\infty(0,T),
    \end{equation*}
and
        \begin{equation*}
          \begin{aligned}
            \|v^{B,2,2}(t)\|_{k,1,0} = & \sum_{0\le j\le 1}\left(\iint_{x,Y>0} (1+Y^{2k}) \left|\int_Y^\infty \partial_x^{j+1} u^{B,1,2}(t,x,Y') \mathrm{d} Y'\right|^2 \, \mathrm{d} x\mathrm{d} Y\right)^{\frac{1}{2}} \\
            \le & 2 \sum_{0\le j\le 1} \left(\int_{x>0} \int_{Y>0} \left(Y(1+Y^{2k})^{\frac{1}{2}} \partial_x^{j+1} u^{B,1,2} (t,x,Y)\right)^2 \, \mathrm{d} Y \mathrm{d} x\right)^{\frac{1}{2}} \\
            \le & 2\|\partial_xu^{B,1,2}\|_{k+1, 1, 0}
=2\|u^{B,1,2}\|_{k+1, 2, 0}\in L^\infty (0,T)
        \end{aligned}
        \end{equation*}
by using \eqref{pro4.3-1}, thus we have the conclusion \eqref{pro4.3-2}.


(5)   To study \eqref{pro4.3-3}, by integrating \eqref{B-1-2} over $Y \in \R_+$, we have 
    \begin{equation*}
        \begin{aligned}
            \dt \left(\int_0^\infty u^{B,1,2} \, \d Y\right) & = \left(\underline{\dy u^{E,0}} - 2\alpha \underline{u^{E,0}}\right)
             + \left( 2 \underline{\dy v^{E,0}} -  \underline{u^{E,0}} \dx \right)\left(\int_0^\infty u^{B,1,2} \, \d Y\right),
        \end{aligned}
    \end{equation*}
   which implies that
    \begin{equation*}
        \begin{aligned}
            \left\|\dt \left(\int_0^\infty u^{B,1,2} \, \d Y\right)(t)\right\|_{H^1_{x}(\R_+)} \|\u^{E,0}(t)\|_{H^3_{x,y}(\Q)} \left(1+ \|u^{B,1,2}(t)\|_{H^{1,2,0}_{x,Y}(\Q_{x,Y})}\right)
        \end{aligned}
    \end{equation*}
    being bounded in $L^\infty(0,T)$.

    Taking the restriction of \eqref{B-1-2} on $\{x=0\}$, one gets
    \begin{equation}\label{u-b,1,2}
        \begin{cases}
            \dt \overline{u^{B,1,2}} + Y \widehat{\dy v^{E,0}} \dY \overline{u^{B,1,2}} + \widehat{\dx u^{E,0}} \overline{u^{B,1,2}} - \dY^2 \overline{u^{B,1,2}} = 0,\quad Y \in \R_+,\\
            \dY \overline{u^{B,1,2}}_{|Y = 0} = 0, \\
            \lim_{Y \to \infty} \overline{u^{B,1,2}} = 0, \\
            {\overline{u^{B,1,2}}}_{|t=0} = \overline{u^{B,1}_0}.
        \end{cases}
    \end{equation}
    Applying $\dY$ and $\dY^2$ on \eqref{u-b,1,2}, we obtain
    \begin{equation}\label{u-b,1,2-1}
        \dt (\dY \overline{u^{B,1,2}}) + Y \widehat{\dy v^{E,0}} \dY (\dY \overline{u^{B,1,2}}) - \dY^2 (\dY \overline{u^{B,1,2}}) = 0,
    \end{equation}
    and
    \begin{equation}\label{u-b,1,2-2}
        \dt (\dY^2 \overline{u^{B,1,2}}) + Y \widehat{\dy v^{E,0}} \dY (\dY^2 \overline{u^{B,1,2}}) + \widehat{\dy v^{E,0}} (\dY^2 \overline{u^{B,1,2}}) - \dY^2 (\dY^2 \overline{u^{B,1,2}}) = 0.
    \end{equation}
    Moreover, the restriction of \eqref{u-b,1,2-1} on $Y = 0$ implies that
    $\dY^3 \overline{u^{B,1,2}}_{|Y = 0} = 0$.

    Then, if for all $k \in \mathbb{N}$, $\langle Y \rangle^k \overline{u^{B,1}_0} \in H^2(\R_+)$, it is easy by similar energy estimates as above to obtain that for all $k\in \mathbb{N}$ and $l= 0,1,2$
    \begin{equation*}
        \|\langle Y\rangle^k \dY^l\overline{u^{B,1,2}}(t)\|^2 + \int_0^t \|\langle Y\rangle^k \dY^{l+1} \overline{u^{B,1,2}}(\tau)\|^2\le \|\langle Y\rangle^k \dY^l\overline{u^{B,1}_0}\|^2 + C_k\int_0^t \|\langle Y\rangle^k \dY^l\overline{u^{B,1,2}}(\tau)\|^2,
    \end{equation*}
    where $C_k$ is a constant depending only on $\|\widehat{\dy v^{E,0}}\|_{L^\infty(0,T)}$ and $k$. 
    Thus, it concludes
    $$
    \langle Y \rangle^k \overline{u^{B,1,2}} \in L^\infty(0,T;H^2_Y(\R_+))\cap L^2(0,T; H^3_Y(\R_+)), \quad \forall k \in \mathbb{N}.
    $$
   From the equation \eqref{u-b,1,2}, we see that
    \begin{equation*}
        \langle Y \rangle^k \dt \overline{u^{B,1,2}} \in L^\infty(0,T;L^2_Y(\R_+))\cap L^2(0,T; H^1_Y(\R_+)), \quad \forall k \in \mathbb{N}.
    \end{equation*}

    On the other hand, recall that the compatibility conditions from Remark \ref{rmk-1} of the initial data imply that
    \begin{equation*}
        \int_0^\infty  u^{B,1}_0(x, Y) \,\d Y \equiv C,
    \end{equation*}
    for some constant $C$. Under the assumption $u^{B,1}_0 \in \cap_{k \in \mathbb{N}} H^{k,2,0}_{x,Y}(\Q_{x,Y})$, we know that $$ \lim_{x\to \infty} \|u^{B,1}_0(x,\cdot)\|_{L^1_Y(\R_+)} = 0, $$ which implies that $C = 0$.
    Now we take the restriction of \eqref{B-1-2} on $x=0$ and integrate over $Y \in \R_+$ to get
    \begin{equation*}
        \begin{cases}
                \partial_t \left(\int_0^\infty\overline{u^{B,1,2}} \,\d Y\right) + 2\widehat{\partial_x u^{E,0}} \left(\int_0^\infty\overline{u^{B,1,2}} \,\d Y\right) = 0, \quad Y >0,\\
                \left(\int_0^\infty\overline{u^{B,1,2}} \,\d Y\right)_{|t=0} = \int_0^\infty u^{B,1}_0 (0,Y) \, \d Y = 0.
            \end{cases}
    \end{equation*}
    This implies $\int_0^\infty\overline{u^{B,1,2}} \,\d Y = 0$.

\end{proof}

\subsection{The outer flow profile $\u^{E,2}$}

From Conclusion \ref{concl} in \S2, we know that the outer flow profile $\u^{E,2}(t,x,y)$ satisfies
\begin{equation}\label{E,2}
        \begin{cases}
        \partial_t \u^{E,2} + (\u^{E,0} \cdot \nabla) \u^{E,2} + (\u^{E,2}\cdot \nabla) \u^{E,0} + \nabla p^{E,2} = \Delta \u^{E,0},\quad x,y>0\\
        \nabla \cdot \u^{E,2} = 0,\\
        {u^{E,2}}_{|x=0} = - {u^{B,2,1}}(t,0,y), \quad {v^{E,2}}_{|y=0} = - {v^{B,2,2}}(t,x,0),\\
        {\u^{E,2}}_{|t=0} = \u^{E,2}_0(x,y),
        \end{cases}
\end{equation}
where $v^{B,2,2}(t,x,Y)$ is the boundary layer profile given in Proposition \ref{prop-Prandtl}, and $u^{B,2,1}(t,X,y)$ is the boundary layer profile near $\{x=0\}$. For this problem, one has:

\begin{proposition}\label{prop-E,2}
    Under the assumptions given in Theorem \ref{thm-Euler} and $\u^{E,2}_0\in H^{1}(\Q)$,
    there exists a unique solution $(\u^{E,2},p^{E,2})$ to the problem \eqref{E,2}, satisfying
    \begin{equation*}
        (\u^{E,2}, \nabla p^{E,2}) \in L^\infty(0,T;H^{1}(\Q)) \times L^\infty(0,T;L^2(\Q)).
    \end{equation*}
\end{proposition}
\begin{proof}
    We first homogenize the boundary data given in \eqref{E,2} by constructing an auxiliary function $\mathbf{z}^{E,2} = (z^{E,2}_1, z^{E,2}_2)$ such that
    \begin{equation*}
        \nabla\cdot \mathbf{z}^{E,2} = \nabla \times \mathbf{z}^{E,2} = 0, \quad ({z^{E,2}_1})_{|x=0} = - \overline{u^{B,2,1}}, \, ({z^{E,2}_2})_{|y=0} = - \underline{v^{B,2,2}}.
    \end{equation*}
    Indeed, $\mathbf{z}^{E,2}$ equals $(I-\mathbb{P}) \u^{E,2}$ with $\mathbb{P}$ being the Leray projection.
    It suffices to take $\mathbf{z}^{E,2} = \nabla \phi$, where $\phi$ solves
    \begin{equation}\label{LN-E,2}
        \Delta \phi = 0, \quad \text{in }\Q,\quad (\dx \phi)_{|x=0} = - \overline{u^{B,2,1}}, \quad (\dy \phi)_{|y=0} = - \underline{v^{B,2,2}}.
    \end{equation}
    Note that $\int_{\R_+} \overline{u^{B,2,1}} \d y = \int_{\R_+} \underline{v^{B,2,2}} \d x = 0$, since, for example,
    \begin{equation*}
        \int_{\R_+} \underline{v^{B,2,2}} \d x = \int_{\R_+} \int_{\R_+} \partial_x u^{B,1,2} \d Y \d x = - \int_{\R_+} \overline{u^{B,1,2}} \d Y = 0,
    \end{equation*}
    by using \eqref{pro4.3-3}.
    Hence, by Lemma \ref{lemma-LN-refl} with $l=1$, there exists $\phi\in L^\infty(0,T;H^2(\Q)/\R)$ such that
    \begin{equation*}
        \begin{aligned}
            \|\mathbf{z}^{E,2}(t)\|_{H^1(\Q)}  = & \|\nabla \phi(t)\|_{H^1(\Q)} \lesssim \|\overline{u^{B,2,1}}(t)\|_{H^{\frac{1}{2}}(\R_+)} + \|\underline{v^{B,2,2}}(t)\|_{H^{\frac{1}{2}}(\R_+)} \\
           & + \left\|\int_0^y \overline{u^{B,2,1}}(t,s) \d s\right\|_{L^2(\R_+)} + \left\|\int_0^x \underline{v^{B,2,2}}(t,s) \d s\right\|_{L^2(\R_+)}.
        \end{aligned}
    \end{equation*}
    Since
    \begin{equation*}
        \begin{aligned}
            \int_0^x \underline{v^{B,2,2}}(t,s) \d s = & \int_0^x \int_0^\infty \partial_s u^{B,1,2}(t,s,Y) \d Y \d s\\ = & \int_0^\infty \left(u^{B,1,2}(t,x,Y)-\overline{u^{B,1,2}}(t,Y)\right) \d Y\\ = & \left(\int_Y^\infty u^{B,1,2}(t,x,Y') \d Y'\right)_{|Y=0},
        \end{aligned}
    \end{equation*}
    where we used \eqref{pro4.3-3} in the third equality, and by a trace theorem, one has
    \begin{equation*}
        \begin{aligned}
            \left\|\int_0^x \underline{v^{B,2,2}}(t,s) \d s\right\|_{L^2_x(\R_+)} \lesssim & \left\|\int_Y^\infty u^{B,1,2}(t,x,Y') \d Y'\right\|_{H^{1}_Y L^2_{x}(\Q_{x,Y})}
            \\
            \lesssim & \|u^{B,1,2}(t)\|_{H^{1,1,0}_{x,Y}(\Q_{x,Y})},
        \end{aligned}
    \end{equation*}
    where we have used the fact that $\dY \left(\int_Y^\infty u^{B,1,2} \d Y'\right) = u^{B,1,2}$. Therefore, by a similar analysis on $\int_0^y \overline{u^{B,2,1}}(t,s) \d s$, we obtain $\mathbf{z}^{E,2} \in L^\infty(0,T; H^1(\Q))$.

    Moreover, from Lemma \ref{lemma-LN-refl} with $l=0$, we see that
    \begin{equation*}
        \begin{aligned}
            \|\dt \mathbf{z}^{E,2}(t)\|_{L^2(\Q)} \lesssim & \left\|\int_0^y \dt\overline{u^{B,2,1}}(t,s) \d s\right\|_{H^1(\R_+)} + \left\|\int_0^x \dt\underline{v^{B,2,2}}(t,s) \d s\right\|_{H^1(\R_+)}\\
            \lesssim & \left\|\int_0^\infty \dt v^{B,1,1}(t,X,y) \d X\right\|_{H^1(\R_+)} + \left\|\int_0^\infty \dt u^{B,1,2}(t,x,Y) \d Y\right\|_{H^1(\R_+)},
        \end{aligned}     
    \end{equation*}
    which implies that $\dt \mathbf{z}^{E,2} \in L^\infty(0, T; L^2(\Q))$ by using Proposition \ref{prop-Prandtl}. 
    
    On the other hand, since
    \begin{equation*}
        \overline{u^{B,2,1}}_{|t=0} = \overline{u^{B,2}_0} = 0,\quad   \underline{v^{B,2,2}}_{|t=0} = \underline{v^{B,2}_0} = 0,
    \end{equation*}
    we find by the uniqueness (modulo constants) of the solution to \eqref{LN-E,2} that ${\mathbf{z}^{E,2}}_{|t=0} = 0$.

    Let $\u^{E,2} = \mathbf{z}^{E,2} + \mathbf{v}^{E,2}$, then from \eqref{E,2} we know that ${\mathbf{v}}^{E,2} = (v^{E,2}_1,v^{E,2}_2)$ satisfies the following problem:
    \begin{equation}\label{li-Euler'}
\begin{cases}
\begin{aligned}
    \partial_t \mathbf{v}^{E,2}& + (\u^{E,0} \cdot \nabla) \mathbf{v}^{E,2}+ (\mathbf{v}^{E,2}\cdot \nabla) \u^{E,0} + \nabla p^{E,2} \\
    & = \Delta \u^{E,0} - \left(\dt \mathbf{z}^{E,2} + (\u^{E,0} \cdot \nabla) \mathbf{z}^{E,2} + (\mathbf{z}^{E,2} \cdot \nabla) \u^{E,0}\right),\quad \text{in }\Q,
\end{aligned}
\\
\nabla \cdot \mathbf{v}^{E,2}= 0, \quad \text{in }\Q,\\
{\left(\mathbf{v}^{E,2}\cdot \vec{n}\right)}_{|\partial\Q\setminus\{(0,0)\}} = 0 \\
{\mathbf{v}^{E,2}}_{|t=0} = \u^{E,2}_0(x,y),
\end{cases}    
\end{equation}
By the regularity of $\u^{E,0}$ and $\mathbf{z}^{E,2}$, the right-hand side of the equation is bounded in $L^\infty(0,T; L^2(\Q))$.
It is easy to see that \eqref{li-Euler'} admits a weak solution $\mathbf{v}^{E,2} \in L^\infty(0,T; L^2(\Q))$ via an even-odd extension argument.
Moreover, $\omega_{v}^{E,2} := \nabla \times \mathbf{v}^{E,2}$ satisfies
\begin{equation}\label{vor-E,2}
    \dt \omega_v^{E,2} + (\u^{E,0} \cdot \nabla ) \omega_v^{E,2} = \Delta \omega^{E,0} - (\mathbf{u}^{E,2} \cdot \nabla) \omega^{E,0}.
\end{equation}
Obviously,
\begin{equation*}
    \begin{aligned}
        \|(\mathbf{u}^{E,2} \cdot \nabla) \omega^{E,0} (t)\|_{L^2(\Q)} \lesssim  \|\u^{E,2}(t)\|_{L^2(\Q)} \|\u^{E,0}(t)\|_{H^4(\Q)},
    \end{aligned}
\end{equation*}
so the right-hand side of \eqref{vor-E,2} belongs to $L^\infty(0,T; L^2(\Q))$. 
By standard energy argument for the transport equation \eqref{vor-E,2}, one has $\omega^{E,2}_v \in L^\infty(0,T; L^2(\Q))$ providing that $\u^{E,2}_0\in H^{1}(\Q)$.
Therefore, by the Hodge estimate \eqref{Hodge}, we conclude that $\u^{E,2} = \mathbf{v}^{E,2} + \mathbf{z}^{E,2} \in L^\infty(0,T; H^1(\Q))$. The uniqueness of the solution $\u^{E,2}$ can be easily obtained by a classical argument for the problem \eqref{E,2}.

Applying $\nabla \cdot$ to both sides of \eqref{E,2}, one has
\begin{equation*}
    \begin{cases}
        \Delta p^{E,2} = \nabla \cdot F,\quad \text{in }\Q,\\
       {\partial_{\vec{n}} p^{E,2}}_{|\partial\Q} = F \cdot \vec{n} 
       + \left(\dt \overline{u^{B,2,1}}, \dt \underline{v^{B,2,2}} \right)\cdot \vec{n}
       ,
    \end{cases}
\end{equation*}
where
\begin{equation*}
    F = \Delta\u^{E,0} -(\u^{E,2}\cdot \nabla \u^{E,0} + \u^{E,0}\cdot \nabla \u^{E,2}).
\end{equation*}
By standard energy estimates and Lemma \ref{lemma-LN-refl}, one has
\begin{equation}\label{p-E,2}
    \begin{aligned}
        \|\nabla p^{E,2}(t)\|_{L^2} \lesssim & \|F(t)\|_{L^2} + \|\int_0^y\dt \overline{u^{B,2,1}}(t,s) \d s\|_{H^1(\R_+)} + \|\int_0^x \dt \underline{v^{B,2,2}}(t,s) \d s\|_{H^1(\R_+)}\\
        \lesssim & \|\Delta\u^{E,0}  -(\u^{E,2}\cdot \nabla \u^{E,0} + \u^{E,0}\cdot \nabla \u^{E,2} )(t)\|_{L^2}\\
        & + \left\|\int_0^\infty \dt v^{B,1,1}(t,X,y) \d X\right\|_{H^1(\R_+)} + \left\|\int_0^\infty \dt u^{B,1,2}(t,x,Y) \d Y\right\|_{H^1(\R_+)}.
    \end{aligned}
\end{equation}
Therefore, by \eqref{pro4.3-3}, $\nabla p^{E,2} \in L^\infty(0,T; L^2(\Q))$.
\end{proof}

\section{Well-posedness of the corner layer problem}
\setcounter{proposition}{0} 
In this section, we study the corner layer problem \eqref{u-C,1}.
First, introduce some weighted Sobolev spaces as follows.

\begin{defn}\label{defn4-1} (1) For any fixed integers $k,m\ge 0$, denote by 
\begin{equation*}
    H^{k,m}_{X,Y}(\Q) = \{u \in L^2(\Q): \langle (X,Y)\rangle^k \partial_{X,Y}^\alpha u \in L^2(\Q), \quad \text{for }|\alpha|\le m\},
\end{equation*}
with the norm
\begin{equation*}
    \|u\|_{H^{k,m}_{X,Y}(\Q)} = \sum_{0\le j\le m}\|u\|_{k,j} = \sum_{0\le j\le m} \sum_{|\alpha| = j} \|\langle (X,Y)\rangle^k \partial_{X,Y}^\alpha u\|_{L^2(\Q)},
\end{equation*}
where $\langle(X,Y)\rangle = (1+X^2 + Y^2)^{\frac{1}{2}}$.

(2)
For $l \in \R_+$, $\beta \in \R$, denote by $V^l_{2,\beta}(\Q)$ the closure of $C^\infty_c(\overline{\Q}\backslash\{0\})$ with respect to the norm
\begin{equation}\label{norm-xi}
    \|u\|_{V^l_{2,\beta}(\Q)} = \left\|e^{(\beta-l+1)\xi} \mathcal{E} u \right\|_{H^l(\mathcal{C})},
\end{equation}
where $\mathcal{E}$ is a coordinate transform such that $(\mathcal{E} u)(\xi,\theta) = u(e^\xi \cos \theta, e^\xi \sin \theta)$, with $\mathcal{C} = \{\xi \in \R, \theta \in (0,\frac{\pi}{2})\}$.
Similarly, $V^l_{2,\beta}(\R^+_r)$ is the closure of $C^\infty_c(\R^+)$ with respect to the norm
\begin{equation*}
    \|u\|_{V^l_{2,\beta}(\R^+_r)} = \left\|e^{(\beta-l+\frac{1}{2})\xi} \mathcal{E} u \right\|_{H^l(\R_\xi)},
\end{equation*}
where $(\mathcal{E}u)(\xi) = u(e^{\xi})$.
\end{defn}

It is easy to see that
\begin{equation}\label{V-H}
    V^m_{2,m+k}(\Q) \cap H^m(\Q) \hookrightarrow H^{k,m}_{X,Y}(\Q).
\end{equation}

\begin{proposition}\label{prop-C,1}
Let $u^{B,1,2}, v^{B,1,1}$ be the boundary layer profiles given in Proposition \ref{prop-Prandtl}. For any given $\u^C_0 \in \cap_{k\in \mathbb{N}}H^{k,1}_{X,Y}(\Q)$, there exists a unique solution $\u^{C,1}$ to the problem \eqref{u-C,1} with the initial data
$\u^{C,1}(0, X, Y)=\u^C_0(X,Y)$, satisfying 
    $$\u^{C,1}\in L^\infty(0,T; H^{k,1}_{X,Y}(\Q)), \quad  \nabla p^{C,2}\in  L^\infty(0,T;L^2_{X,Y}(\Q))$$
    for all $k \in \mathbb{N}$.
\end{proposition}

\begin{proof}
We shall still focus on the a priori estimates for \eqref{u-C,1} in $L^\infty(0, T; H^{k,1}_{X, Y}(\Q))$ for smooth $\u^{C,1}$. By using a standard Galerkin procedure and these estimates, one can deduce the existence of a solution. 

1) \underline{\textbf{Estimates of $\|(I-\mathbb{P})\u^{C,1}\|_{H^1}$.}}
To homogenize the boundary conditions of $\u^{C,1}$ given in \eqref{u-C,1}, we first study the equations
satisfied by $\mathbf{z}^{C,1} = (I-\mathbb{P})\u^{C,1}$, with $\mathbb{P}$ being the Leray projection,
\begin{equation}\label{z-C,1}
    \nabla \cdot \mathbf{z}^{C,1} = \nabla \times \mathbf{z}^{C,1} = 0, \quad (z^{C,1}_1)_{|X=0} = - \overline{u^{B,1,2}}, \quad (z^{C,1}_2)_{|Y=0} = - \underline{v^{B,1,1}}.
\end{equation}
It suffices to take $\mathbf{z}^{C,1} = \nabla \phi$, where $\phi$ solves
\begin{equation}\label{LN-C,1}
    \Delta \phi = 0  \quad \text{in }\Q, \quad (\dX \phi)_{|X=0} = - \overline{u^{B,1,2}}, \quad (\dY \phi)_{|Y=0} = - \underline{v^{B,1,1}}.
\end{equation}
By Proposition \ref{prop-Prandtl}, we know that $\int_{\R_+} \overline{u^{B,1,2}} \d Y = \int_{\R_+}\underline{v^{B,1,1}} \d X = 0$.
Hence, by Lemma \ref{lemma-LN-refl}, there exists $\phi \in L^\infty(0,T; H^2(\Q)/\R)$ such that
\begin{equation}\label{5.4}
    \begin{aligned}
        \|\mathbf{z}^{C,1}(t)\|_{H^1(\Q)}  = &\|\nabla \phi(t)\|_{H^1(\Q)}\lesssim \|\overline{u^{B,1,2}}(t)\|_{H^{\frac{1}{2}}(\R_+)} + \|\underline{v^{B,1,1}}(t)\|_{H^{\frac{1}{2}}(\R_+)} \\
        & + \left\|\int_0^Y \overline{u^{B,1,2}}(t,s) \d s\right\|_{L^2(\R_+)} + \left\|\int_0^X \underline{v^{B,1,1}}(t,s) \d s\right\|_{L^2(\R_+)},
    \end{aligned}
\end{equation}
which implies $\mathbf{z}^{C,1}\in L^\infty(0,T;H^1(\Q))$ by using the boundedness given in Proposition \ref{prop-Prandtl}.  Similarly, one has
\begin{equation}\label{5.5}
    \begin{aligned}
        \|\dt \mathbf{z}^{C,1}(t)\|_{L^2(\Q)} \lesssim & \left\|\int_0^Y \dt\overline{u^{B,1,2}}(t,s) \d s\right\|_{L^2(\R_+)} + \left\|\int_0^X \dt\underline{v^{B,1,1}}(t,s) \d s\right\|_{L^2(\R_+)} \\
    \lesssim & \|\dt \overline{u^{B,1,2}}(t)\|_{H^{1,0}_Y(\R_+)} + \|\dt \underline{v^{B,1,1}}(t)\|_{H^{1,0}_X(\R_+)},
    \end{aligned}
\end{equation}
which gives rise to $\dt \mathbf{z}^{C,1} \in L^\infty(0,T;L^2(\Q))$ by using Proposition \ref{prop-Prandtl} again.

On the other hand, one can write $\mathbf{z}^{C,1} = \nabla^{\perp}\psi$, with $\psi$ solving
\begin{equation}\label{LD-C,1}
    \Delta \psi = 0 \quad\text{in } \Q, \quad \psi_{|X=0} = -\int_0^Y \overline{u^{B,1,2}}(t,s) \d s, \quad \psi_{|Y=0} = - \int_0^X \underline{v^{B,1,1}}(t,s) \d s.
\end{equation}
Then, by Lemma \ref{lemma-kozlov'},
\begin{equation*}
    \begin{aligned}
        \|\mathbf{z}^{C,1}(t)\|_{V^1_{2,2}} \le \|\psi(t)\|_{V^{2}_{2,2}} \lesssim & \left\|\int_0^Y \overline{u^{B,1,2}}(t,s) \d s\right\|_{V^{\frac{3}{2}}_{2,2}} + \left\|\int_0^X \underline{v^{B,1,1}}(t,s) \d s\right\|_{V^{\frac{3}{2}}_{2,2}}\\
        \lesssim & \|\langle Y\rangle^2 \overline{u^{B,1,2}}(t)\|_{H^{\frac{1}{2}}_Y(\R_+)} + \|\langle X\rangle^2 \underline{v^{B,1,1}}(t)\|_{H^{\frac{1}{2}}_X(\R_+)}.
    \end{aligned}
\end{equation*}
Furthermore,
\begin{equation}
    \|(|X|+|Y|)\nabla \mathbf{z}^{C,1}\|_{L^\infty(0,T; L^2(\Q))} \lesssim \|\mathbf{z}^{C,1}\|_{L^\infty(0,T; V^1_{2,2})}.
\end{equation}

    2) \underline{\textbf{Estimates of $\|\mathbb{P}\u^{C,1}\|_{L^2}$.}}
    Multiplying the equation of $\u^{C,1}$ given in \eqref{u-C,1} by $\mathbb{P}\u^{C,1} = \u^{C,1} - \mathbf{z}^{C,1}$ and integrating in space and time, we get by integration by parts that
    \begin{equation}\label{5.8}
        \begin{aligned}
            & \quad \frac{1}{2} \left\|\mathbb{P} \u^{C,1}(t)\right\|^2 - \int_0^t \inp*{\Delta \u^{C,1}}{\mathbb{P}\u^{C,1} (\tau)} \\
            \le & \frac{1}{2} \left\|\mathbb{P} \u^{C}_0\right\|^2 + 
            \left|\int_0^t \inp*{(\u^{G,1} \cdot \nabla) (I-\mathbb{P})\u^{C,1}}{\mathbb{P}\u^{C,1}(\tau)}\right| + \left| \int_0^t \inp*{M\u^{C,1} + B}{\mathbb{P}\u^{C,1}(\tau)}\right|,
        \end{aligned}
    \end{equation}
    where we have used the facts that $\nabla \cdot \u^{G,1} = 0$ and $\u^{G,1}\cdot \vec{n}_{|\partial\Q} = 0$.
    
    Recalling that ${\dY u^{C,1}}_{|Y=0} = {\dX v^{C,1}}_{|X=0} = 0$, one has
    \begin{equation*}
        \begin{aligned}
            - \int_0^t \inp*{\Delta \u^{C,1}}{\mathbb{P}\u^{C,1} (\tau)} = & \int_0^t \inp*{\nabla \u^{C,1}}{\nabla (\mathbb{P} \u^{C,1})(\tau)}\\
            = & \int_0^t \|\nabla \u^{C,1}(\tau)\|^2 - \int_0^t \inp*{\nabla \u^{C,1}}{\nabla \mathbf{z}^{C,1}(\tau)}.
        \end{aligned}
    \end{equation*}
    
    By Hölder's inequality, the convection term satisfies
    \begin{equation*}
        \begin{aligned}
            \left|\int_0^t \inp*{(\u^{G,1} \cdot \nabla) \mathbf{z}^{C,1}}{\mathbb{P}\u^{C,1}(\tau)}\right| \lesssim & \| \mathbf{z}^{C,1}\|_{L^\infty(0,T;H^1)} \int_0^t \|\u^{C,1}(\tau)\|_{L^4(\Q)} \cdot \|\mathbb{P}\u^{C,1}(\tau)\|_{L^4(\Q)} \,\d \tau \\
            & + \int_0^t \left\|(\u^H \cdot \nabla \mathbf{z}^{C,1})(\tau)\right\|^2\,\d \tau + \int_0^t \|\mathbb{P}\u^{C,1}(\tau)\|^2 \, \d \tau,
        \end{aligned}
    \end{equation*}
    where $\u^H = \left(\overline{u^{B,1,2}} + X\widehat{\dx u^{E,0}}, \underline{v^{B,1,1}} + Y \widehat{\dy v^{E,0}}\right)$.
    By Gagliardo-Nirenberg's inequality and Cauchy inequality, one has
    \begin{equation*}
        \begin{aligned}
            & \quad \int_0^t \|\u^{C,1}(\tau)\|_{L^4(\Q)} \cdot \|\mathbb{P}\u^{C,1}(\tau)\|_{L^4(\Q)} \,\d \tau\\
            \le & C \left(\int_0^t \|\u^{C,1}(\tau)\|_{L^2(\Q)} \|\nabla \u^{C,1}(\tau)\|_{L^2(\Q)}\, \d\tau 
            + \int_0^t \|\mathbb{P}\u^{C,1}(\tau)\|_{L^2(\Q)} \|\nabla \left(\mathbb{P}\u^{C,1}\right)(\tau)\|_{L^2(\Q)} \,\d \tau\right) \\
            \le & C\left(\frac{1}{\eta}\int_0^t \|\mathbb{P}\u^{C,1}(\tau)\|^2_{L^2(\Q)} \d \tau + \eta \int_0^t \|\nabla \u^{C,1}(\tau)\|^2_{L^2(\Q)} \d \tau 
            + \int_0^t \|\mathbf{z}^{C,1}(\tau)\|^2_{H^1} \d \tau \right),
        \end{aligned}
    \end{equation*}
    with a small constant $\eta$, and
    \begin{equation*}
        \begin{aligned}
            \left\|(\u^H \cdot \nabla \mathbf{z}^{C,1})\right\|_{L^2(0,T;L^2(\Q))} \lesssim & \|\overline{u^{B,1,2}},\underline{v^{B,1,1}}\|_{L^2(0,T;H^1(\R_+))} \|\mathbf{z}^{C,1}\|_{L^\infty(0,T;H^1(\Q))}\\
        & + \|\widehat{\dx u^{E,0}}\|_{L^2(0,T)} \|\langle(X,Y)\rangle \nabla \mathbf{z}^{C,1}\|_{L^\infty(0,T; L^2(\Q))} \le  C_1
        \end{aligned}
    \end{equation*}
    for a positive constant $C_1$.
    In \eqref{5.8}, the lower-order terms satisfy
    \begin{equation*}
        \begin{aligned}
            & \quad \int_0^t \inp*{M\u^{C,1} + B}{\mathbb{P}\u^{C,1}(\tau)} \lesssim \|\widehat{\dx u^{E,0}}\|_{L^\infty(0,T)} \int_0^t |\u^{C,1}(\tau)| \cdot |\mathbb{P}\u^{C,1}(\tau)|  \\
            & + \int_0^t \inp*{\dY \overline{u^{B,1,2}} (v^{C,1} + \underline{v^{B,1,1}})}{ (\mathbb{P}\u^{C,1})_1}(\tau) + \inp*{\dX \underline{v^{B,1,1}} (u^{C,1} + \overline{u^{B,1,2}})}{ (\mathbb{P}\u^{C,1})_2}(\tau),
        \end{aligned}
    \end{equation*}
   in which, by integration by parts, the last term equals
    \begin{equation*}
        \begin{aligned}
            - \int_0^t  \inp*{\overline{u^{B,1,2}}} {\dY\left( (v^{C,1} + \underline{v^{B,1,1}}) (\mathbb{P}\u^{C,1})_1 \right)}(\tau) 
            - \int_0^t  \inp*{\underline{v^{B,1,1}}}{ \dX\left( (u^{C,1} + \overline{u^{B,1,2}}) (\mathbb{P}\u^{C,1})_2\right) }(\tau).
        \end{aligned}
    \end{equation*}
    Among these terms, one has
    \begin{equation*}
        \begin{aligned}
            \int_0^t  \left|\inp*{\overline{u^{B,1,2}} \dY v^{C,1}}{(\mathbb{P}\u^{C,1})_1} (\tau) \right| \le & \frac{\eta}{2} \int_0^t \|\dY v^{C,1}(\tau)\|^2 \\
            & + \frac{1}{2\eta} \|\overline{u^{B,1,2}}\|^2_{L^\infty(0,T;H^1_Y)}  \int_0^t\left\|(\mathbb{P}\u^{C,1})_1 (\tau) \right\|^2
        \end{aligned}        
    \end{equation*}
   and
   \begin{equation*}
        \begin{aligned}
            & \int_0^t  \left|\inp*{\overline{u^{B,1,2}} (v^{C,1} +  \underline{v^{B,1,1}})}{ \dY(\mathbb{P}\u^{C,1})_1} (\tau) \right| \le \frac{\eta}{2} \int_0^t \|\dY (\mathbb{P}\u^{C,1})_1 (\tau)\|^2 \\
            & \hspace{.3in}+ \frac{1}{2\eta} \|\overline{u^{B,1,2}}\|^2_{L^\infty(0,T;H^1_Y)}  \left(\int_0^t\left\|v^{C,1} (\tau) \right\|^2 + \|\underline{v^{B,1,1}}\|^2_{L^2(0,T;L^2_X)}\right)
        \end{aligned}        
    \end{equation*}
    for a small constant $\eta$. Thus,
    \begin{equation*}
        \int_0^t \inp*{M\u^{C,1} + B}{\mathbb{P}\u^{C,1}(\tau)} \le C_2 \left(1+ \eta \int_0^t \|\nabla \u^{C,1}(\tau)\|^2 + \frac{1}{\eta}\int_0^t \|\mathbb{P}\u^{C,1}(\tau)\|^2 \right)
    \end{equation*}
    for some constant $C_2$.

    Therefore, by substituting the above estimates into \eqref{5.8}, one has
    \begin{equation*}
        \begin{aligned}
            \left\|\mathbb{P} \u^{C,1}(t)\right\|^2 + \int_0^t \|\nabla \u^{C,1}(\tau)\|^2 
            \le \left\|\mathbb{P} \u^{C}_0\right\|^2 + C\left( 1 + \int_0^t \|\mathbb{P}\u^{C,1}(\tau)\|^2 \right),
        \end{aligned}
    \end{equation*}
    with a constant $C$ depending on $\u^{E,0},u^{B,1,2},v^{B,1,1}, \mathbf{z}^{C,1}$.
    By Gronwall's inequality, we conclude
    \begin{equation}\label{5.9}
        \u^{C,1} = \mathbb{P}\u^{C,1} + \mathbf{z}^{C,1} \in L^\infty(0,T;L^2(\Q)) \cap L^2(0,T; H^1(\Q)).
    \end{equation}

    3) \underline{\textbf{Estimates of the derivatives of $\u^{C,1}$.}}

    In order to obtain the weighted estimates of $\u^{C,1}$ and $\nabla \u^{C,1}$ and to avoid the weight falling on the unknown pressure term, we shall first obtain the weighted estimates of the vorticity.
    Setting $\omega^{C,1} = \dX v^{C,1} - \dY u^{C,1}$, from \eqref{u-C,1} we know that the vorticity satisfies the following problem:
\begin{equation}\label{vort-C,1}
        \begin{cases}
            \begin{aligned}
                & \dt \omega^{C,1} + \u^{G,1} \cdot \nabla \omega^{C,1} - \Delta \omega^{C,1} \\
                & \hspace{.4in} =  \left((v^{C,1} + \underline{v^{B,1,1}}) \dY^2 \overline{u^{B,1,2}} - (u^{C,1} + \overline{u^{B,1,2}}) \dX^2 \underline{v^{B,1,1}}\right), \quad \text{in }\Q,
            \end{aligned}\\
        {\omega^{C,1}}_{|X=0} = \dY \overline{u^{B,1,2}},\quad {\omega^{C,1}}_{|Y=0} = - \dX \underline{v^{B,1,1}},\\
        \lim\limits_{X,Y\to \infty} \omega^{C,1} = 0, \\
        {\omega^{C,1}}_{|t=0} = \nabla\times \u^C_0.
        \end{cases}
    \end{equation}

  (1) Estimates of the corrector $G$ for homogenizing the boundary data.
  
 To homogenize the boundary conditions  given in \eqref{vort-C,1}, we construct $G$ such that
    \begin{equation}\label{heat-G}
        \Delta G = 0 \quad \text{in } \Q, \quad G_{|X=0} =  \dY \overline{u^{B,1,2}},\quad G_{|Y=0} =  - \dX \underline{v^{B,1,1}}.
    \end{equation}
Since $\dY \overline{u^{B,1,2}}_{|Y = 0} = - \dX \underline{v^{B,1,1}}_{|X=0} = 0$, by Lemma \ref{lemma-LD-refl}, one has for $l=1,2$
\begin{equation*}
    \begin{aligned}
        \|G(t)\|_{H^l(\Q)} \lesssim & \|\dY \overline{u^{B,1,2}}(t)\|_{\dot{H}^{-\frac{1}{2}}\cap H^{l}(\R_+)} + \|\dX \underline{v^{B,1,1}}(t)\|_{\dot{H}^{-\frac{1}{2}}\cap H^{l}(\R_+)}  \\
    \lesssim & \|\overline{u^{B,1,2}}(t)\|_{H^{l+1}(\R_+)} + \|\underline{v^{B,1,1}}(t)\|_{H^{l+1}(\R_+)},
    \end{aligned}
\end{equation*}
and
\begin{equation*}
    \begin{aligned}
        \|\dt G(t)\|_{L^2(\Q)} \lesssim & \|\dt \dY \overline{u^{B,1,2}}(t)\|_{\dot{H}^{-\frac{1}{2}}(\R_+)} + \|\dt \dX \underline{v^{B,1,1}}(t)\|_{\dot{H}^{-\frac{1}{2}}(\R_+)}  \\
    \lesssim & \|\dt \overline{u^{B,1,2}}(t)\|_{H^{1}(\R_+)} + \|\dt \underline{v^{B,1,1}}(t)\|_{H^{1}(\R_+)},
    \end{aligned}
\end{equation*}
Thus, we conclude that
\begin{equation}\label{5.12}
    \begin{aligned}
        G \in L^\infty(0,T;H^1(\Q))\cap L^2(0,T; H^2(\Q)),\quad \dt G \in L^2(0,T;L^2(\Q)).
    \end{aligned}
\end{equation}
On the other hand, for any $k \in \mathbb{N}\backslash \{0\}$, by using Lemma \ref{lemma-kozlov'} with $l = 2$ and $\beta = 2k$, we get
\begin{equation*}
    \begin{aligned}
        \|G(t)\|_{V^2_{2,2k}(\Q)} \lesssim & \|\dY \overline{u^{B,1,2}}(t)\|_{V^{\frac{3}{2}}_{2,2k}(\R_+)} + \|\dX \underline{v^{B,1,1}}(t)\|_{V^{\frac{3}{2}}_{2,2k}(\R_+)} \\
        \lesssim & \|\dY \overline{u^{B,1,2}}(t)\|_{V^{2}_{2,2k+\frac{1}{2}}(\R_+)} + \|\dX \underline{v^{B,1,1}}(t)\|_{V^{2}_{2,2k+\frac{1}{2}}(\R_+)}\\
        \lesssim & \sum_{0\le j\le 2} \left(\left\|\overline{u^{B,1,2}}(t)\right\|_{H^{2k-\frac{3}{2}+j,j+1}_Y(\R_+)} + \left\|\underline{v^{B,1,1}}(t)\right\|_{H^{2k-\frac{3}{2}+j,j+1}_X(\R_+)}\right)\\
        \lesssim & \left\|\overline{u^{B,1,2}}(t)\right\|_{H^{2k+1,3}_Y(\R_+)} + \left\|\underline{v^{B,1,1}}(t)\right\|_{H^{2k+1,3}_X(\R_+)} 
    \end{aligned}
\end{equation*}
which is bounded in $L^2(0,T)$. Similarly, by using Lemma \ref{lemma-kozlov'} with $l = 0$ and $\beta = 2k$, we get
\begin{equation*}
    \begin{aligned}
        \|\dt G(t)\|_{V^0_{2,2k}(\Q)} 
        \lesssim & \|\dt \overline{u^{B,1,2}}(t)\|_{H^{2k+1,1}(\R_+)} + \|\dt \underline{v^{B,1,1}}(t)\|_{H^{2k+1,1}(\R_+)}
    \end{aligned}
\end{equation*}
which is bounded in $L^2(0,T)$.
Therefore, by \eqref{V-H}, one has
\begin{equation}\label{G-weight}
    G \in L^2(0,T;H^{2k,2}_{X,Y}(\Q)), \quad \dt G \in L^2(0,T; H^{2k,0}_{X,Y}(\Q)),\quad \forall k \in \mathbb{N}.
\end{equation}

(2) Estimates of the vorticity $\omega^{C,1}$.

    Plugging $\omega^{C,1} = \omega^C + G$ into \eqref{vort-C,1}, we obtain that $\omega^{C}$ satisfies the following problem:
    \begin{equation}\label{vort-C}
        \begin{cases}
            \dt \omega^C + \u^{G,1} \cdot \nabla \omega^{C} - \Delta \omega^C + \dt G + \u^{G,1} \cdot \nabla G = F^C ,\quad \text{in }\Q,\\
            {\omega^{C}}_{|X=0} = 0,\quad {\omega^{C}}_{|Y=0} = 0,\\
        \lim_{X,Y\to \infty} \omega^{C} = 0, \\
        {\omega^{C}}_{|t=0} = \nabla\times \u^C_0 - G(0),
        \end{cases}
    \end{equation}
    where $F^C =  \left((v^{C,1} + \underline{v^{B,1,1}}) \dY^2 \overline{u^{B,1,2}} - (u^{C,1} + \overline{u^{B,1,2}}) \dX^2 \underline{v^{B,1,1}}\right)$, and $G(0)$ stands for the initial value of $G$.

Multiplying \eqref{vort-C} by $\langle (X,Y) \rangle^{4k} \omega^{C}$ and integrating over $(0,t)\times \Q$, we obtain
\begin{equation}\label{I0}
    \frac{1}{2}\|\omega^C(t)\|^2_{H^{2k,0}_{X,Y}(\Q)} + \int_0^t\|\nabla \omega^C(\tau)\|^2_{H^{2k,0}_{X,Y}(\Q)} = \frac{1}{2}\|\omega^C(0)\|^2_{H^{2k,0}_{X,Y}(\Q)}+
\sum_{i=1}^4I_i
\end{equation}
where 
$$
\begin{cases}
I_1=- 4k \int_0^t 
        \inp*{\langle (X,Y) \rangle^{4k-2} \left(X \dX \omega^C + Y \dY \omega^C\right)}{\omega^C(\tau)}\\[2mm]
I_2=2k \int_0^t \int_{\Q} \langle (X,Y) \rangle^{4k-2} \left(X u^{G,1} + Y v^{G,1}\right)  |\omega^{C}|^2(\tau)\\[2mm]
I_3=- \int_0^t \inp*{\dt G + \u^{G,1}\cdot \nabla G}{\langle (X,Y) \rangle^{4k} \omega^C(\tau)}\\[2mm]
I_4=\int_0^t \inp*{F^C}{\langle (X,Y) \rangle^{4k} \omega^C(\tau)}.
    
\end{cases}
$$

\underline{Estimate of $I_1$.} By $|X|+|Y| \lesssim \langle (X,Y) \rangle^2$, one has
\begin{equation}\label{I1}
    \begin{aligned}
        |I_1| \le & \eta \int_0^t \|\nabla \omega^C(\tau)\|^2_{H^{2k,0}_{X,Y}(\Q)} + \frac{8k^2}{\eta} \int_0^t \|\omega^{C,1}(\tau)\|^2_{H^{2k,0}_{X,Y}(\Q)}
    \end{aligned}
\end{equation}
for some small constant $\eta$.

\underline{Estimate of $I_2$.} Recalling that $\u^{G,1} = \u^{C,1} + \left(\overline{u^{B,1,2}} , \underline{v^{B,1,1}} \right)^T + \widehat{\dx u^{E,0}}\left( X , -Y \right)^T$, by Hölder's inequality and Gagliardo-Nirenberg's inequality, one has
\begin{equation}\label{I2}
    \begin{aligned}
        |I_2| 
        & 
        \lesssim 
        \sup_{t\in [0,T]} \left(\|\u^{C,1}\|_{L^2(\Q)} + \left\|\frac{X}{\langle (X,Y)\rangle^2 }\overline{u^{B,1,2}}(Y), \frac{Y}{\langle (X,Y)\rangle^2 }\underline{v^{B,1,1}}(X)\right\|_{L^2(\Q)} \right) \\
        & \cdot \int_0^t \|\langle (X,Y) \rangle^{2k} \omega^{C}(\tau)\|^2_{L^4}  + \sup_{t\in [0,T]} |\widehat{\dx u^{E,0}}| \int_0^t \|\langle (X,Y) \rangle^{2k} \omega^{C}(\tau)\|^2_{L^2} 
        \\
        & 
        \le C \left( \left(2k^2+\frac{1}{\eta} \right)\int_0^t \|\omega^{C}(\tau)\|^2_{H^{2k,0}_{X,Y}(\Q)}  + \eta \int_0^t\|\nabla \omega^{C}(\tau)\|^2_{H^{2k,0}_{X,Y}(\Q)} \right),
    \end{aligned}
\end{equation}
for a small constant $\eta>0$, where $C$ is a constant depending on $\|\u^{C,1}\|_{L^\infty(0,T;L^2(\Q))}$, $\|\overline{u^{B,1,2}}, \underline{v^{B,1,1}}\|_{L^\infty(0,T;L^2(\R_+))}$, and $\|\widehat{\dx u^{E,0}}\|_{L^\infty(0,T)}$.

\underline{Estimate of $I_3$.}
By integration by parts, one has
\begin{equation*}
    \begin{aligned}
      |I_3| \le & \int_0^t \|\dt G(\tau)\|^2_{H^{2k,0}_{X,Y}(\Q)} 
        + \int_0^t \|(\u^{G,1}-\u^{C,1})\cdot \nabla G(\tau)\|^2_{H^{2k,0}_{X,Y}(\Q)} \\
        & + \int_0^t \|\omega^C(\tau)\|^2_{H^{2k,0}_{X,Y}(\Q)} + \int_0^t \left|\inp*{\langle (X,Y) \rangle^{4k} \u^{C,1} \cdot \nabla \omega^C }{G(\tau)}\right|\\
        & + 4k\int_0^t \left|\inp*{\langle (X,Y) \rangle^{4k-2}\left(X u^{C,1} + Y v^{C,1}\right)\omega^C}{G(\tau)}\right| \\
         \le & \|\dt G\|^2_{L^2(0,T;H^{2k,0}_{X,Y}(\Q))} + (1+2k)\int_0^t \|\omega^C(\tau)\|^2_{H^{2k,0}_{X,Y}(\Q)} \\
        &
        + C\|\overline{u^{B,1,2}}, \underline{v^{B,1,1}}\|^2_{L^\infty(0,T;L^2(\R_+))} \|\nabla G\|^2_{L^2(0,T;H^{2k,1}_{X,Y}(\Q))} \\
        & + C\|\widehat{\dx u^{E,0}}\|^2_{L^\infty(0,T)} \|\nabla G\|^2_{L^2(0,T; H^{2k+1,0}_{X,Y}(\Q))} + \eta \int_0^t \|\nabla \omega^{C,1}(\tau)\|^2_{H^{2k,0}_{X,Y}(\Q)}
        \\
        & + (2k+\frac{1}{\eta}) \|\u^{C,1}\|^2_{L^\infty(0,T;L^2(\Q))} \|G\|^2_{L^2(0,T; H^{2k,2}_{X,Y}(\Q))} ,
    \end{aligned}
\end{equation*}
where we have used the embedding $H^2(\Q) \hookrightarrow L^\infty(\Q)$.
Thus, there exists $C$ depending on $k$, $\|\u^{C,1}\|_{L^\infty(0,T;L^2(\Q))}$, $\|\overline{u^{B,1,2}}, \underline{v^{B,1,1}}\|_{L^\infty(0,T;L^2(\R_+))}$, $\|\widehat{\dx u^{E,0}}\|_{L^\infty(0,T)}$, $\|\dt G\|_{L^2(0,T; H^{2k,0}_{X,Y}(\Q))}$, $\|G\|_{L^2(0,T; H^{2(k+1),2}_{X,Y}(\Q))}$ and a small $\eta$, such that
\begin{equation}\label{I3}
    |I_3| \le \eta \int_0^t \|\nabla \omega^{C,1}(\tau)\|^2_{H^{2k,0}_{X,Y}(\Q)} + C \int_0^t \|\omega^{C,1}(\tau)\|^2_{H^{2k,0}_{X,Y}(\Q)} + C.
\end{equation}

\underline{Estimate of $I_4$.} To estimate $I_4$ is more difficult, for example, since $v^{C,1} \dY^2 \overline{u^{B,1,2}}$ only absorbs $(1+Y^{2k})$ but not $X^{2k}$, we shall use integration by parts and the relations given in Lemma \ref{lemma-kozlov'} to reveal its structure. We first estimate:
\begin{equation}\label{5.19}
    \begin{aligned}
        & \quad \begin{aligned}
            \int_0^t \left|\inp*{v^{C,1} \dY^2 \overline{u^{B,1,2}}}{X^{4k} \omega^C(\tau)}\right|   = &  \int_0^t \left|\inp*{ \dY v^{C,1} \dY \overline{u^{B,1,2}}}{X^{4k} \omega^C(\tau)}\right| \\
        & + \int_0^t \left|\inp*{v^{C,1} \dY \overline{u^{B,1,2}}}{X^{4k} \dY \omega^C(\tau)} \right| 
        \end{aligned}\\
        \lesssim & \|\overline{u^{B,1,2}}\|_{L^\infty(0,T;H^2(\R_+))} \left(\int_0^t \left|\inp*{\dY v^{C,1} }{X^{4k} \omega^C(\tau)} \right| + \int_0^t \left|\inp*{v^{C,1} }{X^{4k} \dY \omega^C(\tau)} \right|\right),
    \end{aligned}
\end{equation}
where we have used the fact that $\|\dY \overline{u^{B,1,2}}\|_{L^\infty(0,T;L^\infty(\R_+))} \lesssim \|\overline{u^{B,1,2}}\|_{L^\infty(0,T;H^2(\R_+))}$.
For two terms on the right hand side of \eqref{5.19}, one has
\begin{equation}\label{5.20}
    \begin{aligned}
        \int_{\Q}\left|\dY v^{C,1} \cdot X^{4k} \omega^C(\tau) \right| \d X \d Y \le & \|X^{2k}\dY v^{C,1}(\tau)\|_{L^2(\Q)} \|X^{2k} \omega^{C}(\tau)\|_{L^2(\Q)} \\
        \lesssim & \|\dY v^{C,1}(\tau)\|_{V^0_{2,2k}(\Q)} \|\omega^C(\tau)\|_{V^0_{2,2k}(\Q)},
    \end{aligned}
\end{equation}
and
\begin{equation}\label{5.21}
    \begin{aligned}
    \int_{\Q}\left|v^{C,1} \cdot X^{4k} \dY \omega^C(\tau) \right| \d X \d Y \le & \|v^{C,1}(\tau)\|_{V^1_{2,2k}(\Q)} \|X^{4k} \dY \omega^{C}(\tau)\|_{V^{-1}_{2,-2k}(\Q)} \\
        \lesssim & \|v^{C,1}(\tau)\|_{V^1_{2,2k}(\Q)} \|\dY \omega^C(\tau)\|_{V^{-1}_{2,2k}(\Q)} \\
\lesssim & \|v^{C,1}(\tau)\|_{V^1_{2,2k}(\Q)} \| \omega^C(\tau)\|_{V^{0}_{2,2k}(\Q)}.
    \end{aligned}
\end{equation}
by using the inequality given in \eqref{V-1}.

Now we study the relations between $\|\dY v^{C,1}(\tau)\|_{V^0_{2,2k}(\Q)}$, $\|v^{C,1}(\tau)\|_{V^1_{2,2k}(\Q)}$, and $\| \omega^C(\tau)\|_{V^{0}_{2,2k}(\Q)}$.
In fact, one can recover $\u^{C,1}$ by the following elliptic system:
\begin{equation*}
    \nabla \times \u^{C,1} = \omega^C + G, \quad \nabla \cdot \u^{C,1} = 0, \quad {u^{C,1}}_{|X=0} = - \overline{u^{B,1,2}},\, {v^{C,1}}_{|Y=0} = - \underline{v^{B,1,1}}.
\end{equation*}
Then the stream function $\phi^{C,1}$ of $\u^{C,1} = (-\dY \phi^{C,1}, \dX \phi^{C,1})$ satisfies
\begin{equation*}
        \Delta \phi^{C,1} = \omega^C + G,\quad
        {\phi^{C,1}}_{|X=0} = \int_0^Y \overline{u^{B,1,2}}(t,Y') \d Y', \, {\phi^{C,1}}_{|Y=0} = - \int_0^X \underline{v^{B,1,1}}(t,X') \d X'.
\end{equation*}
According to Lemma \ref{lemma-kozlov'} with $l = 2, \beta = 2k$, we have
\begin{equation*}
    \begin{aligned}
        \|\phi^{C,1}(t)\|_{V^2_{2,2k}(\Q)}  \lesssim & \|\omega^C(t)\|_{V^0_{2,2k}(\Q)} + \|G(t)\|_{V^0_{2,2k}(\Q)} \\
        & + \left\|\int_0^Y \overline{u^{B,1,2}}(t,Y') \d Y'\right\|_{V^{\frac{3}{2}}_{2,2k}(\R_+)} + \left\|\int_0^X \underline{v^{B,1,1}}(t,X') \d X'\right\|_{V^{\frac{3}{2}}_{2,2k}(\R_+)}.
    \end{aligned}
\end{equation*}
Here $\|G(t)\|^2_{V^0_{2,2k}(\Q)} \lesssim \|G(0)\|^2_{V^0_{2,2k}(\Q)} + \frac{1}{2}\int_0^t \|\dt G(\tau)\|^2_{V^0_{2,2k}(\Q)}$ is bounded in $L^\infty(0,T)$, and 
\begin{equation*}
    \begin{aligned}
       \left\|\int_0^Y \overline{u^{B,1,2}}(t,Y') \d Y'\right\|_{V^{\frac{3}{2}}_{2,2k}(\R_+)} \le & \left\|\int_0^Y \overline{u^{B,1,2}}(t,Y') \d Y'\right\|_{V^{2}_{2,2k+\frac{1}{2}}(\R_+)} \\
        \le & \left\|\int_0^Y \overline{u^{B,1,2}}(t,Y') \d Y'\right\|_{H^{2k+1,2}_Y(\R_+)} \lesssim \|\overline{u^{B,1,2}}(t)\|_{H^{2k+2,1}_{Y}(\R_+)},
    \end{aligned}
\end{equation*}
and is also bounded in $L^\infty(0,T)$, as is $\left\|\int_0^X \underline{v^{B,1,1}}(t,X') \d X'\right\|_{V^{\frac{3}{2}}_{2,2k}(\R_+)}$. Then there exists a constant $C_0$ such that
\begin{equation*}
    \|\phi^{C,1}(t)\|^2_{V^2_{2,2k}(\Q)} \le  C_0 \|\omega^C(t)\|^2_{V^0_{2,2k}(\Q)} + C_0.
\end{equation*}
Hence,
\begin{equation}\label{V-control}
    \begin{aligned}
        & \|\dY v^{C,1}(\tau)\|_{V^0_{2,2k}(\Q)} \lesssim \|\nabla^2 \phi^{C,1}(\tau)\|_{V^{0}_{2,2k}(\Q)} \lesssim  \|\phi^{C,1}(\tau)\|_{V^{2}_{2,2k}(\Q)},\\
        &  \|v^{C,1}(\tau)\|_{V^1_{2,2k}(\Q)} \lesssim  \|\nabla \phi^{C,1}(\tau)\|_{V^{1}_{2,2k}(\Q)} \lesssim  \|\phi^{C,1}(\tau)\|_{V^{2}_{2,2k}(\Q)},
    \end{aligned}
\end{equation}
are both controlled by $C_0  \|\omega^C(\tau)\|^2_{V^0_{2,2k}(\Q)} + C_0$.
Therefore, by plugging \eqref{V-control} in \eqref{5.20} and \eqref{5.21}, from \eqref{5.19} we get 
\begin{equation}\label{I4}
    \int_0^t \left|\inp*{v^{C,1} \dY^2 \overline{u^{B,1,2}}}{X^{4k} \omega^C(\tau)}\right|  \lesssim 1 + \int_0^t \|\omega^C(\tau)\|^2_{H^{2k,0}_{X,Y}(\Q)}.
\end{equation}
Similar control holds for $\int_0^t \left|\inp*{u^{C,1} \dX^2 \underline{v^{B,1,1}}}{Y^{4k} \omega^C(\tau)}\right|$, a term contained in $I_4$.
Denote by $I_4'$ the remaining terms given in $I_4$. It's easy to see that
\begin{equation}\label{I4'}
    \begin{aligned}
        |I'_4| \lesssim & \int_0^t \inp*{v^{C,1} \langle Y\rangle^{2k} \dY \overline{u^{B,1,2}}}{\langle Y\rangle^{2k} \omega^C(\tau)} + \int_0^t \inp*{u^{C,1} \langle X\rangle^{2k} \dX \underline{v^{B,1,1}}}{\langle X\rangle^{2k} \omega^C(\tau)} \\
        & + \int_0^t\inp*{\langle (X,Y)\rangle^{2k}  \left(\underline{v^{B,1,1}} \dY \overline{u^{B,1,2}} - \overline{u^{B,1,2}} \dX^2 \underline{v^{B,1,1}}\right)}{\langle (X,Y)\rangle^{2k} \omega^C(\tau)} \\
         \lesssim 
         &
         \|\dY \overline{u^{B,1,2}}\|^2_{L^\infty(0,T; H^{2k,0}_Y(\R_+))}\int_0^t \left(\|v^{C,1}(\tau)\|^2_{H^1(\Q)} + \|\underline{v^{B,1,1}}(\tau)\|^2_{H^{2k,0}_X(\R_+)} \right) \\
        & + \|\dX \underline{v^{B,1,1}}\|^2_{L^\infty(0,T; H^{2k,0}_X(\R_+))}\int_0^t \left(\|u^{C,1}(\tau)\|^2_{H^1(\Q)} + \|\overline{u^{B,1,2}}(\tau)\|^2_{H^{2k,0}_Y(\R_+)} \right)\\
        & + \int_0^t \|\omega^C(\tau)\|^2_{H^{2k,0}_{X,Y}(\Q)}.
    \end{aligned}
\end{equation}

By plugging the estimates \eqref{I1}, \eqref{I2}, \eqref{I3}, \eqref{I4} and \eqref{I4'} into \eqref{I0}, and use Gronwall's inequality, it follows
\begin{equation*}
    \omega^C \in L^\infty(0,T; H^{2k,0}_{X,Y}(\Q)),
\end{equation*}
for any $k \in \mathbb{N}$. 
With the help of \eqref{V-control} and the embedding $V^{1}_{2,2k} \hookrightarrow V^{0}_{2,2k-1}$, we conclude that
\begin{equation}\label{u-C-decay}
    \u^{C,1} \in L^\infty(0,T; H^{2k-1,0}_{X,Y}(\Q)), \quad \nabla \u^{C,1} \in L^\infty(0,T; H^{2k,0}_{X,Y}(\Q)),
\end{equation}
for any $k \in \mathbb{N}$. 
Using the inequality $$ \langle (X,Y)\rangle^{2k-1} \lesssim \langle (X,Y)\rangle^{2k} \lesssim \langle (X,Y)\rangle^{2k+1},$$
the conclusion $\u^{C,1}\in L^\infty(0,T; H^{k,1}_{X,Y}(\Q))$  follows immediately.

In particular, $\u^{C,1} \in L^\infty(0,T;H^1_{X,Y}(\Q))$. A standard difference argument yields the uniqueness of the solution to the problem \eqref{u-C,1}. By an argument similar to that in the last part of the proof of Proposition \ref{prop-E,2}, one has $\nabla p^{C,2} \in L^\infty(0,T;L^2_{X,Y}(\Q))$.
\end{proof}

\section{Justification of the asymptotic expansion in the small viscosity limit}

The aim of this section is to prove the  asymptotic expansion \eqref{expansion} in the small viscosity limit. The key point is to establish the estimates of the remainder uniformly in the viscosity. 
Throughout this section, $C$ may depend on $|\alpha|$ and on the fixed profiles, but is independent of $\epsilon$.

\subsection{The equations of remainder $R^\epsilon$}

To study the asymptotic expansion \eqref{expansion} of the solution  $(u^{\epsilon},v^{\epsilon})$ to the problem \eqref{INSE}, we are going to seek the solution of \eqref{INSE} in the form
\begin{equation}\label{exp-u,v}
    \begin{aligned}
        \u^\epsilon(t,x,y)  = & \u^{E,0}(t,x,y)
        + \sqrt{\epsilon} \left(u^{B,1,2}(t,x,\frac{y}{\sqrt{\epsilon}}), v^{B,1,1}(t,\frac{x}{\sqrt{\epsilon}},y) \right)^T \\
        & + \sqrt{\epsilon} \u^{C,1}\left(t, \frac{x}{\sqrt{\epsilon}}, \frac{y}{\sqrt{\epsilon}}\right) + \epsilon \u^{E,2}(t,x,y)  + \epsilon \left(u^{B,2,1}(t,\frac{x}{\sqrt{\epsilon}}, y), v^{B,2,2}(t,x,\frac{y}{\sqrt{\epsilon}})\right)^T \\
        & + \epsilon R^\epsilon (t,x,y) \\
         =: & \u^a + \epsilon R^\epsilon,
        \end{aligned}
\end{equation}
where $R^\epsilon = (R_1^\epsilon, R_2^\epsilon)^T$, and write the pressure $p^\epsilon$ in the form:
\begin{equation}\label{exp-p}
    \begin{aligned}
        p^\epsilon (t,x,y)  = & p^{E,0}(t,x,y)
        + \epsilon \left(p^{E,2}(t,x,y) + p^{C,2}\left(t,\frac{x}{\sqrt{\epsilon}},\frac{y}{\sqrt{\epsilon}}\right)\right) \\
        & + \epsilon \pi^\epsilon(t,x,y).
        \end{aligned}
\end{equation}
For $k\in \mathbb{N}$, introduce the notations 
\begin{equation}
\begin{cases}
    \overline{u_{(k)}}(t,x,y) = u(t,x,y) - \sum_{l=0}^k \frac{x^l}{l!}\overline{\partial_x^l u}(t,y),\\
    \underline{u_{(k)}}(t,x,y) = u(t,x,y) - \sum_{l=0}^k \frac{y^l}{l!}\underline{\partial_y^l u}(t,x),\\
    \widehat{u_{(k)}}(t,x,y) = u(t,x,y) - \sum_{0\le i+j\le k} \frac{x^i}{i!}\frac{y^j}{j!}\widehat{\partial_x^i\partial_y^j u}(t).
    \end{cases}
\end{equation}
Plugging the expansions \eqref{exp-u,v} and \eqref{exp-p} into the momentum equation in \eqref{INSE}, using the equations given in \eqref{E,0}, \eqref{B-1-2}, \eqref{u-C,1}, \eqref{E,2}, we obtain that the remainder $R^\epsilon$ must satisfy the following equations:
    \begin{equation}\label{remainder}
    \begin{cases}
        \partial_t R^\epsilon + (\u^\epsilon\cdot \nabla) R^\epsilon + (R^\epsilon\cdot \nabla) \u^a + \nabla \pi^\epsilon - \epsilon\Delta R^\epsilon = - F^\epsilon , \quad \text{in }\Q\\
        \nabla \cdot R^\epsilon = 0, \quad \text{in }\Q\\
        R^\epsilon\cdot \vec{n} = 0, \text{on }\partial\Q\\
        \left(\partial_x R^\epsilon_2 - 2\alpha R^\epsilon_2\right)_{|x=0} = \overline{b^\epsilon_1},\quad \left(\partial_y R^\epsilon_1 - 2\alpha R^\epsilon_1\right)_{|y=0} = \underline{b^\epsilon_2},\\
        {R^\epsilon}_{|t=0} = 0.
    \end{cases}
    \end{equation}
where $F^\epsilon = (F_1^\epsilon, F_2^\epsilon)^T$, with $F^\epsilon_1 = \sum_{l=0}^2 \epsilon^{-\frac{l}{2}} F^\epsilon_{1,l}$, $F^\epsilon_{1,0} =  F^{\epsilon,\flat,1}_{1,0} + F^{\epsilon,\flat,2}_{1,0} + F^{\epsilon,\sharp}_{1,0}:$
\begin{equation}\label{F101}
     \begin{cases} \begin{aligned}
F^{\epsilon,\flat,1}_{1,0}  = & \left( u^{B,1,2} + u^{C,1} \right) \dx u^{B,1,2} + \sqrt{\epsilon} u^{E,2} \dx u^{B,1,2} \\
        & + v^{E,2} \dY u^{B,1,2} + \sqrt{\epsilon} \left(u^{B,1,2} \dx u^{E,2} + v^{B,1,1} \dy u^{E,2}\right) - \sqrt{\epsilon} \dx^2 u^{B,1,2}\\
        & + \partial_t u^{B,2,1} 
    + \left(u^{B,1,2} + u^{C,1} + \epsilon^{\frac{1}{2}} u^{E,2} \right)  \partial_X u^{B,2,1} \\
    & + \left(v^{E,0} + \epsilon^{\frac{1}{2}} (v^{B,1,1} + v^{C,1}) + \epsilon v^{E,2} \right) \partial_y u^{B,2,1}\\
    & + u^{B,2,1} \left(\partial_x u^{E,0} + \epsilon^{\frac{1}{2}} \dx u^{B,1,2} + \dX u^{C,1} + \epsilon \dx u^{E,2} \right)\\
    & + v^{B,2,2} \left(\partial_y u^{E,0} + \dY u^{B,1,2} + \dY u^{C,1} + \epsilon \dy u^{E,2} \right) \\
    & + \epsilon^{\frac{1}{2}} u^{B,2,1} \dX u^{B,2,1} + \epsilon v^{B,2,2} \dy u^{B,2,1},
    \end{aligned}\\
\begin{aligned}
        F^{\epsilon,\flat,2}_{1,0}  = & \epsilon^{\frac{1}{2}}\left(\underline{u^{E,0}_{(0)}} + u^{B,1,2} + u^{C,1} + \epsilon^{\frac{1}{2}}(u^{E,2} + u^{B,2,1} + u^{B,2,2})\right) \dx u^{B,2,2} \\
        & + \left(v^{B,1,1} + v^{C,1} + \epsilon^{\frac{1}{2}}(v^{E,2} + v^{B,2,1} + v^{B,2,2})\right) \dY u^{B,2,2} \\
        & + u^{B,2,2} \left(\dx \underline{u^{E,0}_{(0)}} + \epsilon^{\frac{1}{2}} \dx u^{B,1,2} + \dX u^{C,1} + \epsilon \dx u^{E,2} + \epsilon^{\frac{1}{2}} \dX u^{B,2,1}\right) \\
        & + v^{B,2,1} \left(\dy u^{E,0} + \dY u^{B,1,2} + \dY u^{C,1} + \epsilon (\dy u^{E,2} + \dy u^{B,2,1})\right)
    \end{aligned}\\
\begin{aligned}
F^{\epsilon,\sharp}_{1,0}  = &\epsilon \left(\u^{E,2} \cdot \nabla_{x,y} u^{E,2} - \Delta u^{E,2}\right) + \sqrt{\epsilon} \dx \left(u^{E,2} u^{C,1}(t,\frac{x}{\sqrt{\epsilon}},\frac{y}{\sqrt{\epsilon}})\right) \\
& + \sqrt{\epsilon} \dy \left(v^{E,2} u^{C,1}(t,\frac{x}{\sqrt{\epsilon}},\frac{y}{\sqrt{\epsilon}})\right) - \epsilon \Delta_{x,y} \left( u^{B,2,1}(t,\frac{x}{\sqrt{\epsilon}},y) \right) \\
& - \epsilon \dx^2 u^{B,2,2},
\end{aligned}
\end{cases}
\end{equation}

\begin{equation}\label{F11}
\begin{cases}
    \begin{aligned}
        F^\epsilon_{1,1}  = & \underline{u^{E,0}_{(0)}} \dx u^{B,1,2} + u^{B,1,2} \dx \underline{u^{E,0}_{(0)}} + \overline{u^{B,1,2}_{(0)}} \dX u^{C,1} + u^{C,1} \widehat{\dx u^{E,0}_{(0)}} \\
        & + (v^{B,1,1}+ v^{C,1}) \dy \overline{u^{E,0}_{(0)}} + \underline{v^{B,1,1}_{(0)}} \dY u^{C,1} 
        + v^{B,1,1} \dY \overline{u^{B,1,2}_{(0)}} \\
        & + \underline{v^{B,1,1}_{(0)}} \dY \overline{u^{B,1,2}} 
        + v^{C,1} \dY \overline{u^{B,1,2}_{(0)}} + \overline{u^{E,0}_{(0)}} \dX u^{B,2,1} + \underline{v^{E,0}_{(1)}} \dY u^{B,2,2},
    \end{aligned}
    \\[10mm]
    \begin{aligned}
        F^\epsilon_{1,2}  = \widehat{u^{E,0}_{(1)}} \dX u^{C,1} + \underline{v^{E,0}_{(1)}}\dY u^{B,1,2} + \widehat{v^{E,0}_{(1)}} \dY u^{C,1},
    \end{aligned}
    \end{cases}
\end{equation}
and
\begin{equation}\label{b}
\begin{aligned}
    b^\epsilon_1 = \frac{2\alpha}{\sqrt{\epsilon}}  (v^{B,1,1} + v^{C,1})  +  \left(2\alpha v^{B,2,2} - \dx v^{B,2,2}\right),
\end{aligned}
\end{equation}
and $F^\epsilon_2, b^\epsilon_2$ have the same form, with the roles of $(u,x,Y, \underline{(\cdot)}, (\cdot)^{B,k,2})$ and $(v,y, X, \overline{(\cdot)}, (\cdot)^{B,k,1})$ interchanged.

\subsection{Estimates of the source term $F^\epsilon$}
We start by using the regularity of profiles given in Theorem \ref{thm-Euler} (with $m=4$), Propositions \ref{prop-Prandtl}--\ref{prop-C,1} to deduce some uniform smallness of $F^\epsilon_1$ and $b^\epsilon_1$ in various norms. 

First, since $\u^{C,1} \in L^\infty(0,T; H^{k,1}_{X,Y}(\Q_{X,Y}))$ in the fast variables, we have
\begin{equation}\label{uC-small}
    \Big\langle \left(\frac{x}{\sqrt{\epsilon}}, \frac{y}{\sqrt{\epsilon}} \right) \Big\rangle^k \left( \u^{C,1}, \nabla_{X,Y} \u^{C,1} \right) \left(t, \frac{x}{\sqrt{\epsilon}}, \frac{y}{\sqrt{\epsilon}}\right) = O(\epsilon^{\frac{1}{2}}) \quad \text{in } L^\infty(0,T; L^2(\Q)),
\end{equation}
for any $k \in \mathbb{N}$. Similarly, for $l = 0,1$ and $k \in \mathbb{N}$, one has
\begin{equation}\label{uB-small}
    \begin{aligned}
        & \Big\langle \frac{y}{\sqrt{\epsilon}} \Big\rangle^k \dY^l u^{B,1,2} \left(t, x, \frac{y}{\sqrt{\epsilon}}\right) = O(\epsilon^\frac{1}{4}) \quad \text{in } L^2(0,T; H^2_xL^2_y(\Q)),\\
        & \Big\langle \frac{x}{\sqrt{\epsilon}} \Big\rangle^k \dX^l v^{B,1,1} \left(t, \frac{x}{\sqrt{\epsilon}}, y \right) = O(\epsilon^\frac{1}{4}) \quad \text{in } L^2(0,T; L^2_xH^2_y(\Q)),
    \end{aligned}
\end{equation}
and
\begin{equation}\label{uB2-small}
    \begin{aligned}
        & \Big\langle \frac{x}{\sqrt{\epsilon}}\Big\rangle^k \dX^l u^{B,2,1} = O(\epsilon^{\frac{1}{4}}) \quad \text{in } L^\infty(0,T; L^2_x  H^1_y(\Q)),\\
        & \Big\langle \frac{y}{\sqrt{\epsilon}}\Big\rangle^k \dY^l v^{B,2,2} = O(\epsilon^{\frac{1}{4}})\quad \text{in } L^\infty(0,T; H^1_x  L^2_y(\Q)).
    \end{aligned}
\end{equation}

Recalling that ${\u^{E,0}}_{|(0,\delta)^2} \in L^\infty(0,T; W^{2,\infty})$, and ${\u^{E,0}}_{|\Q\backslash (0,\delta)^2} \in L^\infty(0,T; H^4(\Q\backslash (0,\delta)^2))$ from Theorem \ref{thm-Euler} (with $m=4$)and the Sobolev embedding in two dimensions, we have
\begin{equation}\label{uE-1,0}
    \begin{aligned}
        & \left\|\frac{\widehat{\u^{E,0}_{(1)}}}{x^2+y^2}\right\|_{L^\infty_tL^\infty}, \left\|y^{-2}\underline{v^{E,0}_{(1)}}\right\|_{L^\infty_tL^\infty}, \left\|y^{-1}\underline{u^{E,0}_{(0)}}\right\|_{L^\infty_t W^{1,\infty}_x L^\infty_y}, \left\|\frac{\overline{u^{E,0}_{(0)}}}{x}\right\|_{L^\infty_t L^\infty},\\
        & \left\|\frac{\widehat{\dx u^{E,0}_{(0)}}}{\sqrt{x^2+y^2}}\right\|_{L^\infty_t L^\infty}, \left\|\frac{\dy \overline{u^{E,0}_{(0)}}}{x}\right\|_{L^\infty_t L^\infty}\quad \text{are all bounded.}
    \end{aligned}
\end{equation}
These bounds follow from the integral Taylor formula and the fundamental theorem of calculus. For example, one has
\begin{equation*}
    \begin{aligned}
        & \quad \left\|\frac{\widehat{\u^{E,0}_{(1)}}}{x^2+y^2}\right\|_{L^\infty_tL^\infty} = \left\|\frac{\u^{E,0} - \widehat{\u^{E,0}} - x \widehat{\dx \u^{E,0}} - y \widehat{\dy \u^{E,0}}}{x^2+y^2}\right\|_{L^\infty_tL^\infty} \\
        \lesssim & \left\|{\u^{E,0}}_{|(0,\delta)^2}\right\|_{L^\infty_t W^{2,\infty}} + \left\|{\u^{E,0}}_{|\Q\backslash (0,\delta)^2}\right\|_{L^\infty_t L^\infty}  + \left\|\widehat{\nabla_{x,y}\u^{E,0}}\right\|_{L^\infty_t},
    \end{aligned}
\end{equation*}
and similar estimates for other terms given in \eqref{uE-1,0}.

\underline{Estimate of $F_{1,2}^\epsilon$.}

Combining \eqref{uC-small}-\eqref{uB-small} and \eqref{uE-1,0}, we have
\begin{equation}\label{F-1,2}
    \begin{aligned}
        \|\epsilon^{-1} F^\epsilon_{1,2}\|_{L^2_tL^2}  \le & \left\|\frac{\widehat{u^{E,0}_{(1)}}}{x^2+y^2}\right\|_{L^\infty_tL^\infty} \left\|\Big\langle \left(\frac{x}{\sqrt{\epsilon}}, \frac{y}{\sqrt{\epsilon}}\right)\Big\rangle^2 \dX u^{C,1}\right\|_{L^2_t L^2} \\
        & + \left\|y^{-2}\underline{v^{E,0}_{(1)}}\right\|_{L^\infty_tL^\infty} \left\|\Big\langle \frac{y}{\sqrt{\epsilon}}\Big\rangle^2 \dY u^{B,1,2}\right\|_{L^2_t L^2} \\
        & + \left\|\frac{\widehat{v^{E,0}_{(1)}}}{x^2+y^2}\right\|_{L^\infty_tL^\infty} \left\|\Big\langle \left(\frac{x}{\sqrt{\epsilon}}, \frac{y}{\sqrt{\epsilon}}\right)\Big\rangle^2 \dY u^{C,1}\right\|_{L^2_t L^2} \\
        = & O(\epsilon^{\frac{1}{4}}).
    \end{aligned}
\end{equation}

\underline{Estimate of $F_{1,1}^\epsilon$.}

To estimate $F_{1,1}^\epsilon$, first noting that $u^{B,1,2} \in L^2(0,T;H^2_x H^1_Y(\Q_{x,Y}))$ and that $\overline{u^{B,1,2}} \in L^2(0,T; H^3_{Y}(\R_+))$, we have that
\begin{equation}\label{uB-over-0}
    \begin{aligned}
        & \quad \left\|\frac{\overline{u^{B,1,2}_{(0)}}}{x}\right\|_{L^2_tL^\infty_{x,y}} = \left\|\frac{u^{B,1,2} - \overline{u^{B,1,2}}}{x}\right\|_{L^2_tL^\infty_{x,Y}} \\
        \lesssim & \left\|\dx u^{B,1,2} \right\|_{L^2_tL^\infty_{x,Y}(|x|\le 1)} + \left\| u^{B,1,2}\right\|_{L^2_tL^\infty_{x,Y}(|x|\ge 1)} + \left\| \overline{u^{B,1,2}} \right\|_{L^2_t L^\infty_Y(\R_+)} \\
        \lesssim & \left\|u^{B,1,2}\right\|_{L^2_t W^{1,\infty}_x L^\infty_Y} + \left\| \overline{u^{B,1,2}} \right\|_{L^2_t L^\infty_Y(\R_+)} 
    \end{aligned}
\end{equation}
is bounded by using $H^2_xH^1_Y(\Q_{x,Y}) \hookrightarrow W^{1,\infty}_x L^\infty_Y (\Q_{x,Y})$ and $H^3_{Y}(\R_+) \hookrightarrow L^\infty_Y(\R_+)$.
Analogously,
\begin{equation}\label{vB-under-0}
    \left\|y^{-1}\underline{v^{B,1,1}_{(0)}}\right\|_{L^2_tL^\infty_{x,y}} \quad \text{is bounded},
\end{equation}
and
\begin{equation}\label{uB-dY-over-0}
    \begin{aligned}
        & \quad \left\|\frac{\dY\overline{u^{B,1,2}_{(0)}}}{x}\right\|_{L^2_tL^\infty_x L^2_y} = \epsilon^{\frac{1}{4}}\left\|\frac{\dY u^{B,1,2} - \dY \overline{u^{B,1,2}}}{x}\right\|_{L^2_tL^\infty_x L^2_{Y}} \\
        \lesssim & \epsilon^{\frac{1}{4}} \left(\left\|\dY u^{B,1,2}\right\|_{L^2_t W^{1,\infty}_x L^2_Y} + \left\|\dY \overline{u^{B,1,2}}\right\|_{L^2_t L^2_Y}\right)
        = O(\epsilon^{\frac{1}{4}}).
    \end{aligned}
\end{equation}

Combining \eqref{uC-small}-\eqref{uB-small} and \eqref{uE-1,0}, \eqref{uB-over-0}-\eqref{uB-dY-over-0}, and using $H^1(\R_+) \hookrightarrow L^\infty(\R_+)$, one has
\begin{equation}\label{F-1,1}
    \begin{aligned}
        \|\epsilon^{-\frac{1}{2}} F^\epsilon_{1,1}\|_{L^2_t L^2}  
        \le & 
        \left\| y^{-1}\underline{u^{E,0}_{(0)}}\right\|_{L^\infty_t L^\infty} \left\|\Big\langle \frac{y}{\sqrt{\epsilon}} \Big\rangle\dx u^{B,1,2}  \right\|_{L^2_t L^2}
        + \left\| \frac{\dx \underline{u^{E,0}_{(0)}}}{y}\right\|_{L^\infty_t L^\infty} 
        \left\|\Big\langle \frac{y}{\sqrt{\epsilon}} \Big\rangle u^{B,1,2}  \right\|_{L^2_t L^2} \\
        & + \left\|\frac{\overline{u^{B,1,2}_{(0)}}}{x}\right\|_{L^2_tL^\infty} \left\|\Big\langle \frac{x}{\sqrt{\epsilon}} \Big\rangle\dX u^{C,1}  \right\|_{L^\infty_t L^2}
        + \left\|y^{-1}\underline{v^{B,1,1}_{(0)}}\right\|_{L^2_tL^\infty} \left\|\Big\langle \frac{y}{\sqrt{\epsilon}} \Big\rangle\dY u^{C,1}  \right\|_{L^\infty_t L^2}
        \\
        &
        + \left\|\frac{\widehat{\dx u^{E,0}_{(0)}}}{\sqrt{x^2+y^2}}\right\|_{L^\infty_t L^\infty} 
        \left\|\Big\langle \left(\frac{x}{\sqrt{\epsilon}}, \frac{y}{\sqrt{\epsilon}} \right)\Big\rangle u^{C,1}  \right\|_{L^2_t L^2}\\
        & + \left\|\frac{\dy \overline{u^{E,0}_{(0)}}}{x}\right\|_{L^\infty_t L^\infty} \left( \left\|\Big\langle \frac{x}{\sqrt{\epsilon}}\Big\rangle v^{B,1,1}\right\|_{L^2_t L^2} + \left\|\Big\langle \frac{x}{\sqrt{\epsilon}}\Big\rangle v^{C,1}\right\|_{L^2_t L^2} \right) \\
        & + \left\|\frac{\dY\overline{u^{B,1,2}_{(0)}}}{x}\right\|_{L^2_tL^\infty_x L^2_y} 
        \left( \left\|\Big\langle \frac{x}{\sqrt{\epsilon}}\Big\rangle v^{B,1,1}\right\|_{L^\infty_t L^2_x H^1_y} + \left\|\Big\langle \frac{x}{\sqrt{\epsilon}}\Big\rangle v^{C,1}\right\|_{L^\infty_t L^2_x H^1_Y} \right) \\
        & + \left\|y^{-1}\underline{v^{B,1,1}_{(0)}}\right\|_{L^2_tL^\infty} \left\|\Big\langle \frac{y}{\sqrt{\epsilon}}\Big\rangle \dY u^{B,1,2}\right\|_{L^\infty_t L^2} + \left\|\frac{\overline{u^{E,0}_{(0)}}}{x}\right\|_{L^\infty_t L^\infty} \left\|\Big\langle \frac{x}{\sqrt{\epsilon}}\Big\rangle u^{B,2,1}\right\|_{L^2_t L^2} \\
        & + \epsilon^{\frac{1}{2}}\left\|y^{-2}\underline{v^{E,0}_{(1)}}\right\|_{L^\infty_tL^\infty} \left\|\Big\langle\frac{y}{\sqrt{\epsilon}}\Big\rangle^2\dY u^{B,2,2} \right\|_{L^2_tL^2}\\
         = & O(\epsilon^{\frac{1}{4}}).
    \end{aligned}
\end{equation}

\underline{Estimate of 
$F^{\epsilon,\flat}_{1,0}$.}

Recall that $v^{B,2,2} = \int_Y^\infty \dx u^{B,1,2}(t,x,Y')\,\d Y'$. Applying $\dx$ to \eqref{B-1-2} and integrating over $Y' \in (Y,\infty)$, one has
\begin{equation*}
    \begin{aligned}
        \dt v^{B,2,2} = & - \underline{u^{E,0}} \dx v^{B,2,2} + 3 \underline{\dy v^{E,0}} v^{B,2,2} + Y \underline{\dy v^{E,0}} \dx u^{B,1,2} \\
        & + 2 \dx \underline{\dy v^{E,0}} \int_Y^\infty u^{B,1,2} \d Y' - \dY \dx u^{B,1,2}.
    \end{aligned}
\end{equation*}
Hence, $\dt v^{B,2,2} \in L^2(0,T; L^2_{x,Y}(\Q_{x,Y}))$, and similarly, $\dt u^{B,2,1} \in L^2(0,T; L^2_{X,y}(\Q_{X,y}))$. Moreover, in the slow variables, one has
\begin{equation}\label{uB2-dt-small}
    \dt u^{B,2,1} = O(\epsilon^{\frac{1}{4}}) \quad \text{in } L^2(0,T; L^2_{x,y}(\Q)).
\end{equation}

By \eqref{uC-small}, \eqref{uB-small}, and again using $H^1(\R_+) \hookrightarrow L^\infty(\R_+)$, one has
\begin{equation}\label{F-1,0-flat}
    \begin{aligned}
        \|F^{\epsilon,\flat}_{1,0}\|_{L^2_t L^2}  \lesssim & \|\dx u^{B,1,2}\|_{L^2_t H^1_x H^1_Y} \left(\|u^{B,1,2}\|_{L^\infty_t L^2} + \|u^{C,1}\|_{L^\infty_t L^2} + \sqrt{\epsilon} \|u^{E,2}\|_{L^\infty_t L^2}\right)   \\
        & 
        + \|v^{E,2}\|_{L^\infty_t L^2_x H^1_y} \|\dY u^{B,1,2}\|_{L^2_t H^1_x L^2_y} + \sqrt{\epsilon} \|\dx^2 u^{B,1,2}\|_{L^2_tL^2}\\
        & + \sqrt{\epsilon} \|\nabla u^{E,2}\|_{L^\infty_tL^2}\left(\|u^{B,1,2}\|_{L^2_tH^1_xH^1_Y}  + \|v^{B,1,1}\|_{L^2_t H^1_XH^1_y} \right) \\
        & + \|\partial_t u^{B,2,1} \|_{L^2_tL^2}  + \left(\|u^{B,1,2}\|_{L^\infty_t H^1_x L^2_y} + \|u^{C,1}\|_{L^\infty_t H^1_X L^2_y}\right)  \|\partial_X u^{B,2,1}\|_{L^2_t L^2_x H^1_y} \\
        & + \sqrt{\epsilon} \|u^{E,2}\|_{L^\infty_t H^1_x L^2_y}  \|\partial_X u^{B,2,1}\|_{L^2_t L^2_x H^1_y} + \|v^{E,0}\|_{L^\infty_t L^2_x H^1_y} \|\dy u^{B,2,1}\|_{L^2_t H^1_X L^2_y} \\
    & + \sqrt{\epsilon}\left(\|v^{B,1,1}\|_{L^\infty_t L^2_x H^1_y} + \|v^{C,1}\|_{L^\infty_t L^2_x H^1_Y} + \sqrt{\epsilon} \|v^{E,2}\|_{L^\infty_t L^2_x H^1_y} \right) \|\dy u^{B,2,1}\|_{L^2_t H^1_X L^2_y} \\
    & + \left(\|u^{B,2,1}\|_{L^2_t L^2} + \|v^{B,2,2}\|_{L^2_t L^2}\right) \|\nabla_{x,y} u^{E,0}\|_{L^\infty_t H^2} \\
    & \begin{aligned}
        + \|u^{B,2,1}\|_{L^2_t H^1_X H^1_y} &\Big(\sqrt{\epsilon} \|\dx u^{B,1,2}\|_{L^\infty_t(L^2)} + \|\dX u^{C,1}\|_{L^\infty_t(L^2)}\\
    & + \sqrt{\epsilon} \|\dX u^{B,2,1}\|_{L^\infty_t L^2}  + \epsilon \|\dx u^{E,2}\|_{L^\infty_t(L^2)} \Big)
    \end{aligned}\\
    & \begin{aligned}
        + \|v^{B,2,2}\|_{L^2_t H^1_x H^1_Y} & \Big( \|\dY u^{B,1,2}\|_{L^\infty_t(L^2)} + \|\dY u^{C,1}\|_{L^\infty_t(L^2)} \\
    & + \epsilon  \|\dy u^{B,2,1}\|_{L^\infty_t L^2} + \epsilon \|\dy u^{E,2}\|_{L^\infty_t(L^2)} \Big)
    \end{aligned}\\
     = & O(\epsilon^{\frac{1}{4}}),
    \end{aligned}
\end{equation}
Similarly, $\|F^{\epsilon,\flat}_{2,0}\|_{L^2_t L^2} = O(\epsilon^{\frac{1}{4}})$.

\underline{Estimate of 
$F^{\epsilon,\sharp}_{1,0}$.}

Since we only have $u^{B,2,1} \in L^\infty(0,T; H^1_XH^1_y(\Q_{X,y}))$, $v^{B,2,2} \in L^\infty(0,T; H^1_xH^1_Y(\Q_{x,Y}))$ from Proposition \ref{prop-Prandtl}, and $\u^{E,2} \in L^\infty(0,T;H^1(\Q))$ from Proposition \ref{prop-E,2}, $F^{\epsilon,\sharp}_{1,0}$ only belongs to $L^\infty(0,T; H^{-1}(\Q))$. Since $\nabla\cdot \u^{E,2} = 0$, for any $\phi \in H^1_0(\Q)$, one has
\begin{equation}\label{F-phi}
    \begin{aligned}
        \inp*{F^{\epsilon,\sharp}_{1,0}}{\phi}  = & - \epsilon \inp*{\u^{E,2}\cdot \nabla\phi}{u^{E,2}} - \sqrt{\epsilon} \inp*{\u^{E,2}\cdot \nabla \phi}{u^{C,1}(t,\frac{x}{\sqrt{\epsilon}},\frac{y}{\sqrt{\epsilon}})} \\
        & + \epsilon \left(\inp*{\dx \u^{E,2}}{\dx \phi}_{L^2(\Q)} + \inp*{\dy \u^{E,2}}{\dy \phi}_{L^2(\Q)} \right) \\
        & + \sqrt{\epsilon} \inp*{\dX u^{B,2,1}}{\dx \phi} + \epsilon \inp*{\dy u^{B,2,1}}{\dy \phi} \\
        & + \epsilon \inp*{\dx u^{B,2,2}}{\dx \phi},
    \end{aligned}
\end{equation}
which implies
\begin{equation}\label{F-1,0-sharp}
    \begin{aligned}
        \left\| F^{\epsilon,\sharp}_{1,0} \right\|_{L^\infty(0,T; H^{-1}(\Q))}  \lesssim & 
        \epsilon \left(1+\|\u^{E,2}\|_{L^\infty_t(H^1(\Q))} \right)^2\\
        & + \sqrt{\epsilon} \|\dX u^{B,2,1}\|_{L^\infty_t (L^2(\Q))} + \epsilon \|\dy u^{B,2,1}\|_{L^\infty_t (L^2(\Q))}\\
        & + \sqrt{\epsilon} \|\u^{E,2}\|_{L^\infty_t (L^\infty_x L^2_y(\Q))} \|u^{C,1}(t,\frac{x}{\sqrt{\epsilon}}, Y)\|_{L^\infty_t (L^2_x L^\infty_Y)} + \epsilon \|\dx u^{B,2,2}\|_{L^\infty_t(L^2(\Q))}\\
         = & O(\epsilon) + O(\sqrt{\epsilon}) \cdot \left(\|\dX u^{B,2,1}\|_{L^\infty_t (L^2(\Q))} + \|u^{C,1}\|_{L^\infty_t (L^2_x H^1_Y)}\right) \\
         = & O(\epsilon^{\frac{3}{4}}).
    \end{aligned}
\end{equation}

\subsection{Proof of Theorem \ref{thm-1}}

For simplicity of notation, denote by $\text{Eqn}(R^\epsilon)$ the Navier-Stokes operator acting on $R^\epsilon$. Then one can write
\begin{equation}\label{energy}
    \begin{aligned}
        \epsilon^2 \inp*{\text{Eqn}(R^\epsilon)}{R^\epsilon}_{\Q} = & \inp*{\text{Eqn}(\u^\epsilon)}{\u^\epsilon}_{\Q} - \inp*{\text{Eqn}(\u^\epsilon)}{\u^a}_{\Q} \\
        & - \inp*{\text{Eqn}(\u^a)}{\u^\epsilon}_{\Q} + \inp*{\text{Eqn}(\u^a)}{\u^a}_{\Q}.
    \end{aligned}
\end{equation}
Recalling that $(I-\mathbb{P})R^\epsilon = (I-\mathbb{P}) \u^\epsilon = (I-\mathbb{P}) \u^a = 0$, and that
\begin{equation*}
    \begin{aligned}
        & \u^\epsilon \in L^2(0,T; H^1(\Q)), \text{Eqn}(\u^\epsilon) \in L^2(0,T; H^{-1}(\Q)),\\ 
        & \u^a \in L^\infty(0,T; H^1(\Q)), \text{Eqn}(\u^a) \in L^\infty(0,T; H^{-1}(\Q)),
    \end{aligned}
\end{equation*}
all terms on the right-hand side of \eqref{energy} make sense for a.e. $t\in (0,T)$ except for $\inp*{\text{Eqn}(\u^\epsilon)}{\u^\epsilon}_{\Q}$.
Since the equation of $R^\epsilon$ given in \eqref{remainder} is deduced by writing an equation for $u^a$ and subtracting it from the equation for $\u^\epsilon$, the energy inequality \eqref{NS-ineq} verified by $\u^\epsilon$ is equivalent to saying that the right-hand side of \eqref{energy} is $\le 0$. Thus, we start the energy estimate for $R^\epsilon$ by writing
\begin{equation}\label{A}
    \frac{1}{2}\|R^\epsilon(t)\|^2_{L^2(\Q)} - \epsilon \int_0^t \inp*{\Delta R^\epsilon}{R^\epsilon}_{\Q} \le -\int_0^t \inp*{( R^\epsilon \cdot \nabla) \u^a}{R^\epsilon}_{\Q} - \int_0^t \inp*{F^\epsilon}{R^\epsilon}_{\Q},
\end{equation}
in which the  Laplacian term is estimated as
\begin{equation}\label{lap}
    \begin{aligned}
        - \epsilon \int_0^t \inp*{\Delta R^\epsilon}{R^\epsilon}_{\Q} & = \epsilon \int_0^t \|\nabla R^\epsilon(\tau)\|^2_{L^2(\Q)} + \epsilon S
    \end{aligned}
\end{equation}
where
\begin{equation*}
    \begin{aligned}
        S = & \int_0^t \inp*{2\alpha R^\epsilon_2 + \overline{b^\epsilon_1}}{R^\epsilon_2(\tau)}_{|x=0}
        + \int_0^t \inp*{2\alpha R^\epsilon_1 + \underline{b^\epsilon_2}}{R^\epsilon_1(\tau)}_{|y=0} \\
        = & - 2\alpha \int_0^t \int_{\Q}\nabla \cdot \left(|R^\epsilon_2|^2, |R^\epsilon_1|^2\right)(\tau) 
        - \int_0^t \int_{\Q} \nabla \cdot \left(b^\epsilon_1 R^\epsilon_2, b^\epsilon_2 R^\epsilon_1 \right)(\tau),
    \end{aligned}
\end{equation*}
which implies 
\begin{equation*}
    \begin{aligned}
        \epsilon |S| \le & \frac{\epsilon}{10}\int_0^t \|R^\epsilon(\tau)\|^2_{H^1(\Q)} + C \left( \int_0^t \|R^\epsilon(\tau)\|^2_{L^2(\Q)} + \epsilon\|b^\epsilon(\tau)\|^2_{L^2(\Q)} + \epsilon^2\|\nabla \cdot b^\epsilon(\tau)\|^2_{L^2(\Q)}\right).
    \end{aligned}
\end{equation*}
Recalling from \eqref{b} that 
\begin{equation*}
    \begin{aligned}
        \epsilon\|b^\epsilon(\tau)\|^2_{L^2(\Q)}  \le & 4|\alpha|^2 \left(\|v^{B,1,1}(t,\frac{x}{\sqrt{\epsilon}},y)\|^2_{L^2_{x,y}(\Q)} + \|v^{C,1}(t,\frac{x}{\sqrt{\epsilon}},\frac{y}{\sqrt{\epsilon}})\|^2_{L^2_{x,y}(\Q)}\right) \\
        & + 4|\alpha|^2 \epsilon \|v^{B,2,2}(t,x,\frac{y}{\sqrt{\epsilon}})\|^2_{L^2_{x,y}(\Q)} + \epsilon \|\dx v^{B,2,2}(t,x,\frac{y}{\sqrt{\epsilon}})\|^2_{L^2_{x,y}(\Q)} \\
        = & O(\epsilon^{\frac{1}{2}}),
    \end{aligned}
\end{equation*}
and
\begin{equation*}
    \begin{aligned}
        \epsilon^2\|\dx b^\epsilon_1(\tau)\|^2_{L^2(\Q)}  \le & 4|\alpha|^2 \left(\|\dX v^{B,1,1}(t,\frac{x}{\sqrt{\epsilon}},y)\|^2_{L^2_{x,y}(\Q)} + \|\dX v^{C,1}(t,\frac{x}{\sqrt{\epsilon}},\frac{y}{\sqrt{\epsilon}})\|^2_{L^2_{x,y}(\Q)}\right) \\
        & + 4|\alpha|^2 \epsilon^2 \|\dx v^{B,2,2}(t,x,\frac{y}{\sqrt{\epsilon}})\|^2_{L^2_{x,y}(\Q)} + \epsilon^2 \|\dx^2 v^{B,2,2}(t,x,\frac{y}{\sqrt{\epsilon}})\|^2_{L^2_{x,y}(\Q)} \\
         = & O(\epsilon^{\frac{1}{2}}).
    \end{aligned}
\end{equation*}
Thus, one gets
\begin{equation}\label{S}
    \epsilon |S| \le \frac{\epsilon}{10}\int_0^t \|R^\epsilon(\tau)\|^2_{H^1(\Q)} + C \int_0^t \|R^\epsilon(\tau)\|^2_{L^2(\Q)} + C\epsilon^{\frac{1}{2}}.
\end{equation}
By formally writing 
$$\u^a = \u^{E,0} + \sqrt{\epsilon} \left(u^{B,1,2} , v^{B,1,1} \right)^T + \Tilde{\u}^a,$$
we estimate the linear term $\int_0^t \inp*{( R^\epsilon \cdot \nabla) \u^a}{R^\epsilon}_{\Q}$:
\begin{equation*}
    \begin{aligned}
        & \left|\int_0^t \inp*{( R^\epsilon \cdot \nabla) \u^a}{R^\epsilon}_{\Q}\right| 
        \lesssim \left(1+\|\nabla \u^{E,0}\|_{L^\infty_tL^\infty_{x,y}}\right) \int_0^t \|R^\epsilon(\tau)\|^2_{L^2(\Q)} \\
        & + \|\nabla \tilde{\u}^a\|_{L^\infty_t L^2_{x,y}} \int_0^t \|R^\epsilon(\tau)\|\|\nabla R^\epsilon(\tau)\| \d \tau \\
        & + \sqrt{\epsilon} \left(\|\dx u^{B,1,2}\|_{L^\infty_t H^1_x L^2_y} + \|\dy v^{B,1,1}\|_{L^\infty_t L^2_x H^1_y} \right)\int_0^t \|R^\epsilon(\tau)\|_{H^1_{x,y}} \|R^\epsilon(\tau)\|_{L^2_{x,y}} \d \tau \\
        & + \left(\|\dY u^{B,1,2}, \dX v^{B,1,1}\|_{L^\infty_t L^\infty_{x,y}}\right) \int_0^t \|R^\epsilon(\tau)\|^2_{L^2_{x,y}} \d \tau .
    \end{aligned}
\end{equation*}
Since $\|\nabla \u^{E,0}\|_{L^\infty_tL^\infty_{x,y}} \lesssim \|\u^{E,0}\|_{L^\infty_t H^3_{x,y}}$,
$$
\sqrt{\epsilon} \left(\|\dx u^{B,1,2}\|_{L^\infty_t H^1_x L^2_y} + \|\dy v^{B,1,1}\|_{L^\infty_t L^2_x H^1_y} \right) = O(\sqrt{\epsilon}),
$$
and
\begin{equation*}
    \begin{aligned}
        \|\nabla_{x,y} \Tilde{\u}^a\|^2_{L^\infty_t L^2_{x,y}} \le & \epsilon \cdot \epsilon^{-1} \|\nabla_{X,Y}\u^{C,1}\|^2_{L^\infty_t L^2_{x,y}}  + \epsilon^2 \|\nabla_{x,y}\u^{E,2}\|^2_{L^\infty_t L^2_{x,y}} \\
        & + \epsilon^2 \cdot\left( \epsilon^{-\frac{1}{2}} \|\left(\dX u^{B,2,1}, \dY v^{B,2,2}\right)\|^2_{L^\infty_t L^2_{x,y}} + \|\left(\dy u^{B,2,1}, \dx v^{B,2,2}\right)\|^2_{L^\infty_t L^2_{x,y}}\right)  \\
        = & O(\epsilon).
    \end{aligned}
\end{equation*}
Thus, by recalling \eqref{pro4.3-Linfty} and using Cauchy's inequality, one has
\begin{equation}\label{li-R}
    \left|\int_0^t \inp*{( R^\epsilon \cdot \nabla) \u^a}{R^\epsilon}_{\Q}\right| \le C\int_0^t \|R^\epsilon (\tau)\|^2_{L^2(\Q)} + \frac{\epsilon}{10}  \int_0^t \|\nabla R^\epsilon (\tau)\|^2_{L^2(\Q)}.
\end{equation}

By \eqref{F-1,2}-\eqref{F-1,1}, \eqref{F-1,0-flat}-\eqref{F-1,0-sharp}, one has
\begin{equation}\label{source}
    \begin{aligned}
        \left|\int_0^t \inp*{F^\epsilon}{R^\epsilon}_{\Q}\right| \le & \int_0^t \sum_{i=1}^2\|\left(F^{\epsilon,\flat}_{i,0} + \epsilon^{-\frac{1}{2}} F^\epsilon_{i,1} + \epsilon^{-1} F^\epsilon_{i,2}\right)(\tau)\|^2_{L^2(\Q)} \\
        & + \int_0^t \|R^\epsilon(\tau)\|^2_{L^2(\Q)} + \|F^{\epsilon,\sharp}_{1,0}\|_{L^\infty(0,T; H^{-1}(\Q))} \int_0^t\|\nabla_{x,y} R^\epsilon(\tau)\|_{L^2(\Q)} \\
        \le & C \epsilon^{\frac{1}{2}} + \int_0^t \|R^\epsilon(\tau)\|^2_{L^2(\Q)} + C \epsilon^{\frac{3}{4}} \int_0^t\|\nabla_{x,y} R^\epsilon(\tau)\|_{L^2(\Q)} \\
        \le & C \epsilon^{\frac{1}{2}} + \int_0^t \|R^\epsilon(\tau)\|^2_{L^2(\Q)} + \frac{\epsilon}{10} \int_0^t\|\nabla_{x,y} R^\epsilon(\tau)\|^2_{L^2(\Q)}.
    \end{aligned}
\end{equation}

Therefore, combining \eqref{A}-\eqref{source}, one obtains
\begin{equation*}
    \|R^\epsilon(t)\|^2_{L^2(\Q)} + \epsilon \int_0^t \|\nabla R^\epsilon(\tau)\|^2_{L^2(\Q)} \le C \epsilon^{\frac{1}{2}} + C \int_0^t \|R^\epsilon(\tau)\|^2_{L^2(\Q)},
\end{equation*}
which implies by using Gronwall's inequality that
\begin{equation*}
    \|R^\epsilon \|_{L^\infty(0,T; L^2(\Q))} = O(\epsilon^{\frac{1}{4}}),\quad \|\nabla R^\epsilon\|_{L^2(0,T; L^2(\Q))} = O(\epsilon^{-\frac{1}{4}}).
\end{equation*}
This completes the proof of Theorem \ref{thm-1}.

\section{Appendix: Problems of Laplace equation in a quarter plane}
\renewcommand{\thelemma}{A.\arabic{lemma}}
\setcounter{lemma}{0}
\renewcommand{\thedefn}{A.\arabic{defn}}
\setcounter{defn}{0}


In the appendix, we shall give some estimates of solutions to elliptic equations defined in a quarter plane, which were used in previous sections. For boundary value problems of Laplacian equations in domains with corners, there are two main conclusions. The first is due to Kozlov-Mazya-Rossman \cite{kozlov1997elliptic}, who established the well-posedness and regularity in weighted Sobolev spaces using the blow-up techniques. His results apply to infinite angles (in 2D) and cones (in higher dimensions), as well as to bounded domains with corners. The second is due to Grisvard \cite{grisvard2011elliptic}, who established the well-posedness of elliptic equations in bounded domains within the framework of classical Sobolev spaces. Both of these works provide an asymptotic expansion of the solution to the elliptic equation. In fact, from the expansion of solutions on angles/cones in Kozlov et alw.'s expansion, it can be seen that for an angular domain with a general angle, the well-posedness of boundary value problems of Laplacian problems is only established in the weighted spaces:

Consider the following boundary value problem in
$\Q=\{x>0, y>0\}$,
\begin{equation}\label{LN}
   \begin{cases}
 \Delta u = 0, \quad \text{in } \Q, \\
        (\dx u)_{|x=0} = g_1(y), \quad (\dy u)_{|y=0} = g_2(x).      
   \end{cases}     
\end{equation}
By the method of reflection, first we have the following result for \eqref{LN}:
\begin{lemma}\label{lemma-LN-refl}
    Suppose that $g_1, g_2 \in H^{l-\frac{1}{2}}(\R_+)$ for $l=0,1$, both of $G_1(x) := \int_{0}^x g_1(s) \d s$ and $ G_2(x) := \int_{0}^x g_2(s) \d s$ belong to $H^1(\R_+)$,
    and $\int_{\R_+} g_1 = \int_{\R_+} g_2 = 0 $. Then the  problem \eqref{LN} admits a solution $u \in H^{l+1}(\Q)$, and it is unique in the sense of modulus constants. Moreover, there exists a constant $C$, such that
    \begin{equation*}
        \|\nabla u\|_{H^{l}(\Q)} \le C\begin{cases}
            \sum_{j=1,2}\left(\|g_j\|_{H^{\frac{1}{2}}(\R_+)} + \left\|G_j \right\|_{L^2(\R_+)} \right) \quad \text{if } l =1,\\
            \sum_{j=1,2} \left\|G_j \right\|_{H^1(\R_+)}  \quad \text{if } l=0.
            
        \end{cases}
    \end{equation*}
\end{lemma}
\begin{proof}
    Let $\tilde{g}_1, \tilde{g}_2$ be the even extensions of $g_1, g_2$ on $\R$ respectively. Then $\int_{\R} \tilde{g}_1 = \int_{\R} \tilde{g}_2 = 0$. If $l=1$, for $j=1,2$, one has
    \begin{equation*}
        \begin{aligned}
            [\tilde{g}_j]_{H^{\frac{1}{2}}(\R)} = & \left(\int_{\R}\int_{\R} \frac{|\tilde{g}_j(x_1)-\tilde{g}_j(x_2)|^2}{|x_1-x_2|^2} \d x_1 \d x_2\right)^{\frac{1}{2}} \\
            \lesssim & [g_j]_{H^{\frac{1}{2}}(\R_+)} + \left(\int_{\R_+}\int_{\R_+} \frac{|\tilde{g}_j(x_1)-\tilde{g}_j(x_2)|^2}{|x_1+x_2|^2} \d x_1 \d x_2\right)^{\frac{1}{2}}  \\
            \lesssim & [g_j]_{H^{\frac{1}{2}}(\R_+)} + \left(\int_{\R_+}\int_{\R_+} \frac{|\tilde{g}_j(x_1)-\tilde{g}_j(x_2)|^2}{|x_1-x_2|^2} \d x_1 \d x_2\right)^{\frac{1}{2}}  \lesssim  [g_j]_{H^{\frac{1}{2}}(\R_+)},
        \end{aligned}
    \end{equation*}
   which implies
    \begin{equation*}
        \|\tilde{g}_j\|_{H^{\frac{1}{2}}(\R)} \lesssim \|g_j\|_{H^{\frac{1}{2}}(\R_+)} \qquad (j=1,2).
    \end{equation*}
  
Let $u_1, u_2$ be  solutions to the following two problems
    \begin{equation}\label{uu}
        \begin{cases}
        \Delta u_1 = 0, \quad x>0, y\in \R,\\
        (\dx u_1)_{|x=0} = \tilde{g}_1(y),
        \end{cases}\quad
        \begin{cases}
        \Delta u_2 = 0, \quad y>0, x\in \R,\\
        (\dy u_2)_{|y=0} = \tilde{g}_2(x),
        \end{cases}
    \end{equation}
    respectively. It is easy to know that $u_1(x,y)$ ($u_2(x,y)$ resp.) is even in $y$ ($x$ resp.). Thus, $u = (u_1)_{|y>0} + (u_2)_{|x>0}$ satisfies the problem \eqref{LN}.

    Then by using the Fourier transform $\mathcal{F}_{y\to \xi}$ in the above problem of $u_1$, we obtain that
    \begin{equation*}
       \begin{aligned}
            \|\nabla u_1\|_{H^l(\R^2_+)} \approx & \left(\int_{\R} (1+|\xi|^{2l}) \frac{|\mathcal{F}_{y \to \xi} [\tilde{g}_1](\xi)|^2}{|\xi|} \,\d \xi\right)^{\frac{1}{2}} \\
            \lesssim & \left(\int_{\R} (1+|\xi|^{-2}+|\xi|^{2l-1})|\mathcal{F}_{y \to \xi} [\tilde{g}_1](\xi)|^2 \,\d \xi\right)^{\frac{1}{2}}.
       \end{aligned}
    \end{equation*}
   On the other hand, one has
    \begin{equation*}
        \begin{aligned}
            \int_{\R} |\xi|^{-2}|\mathcal{F}_{y \to \xi} [\tilde{g}_1](\xi)|^2 \,\d \xi = \int_{\R} |\mathcal{F}_{y \to \xi} [\tilde{G}_1](\xi)|^2 \, \d \xi = \|\tilde{G}_1\|^2_{L^2(\R)},
        \end{aligned}
    \end{equation*}
    with $\tilde{G}_1 = \int_0^y \tilde{g}_1(s) \d s$ being the odd extension of $G_1$, and 
    \begin{equation*}
        \|\tilde{G}_1\|_{L^2(\R)} \lesssim \|G_1\|_{L^2(\R_+)}.
    \end{equation*}
Thus, we get the estimates.

Analogously, one can deduce a similar estimate for the solution $u_2$ of the problem \eqref{uu}.

The uniqueness of solution to the problem \eqref{LN} follows from the reflection extension and Liouville’s theorem.

\end{proof}

Similarly, for the  problem with the Dirichlet boundary conditions:
\begin{equation}\label{LD}
    \Delta u = 0, \quad \text{in }\Q, \quad u_{|x=0} = g_1(y), \quad u_{|y=0} = g_2(x),
\end{equation}
we have the following result:
\begin{lemma}\label{lemma-LD-refl}
    Suppose that $g_1, g_2 \in H^{l}\cap \dot{H}^{-\frac{1}{2}}(\R_+)$ for $l=0, 1,2$, and when $l\ge 1$, we also assume $g_1(0) = g_2(0) = 0$. Then the Laplacian-Dirichlet problem \eqref{LD} admits a unique solution $u \in H^{l}(\Q)$, and there exists a constant $C$, such that
    \begin{equation*}
        \|u\|_{H^{l}(\Q)} \le C\sum_{j=1,2}\left(\|g_j\|_{H^{l}(\R_+)} + \|g_j\|_{\dot{H}^{-\frac{1}{2}}(\R_+)}\right).
    \end{equation*}
    When $l=0$, the term $\|g_j\|_{H^{l}(\R_+)}$ on the above right-hand side could be omitted.
\end{lemma}
\begin{proof}
    Let $\tilde{g}_1, \tilde{g}_2$ be the odd extensions of $g_1,g_2$ on $\R$ respectively. Obviously, one has
    $$
    \|\tilde{g}_1\|_{H^l(\R)} \lesssim \|g_1\|_{H^l(\R_+)}, \quad \|\tilde{g}_2\|_{H^l(\R)} \lesssim \|g_2\|_{H^l(\R_+)}, \quad \forall l=0,1,2,
    $$
    by noting $g_1(0) = g_2(0) = 0$ as $l=1,2$.

    It is easy to know that the solution $u$ of \eqref{LD} can be decomposed as $u = (u_1)_{|y>0} + (u_2)_{|x>0}$, where $u_1$ and $u_2$ are the solutions to the following problems respectively,
    \begin{equation}\label{uu-}
        \begin{cases}
        \Delta u_1 = 0, \quad x>0, y\in \R,\\
        {u_1}_{|x=0} = \tilde{g}_1(y),
        \end{cases}\quad
        \begin{cases}
        \Delta u_2 = 0, \quad y>0, x\in \R,\\
        {u_2}_{|y=0} = \tilde{g}_2(x).
        \end{cases}
    \end{equation}
    
   By taking the Fourier transform $\mathcal{F}_{y\to \xi}$ in the first problem \eqref{uu-}, we obtain that
    \begin{equation*}
       \begin{aligned}
            \| u_1\|_{H^l(\R^2_+)} 
           = & \left(\int_{\R} (1+|\xi|^{2l}) \frac{|\mathcal{F}_{y \to \xi} [\tilde{g}_1](\xi)|^2}{|\xi|} \,\d \xi\right)^{\frac{1}{2}} \\
            \lesssim & \|\tilde{g}_1\|_{H^{l-\frac{1}{2}}(\R)}+ \left[\int_0^x \tilde{g}_1(s) \d s\right]_{H^{\frac{1}{2}}(\R)} \\
            \lesssim & \|g_1\|_{H^l(\R_+)} + \left[\int_0^x g_1(s) \d s\right]_{H^{\frac{1}{2}}(\R_+)},
       \end{aligned}
    \end{equation*}
    where we have used the fact that $\int_0^x \tilde{g}_1(s) \d s$ is an even extension of $\int_0^x g_1(s) \d s$, and thus $[\int_0^x \tilde{g}_1(s) \d s]_{H^{\frac{1}{2}}(\R)} \lesssim [\int_0^x g_1(s) \d s]_{H^{\frac{1}{2}}(\R_+)}$.
    A similar analysis of $u_2$ ends the proof.

\end{proof}

Finally, we introduce the definition of $V^l_{2,\beta}(\Q)$ for $l<0$ as the dual space of $V^{-l}_{2,-\beta}(\Q)$, defined in Definition \ref{defn4-1}, with the norm
\begin{equation*}
    \|u\|_{V^l_{2,\beta}(\Q)} = \sup \left\{|\inp*{u}{v}|: v\in V^{-l}_{2,-\beta}(\Q),\, \|v\|_{V^{-l}_{2,-\beta}} = 1 \right\},
\end{equation*}
where $\inp*{\cdot}{\cdot}$ is the extension of the scalar product in $L^2(\Q)$ to pairs $(u,v) \in V^l_{2,\beta}(\Q) \times V^{-l}_{2,-\beta}(\Q)$.
Then, by integration by parts, the duality between $V^s_{2,\beta}$ and $V^{-s}_{2,-\beta}$, and the trace theorem for weighted Sobolev spaces, one obtain that for any $u \in C^\infty_c(\overline{\Q}\backslash\{0\})$, one has
\begin{equation}\label{V-1}
    \|\nabla u\|_{V^{-1}_{2,\beta}(\Q)} \lesssim \|u_{|x=0}\|_{V^{-\frac{1}{2}}_{2,\beta}(\R_+)} + \|u_{|y = 0}\|_{V^{-\frac{1}{2}}_{2,\beta}(\R_+)} + \|u\|_{V^{0}_{2,\beta}(\Q)}. 
\end{equation}

Now we recall a result of the Poisson equation in the quarter plane from \cite{kozlov1997elliptic} as follows.

\begin{lemma}(\cite[Theorem 6.1.1-Theorem 6.1.3]{kozlov1997elliptic})\label{lemma-kozlov'}
    The Poisson-Dirichlet problem 
    \begin{equation*}
        \Delta u = f \, \text{in } \Q, \quad u_{|x=0} = g_1(y), \quad u_{|y=0} = g_2(x)
    \end{equation*}
    is uniquely solvable in $V^{l}_{2,\beta}(\Q)$ for given $f \in V^{l-2}_{2,\beta}(\Q)$ and $g_1, g_2 \in V^{l-\frac{1}{2}}_{2,\beta}(\R_+)$ with $\frac{l-1-\beta}{2} \not \in (\mathbb{Z}\backslash\{0\})$, and the solution $u$ satisfies the following estimate
    \begin{equation*}
        \|u\|_{V^l_{2,\beta}(\Q)} \le C\left(
        \|f\|_{V^{l-2}_{2,\beta}(\Q)}+ \|g_1\|_{V^{l-\frac{1}{2}}_{2,\beta}(\R_+)} + \|g_2\|_{V^{l-\frac{1}{2}}_{2,\beta}(\R_+)}\right)
    \end{equation*}
   for a constant $C$.
    In particular, $(l,\beta)$ can be taken as $(0,2k)$ or $(2,2k)$ with $k \in \mathbb{N}$.
\end{lemma}

\vspace{.1in}
\hspace{-.25in}{\bf Acknowledgements.} This research was supported by National Natural Science Foundation of China under grant 12331008, and National Key R\&D Program of China under grant 2024YFA1013302.

\bibliographystyle{IEEEtran}
\bibliography{reference}

@article {MR3054600,
    AUTHOR = {G\'erard-Varet, David and Lacave, Christophe},
     TITLE = {The two-dimensional {E}uler equations on singular domains},
   JOURNAL = {Arch. Ration. Mech. Anal.},
  FJOURNAL = {Archive for Rational Mechanics and Analysis},
    VOLUME = {209},
      YEAR = {2013},
    NUMBER = {1},
     PAGES = {131--170},
      ISSN = {0003-9527,1432-0673},
   MRCLASS = {35Q31 (35A01 76B03)},
  MRNUMBER = {3054600},
MRREVIEWER = {Emil\ Wiedemann},
       DOI = {10.1007/s00205-013-0617-9}
}

@article {MR3981393,
    AUTHOR = {Lacave, Christophe and Zlato\v{s}, Andrej},
     TITLE = {The {E}uler equations in planar domains with corners},
   JOURNAL = {Arch. Ration. Mech. Anal.},
  FJOURNAL = {Archive for Rational Mechanics and Analysis},
    VOLUME = {234},
      YEAR = {2019},
    NUMBER = {1},
     PAGES = {57--79},
      ISSN = {0003-9527,1432-0673},
   MRCLASS = {76B03 (35Q31)},
  MRNUMBER = {3981393},
MRREVIEWER = {Xinyu\ He},
       DOI = {10.1007/s00205-019-01384-7}
}

@article{gues1990probleme,
  title={Probleme mixte hyperbolique quasi-lin{\'e}aire caract{\'e}ristique},
  author={Gues, Olivier},
  journal={Comm. PDE},
  volume={15},
  number={5},
  pages={595--654},
  year={1990},
  publisher={Taylor \& Francis}
}

@book{adams2003sobolev,
  title={Sobolev {S}paces},
  author={Adams, Robert A and Fournier, John JF},
  volume={140},
  year={2003},
  publisher={Elsevier}
}

@article{guo2025well,
  title={Well-posedness of initial boundary value problems for 2{D} compressible {MHD} equations in domains with corners},
  author={Guo, Wen and Wang, Ya-Guang},
  journal={ arXiv:2511.05343},
  year={2025}
}

@article{bardos2013euler,
  title={The {E}uler equations in planar nonsmooth convex domains},
  author={Bardos, Claude and Di Plinio, Francesco and Temam, Roger},
  journal={J. Math. Anal. Appl.},
  volume={407},
  number={1},
  pages={69--89},
  year={2013},
  publisher={Elsevier}
}

@article{xin1999zero,
  title={Zero-viscosity limit of the linearized Navier-Stokes equations for a compressible viscous fluid in the half-plane},
  author={Xin, Zhouping and Yanagisawa, Taku},
  journal={Comm. Pure  Appl. Math.},
  volume={52},
  number={4},
  pages={479--541},
  year={1999},
  publisher={Wiley Online Library}
}

@article{iftimie2011viscous,
  title={Viscous boundary layers for the Navier--Stokes equations with the Navier slip conditions},
  author={Iftimie, Drago{\c{s}} and Sueur, Franck},
  journal={Arch. Rational Mech. Anal.},
  volume={199},
  number={1},
  pages={145--175},
  year={2011},
  publisher={Springer}
}

@book{kozlov1997elliptic,
  title={Elliptic {B}oundary {V}alue {P}roblems in {D}omains with {P}oint {S}ingularities},
  author={Kozlov, Vladimir A. and Maz'ya, Vladimir G. and Rossmann, J{\"u}rgen},
  volume={52},
  year={1997},
  publisher={American Mathematical Soc.}
}

@article{kelliher2006navier,
  title={Navier--{S}tokes equations with {N}avier boundary conditions for a bounded domain in the plane},
  author={Kelliher, James P},
  journal={SIAM J. Math. Anal.},
  volume={38},
  number={1},
  pages={210--232},
  year={2006},
  publisher={SIAM}
}

@article{masmoudi2012uniform,
  title={Uniform regularity for the {N}avier--{S}tokes equation with {N}avier boundary condition},
  author={Masmoudi, Nader and Rousset, Fr{\'e}d{\'e}ric},
  journal={Arch. Rational Mech. Anal.},
  volume={203},
  number={2},
  pages={529--575},
  year={2012},
  publisher={Springer}
}

@article{prandtl,
  title={\"{U}ber fl\"{u}ssigkeitsbewegung bei sehr kleiner reibung},
  author={L. Prandtl},
  journal={Verhandl. III, Internat.
Math. Kongr.},
  pages={484-491},
  year={1904}
}

@article{clopeau1998vanishing,
  author  = {Clopeau, Thierry and Mikeli{\'c}, Andro and Robert, Raoul},
  title   = {On the vanishing viscosity limit for the {2D} incompressible {Navier--Stokes} equations with the friction type boundary conditions},
  journal = {Nonlinearity},
  volume  = {11},
  number  = {6},
  pages   = {1625--1636},
  year    = {1998},
  doi     = {10.1088/0951-7715/11/6/011}
}

@article{lopesfilho2005inviscid,
  author  = {Lopes Filho, M. C. and Nussenzveig Lopes, H. J. and Planas, G.},
  title   = {On the inviscid limit for two-dimensional incompressible flow with {Navier} friction condition},
  journal = {SIAM J. Math. Anal.},
  volume  = {36},
  number  = {4},
  pages   = {1130--1141},
  year    = {2005},
  doi     = {10.1137/S0036141003432341}
}

@article{iftimie2006inviscid,
  author  = {Iftimie, Drago{\c{s}} and Planas, Gabriela},
  title   = {Inviscid limits for the {Navier--Stokes} equations with {Navier} friction boundary conditions},
  journal = {Nonlinearity},
  volume  = {19},
  number  = {4},
  pages   = {899--918},
  year    = {2006},
  doi     = {10.1088/0951-7715/19/4/007}
}

@article{gie2012boundary,
  author  = {Gie, Gung-Min and Kelliher, James P.},
  title   = {Boundary layer analysis of the {Navier--Stokes} equations with generalized {Navier} boundary conditions},
  journal = {J.  Diff. Equ.},
  volume  = {253},
  number  = {6},
  pages   = {1862--1892},
  year    = {2012},
  doi     = {10.1016/j.jde.2012.06.008}
}

@book{grisvard2011elliptic,
  author    = {Grisvard, Pierre},
  title     = {Elliptic Problems in Nonsmooth Domains},
  series    = {Classics in Applied Mathematics},
  volume    = {69},
  publisher = {Society for Industrial and Applied Mathematics},
  address   = {Philadelphia, PA},
  year      = {2011},
  note      = {Reprint of the 1985 original},
  doi       = {10.1137/1.9781611972030}
}

@book{dauge1988elliptic,
  author    = {Dauge, Monique},
  title     = {Elliptic Boundary Value Problems on Corner Domains:
               Smoothness and Asymptotics of Solutions},
  series    = {Lecture Notes in Mathematics},
  volume    = {1341},
  publisher = {Springer-Verlag},
  address   = {Berlin},
  year      = {1988},
  doi       = {10.1007/BFb0086682}
}

@article{zajaczkowski1984solvability,
  author  = {Zajaczkowski, W. M. and Piskorek, A.},
  title   = {Solvability of an Initial Boundary Value Problem for the
             Euler Equations in Two-Dimensional Domains with Corners},
  journal = {Math. Meth. Appl. Sci.},
  volume  = {6},
  number  = {1},
  pages   = {1--22},
  year    = {1984},
  doi     = {10.1002/mma.1670060102}
}

@article{MengWang,
  author  = {Meng, Y. P. and Wang, Y. G.},
  title   = {A uniform estimate for the incompressible magneto-hydrodynamic equations with a slip boundary condition},
  journal = {Quartly in Appl. Math.},
  volume  = {74},
  number  = {1},
  pages   = {27-48},
  year    = {2016}
}

@article{WangWilliams,
  author  = {Wang, Y. G. and Williams, M.},
  title   = {The inviscid limit and stability of characteristic boundary layers for the compressible Navier-Stokes equations with Navier-friction boundary conditions},
  journal = {Ann. Inst. Fourier, Grenoble},
  volume  = {62},
  number  = {6},
  pages   = {2257-2314},
  year    = {2012}
}

@article{WangWangXin,
  author = {Wang, X. P. and Wang, Y. G. and Xin, Z.},
  title   = {Boundary layers in incompressible Navier-Stokes equations with Navier boundary conditions for the vanishing viscosity limit},
  journal = {Comm. Math. Sci.},
  volume  = {8},
  number  = {4},
  pages   = {965-998},
  year    = {2010}
}

@article{rubin1966incompressible,
  author  = {Rubin, Stanley G.},
  title   = {Incompressible flow along a corner},
  journal = {Journal of Fluid Mechanics},
  volume  = {26},
  number  = {1},
  pages   = {97--110},
  year    = {1966}
}
\end{document}